\documentclass[reqno,11pt]{amsart}
\usepackage{geometry}
\usepackage{mathtools,amssymb,amsthm,mathrsfs,color,lineno,paralist,graphicx,float}
\usepackage[colorlinks,
linkcolor=red,
anchorcolor=green,
citecolor=blue,
]{hyperref}

\usepackage[T1]{fontenc}
\usepackage[utf8]{inputenc}

\usepackage{calc}
\definecolor{bleu1}{RGB}{0,57,128}
\def\bleu1{\color{bleu1}}

\usepackage{etoolbox}
\patchcmd{\section}{\normalfont}{\normalfont \bleu1}{}{}
\patchcmd{\subsection}{\normalfont}{\normalfont \bleu1}{}{}
\patchcmd{\subsubsection}{\normalfont}{\normalfont \bleu1}{}{}
\renewcommand{\proofname}{\it \bleu1 Proof}

\newcommand{\SL}{\mathrm{SL}(2,\mathbb{R})}
\newcommand{\GL}{\mathrm{GL}(2,\mathbb R)}
\newcommand{\gl}{\mathrm{gl}(2,\mathbb R)}

\newtheorem{Theorem}{Theorem}[section]
\newtheorem*{Theorem*}{Theorem 1.1}
\newtheorem{Lemma}{Lemma}[section]
\newtheorem{Proposition}{Proposition}[section]
\newtheorem{Corollary}{Corollary}[section]
\newtheorem{Remark}{Remark}[section]

\newtheorem{Definition}{Definition}[section]

\numberwithin{equation}{section}

\theoremstyle{definition}

\newcommand{\id}{\operatorname{id}}

\newcommand{\const}{\mathrm{const}}

\newcommand{\me}{\mathrm{e}}
\newcommand{\mi}{\mathrm{i}}

\newcommand{\N}{{\mathbb N}}
\newcommand{\Q}{{\mathbb Q}}
\newcommand{\R}{{\mathbb R}}
\newcommand{\T}{{\mathbb T}}

\newcommand{\Z}{{\mathbb Z}}

\def\B0{{\bold{0}}}

\catcode`\@=12

\def\Empty{}
\newcommand\oplabel[1]{
  \def\OpArg{#1} \ifx \OpArg\Empty {} \else
    \label{#1}
  \fi}

\newcommand{\comm}[1]{}
\newcommand{\comment}[1]{}

\begin{document}

\title[Almost-periodic ground state of the non-self-adjoint Jacobi operator ]{Almost-periodic ground state of the non-self-adjoint Jacobi operator and its applications}

\author{Xing Liang} \address{
School of Mathematical Sciences, University of Science and Technology
of China, Hefei, Anhui 230026, China}

\email{xliang@ustc.edu.cn}

\author{Hongze Wang} \address{Chern Institute of Mathematics and LPMC, Nankai University, Tianjin
300071, China}

\email{hongzew@mail.nankai.edu.cn}

\author{Qi Zhou} \address{Chern Institute of Mathematics and LPMC, Nankai University, Tianjin
300071, China}

\email{qizhou@nankai.edu.cn}

\setcounter{tocdepth}{1}

\begin{abstract}
We study the ground states of the one-dimensional non-self-adjoint Jacobi operators in the almost periodic media by using the method of dynamical systems. We show the existence of the ground state. Particularly, in the quasi-periodic media, we show that the lower regularity of coefficients can guarantee the existence of ground states.  Besides that, we give two applications:  the first application is to show the existence and uniqueness of the positive steady state of the discrete Fisher-KPP type equation; the second application is to investigate the asymptotic behavior of the discrete stationary parabolic equation with large lower order terms.
\end{abstract}


\maketitle

\section{Introduction and Main result}

The study of the principle eigenvalue problem of the second-order differential operator
  \begin{equation}\label{eq:ellipticop}
 Lu=(a(x)u')'+b(x)u'+c(x)u, x\in I\subset \mathbb R, a>0
 \end{equation}
 and the respective Jacobi operator
 \begin{equation}\label{nonsaoperator}
(\mathcal Lu)(n)=a(n)u(n+1)+b(n)u(n-1)+c(n)u(n), n\in I\subset \mathbb Z, a, b>0
\end{equation}
 has a long history (e.g. see \cite{CD, Minc,RSS} and references therein).

 For the operator \eqref{eq:ellipticop},
 the classic theory originated from the Sturm-Liouville theorem when $I$ is a bounded interval, and the separated boundary conditions (Dirichlet, Neumann, or Robin type) are considered. From a more general framework of the second-order elliptic operators, the spectral theory of  \eqref{eq:ellipticop} was well studied. Particularly, it is known that in this case,  \eqref{eq:ellipticop} has a real eigenvalue $\lambda_1$ with a positive eigenfunction. Moreover, $\lambda_1$ is larger than the real part of any other eigenvalue of  \eqref{eq:ellipticop}. Such an eigenvalue $\lambda_1$ is called the principal eigenvalue of  \eqref{eq:ellipticop} since it dominates all other eigenvalues in modulus. The eigenfunction which is corresponding to $\lambda_1$ is called the principal eigenfunction.
  For the operator \eqref{nonsaoperator}, when $I=\{1,2,\cdots, N\}$ is a finite (bounded) subset of $\mathbb Z$, the spectrum can be studied in the framework of the nonnegative matrix.  It is known after the boundary condition at $1, N$ is given, \eqref{nonsaoperator} also has the principal eigenvalue $\lambda_1$ with a positive positive eigenfunction. Moreover, $\lambda_1$ is larger than the real part of any other eigenvalue of \eqref{nonsaoperator}.  Indeed, the existence and other properties of the principal eigenvalues of both \eqref{eq:ellipticop} and  \eqref{nonsaoperator}  can be obtained by transferring the problem to that of the positive operators and applying the Krein-Rutman theorem (Perron-Frobenius theorem for finite-dimensional case, e.g. see \cite{Deimling1985, Evans1964, Minc, RSS} and references therein).

The above results can easily be generalized if $I=\mathbb R$ or $\mathbb Z$, with periodic boundary conditions, i.e., $a,b,c$ are periodic with the same period since it is strongly positive.
However, the problem becomes complicated when $I=\mathbb R$ or $\mathbb Z$ in the non-periodic media because of loss of compactness. In this paper, we mainly focus on the eigenvalue problem of the Jacobi operator in $\ell^2(\Z)$ if $a,b,c$ are almost periodic.  Recall that a function $f:\mathbb{Z}\rightarrow\mathbb{R}$  is {\rm almost periodic} if ${\{f(\cdot+k)|k\in\mathbb Z\}}$ has a compact closure in $l^\infty(\mathbb Z)$. Denote by $\mathcal H(f)$ the hull of the almost periodic sequence $f$, i.e., $\mathcal H(f)=\overline{\{f(n+k)|k\in\mathbb Z\}}$, the closure in $\ell^\infty(\mathbb Z)$. $\mathcal H(f)$ is isomorphic to $\T^{d}$ with $d\in \N_+$ or $d=\infty$, and we say $f$ is quasi-periodic if $d\in\N$ \cite{Simon2011}.
 One should note in the almost periodic media, principle eigenfunction doesn't necessarily exist in $\ell^2(\Z)$, which is due to the absence of point spectrum \cite{Dinaburg1975,Eliasson1992}. Instead one can consider the natural concept almost-periodic ground state - the eigenfunction (not necessary belongs to $L^2(\R)$ (resp.$\ell^{2}(\Z)$)  corresponding to the minimal value (resp. maximum value) of the energy, since it admits similar properties as the ones of the principal eigenfunction. To make it more precise, let's introduce the following concept:

 \begin{Definition}
 We say that $u$ is an almost-periodic (quasiperiodic) ground state of the Jacobi operator \eqref{nonsaoperator},
if $u$ is a postive almost-periodic  (quasi-periodic) solution of $\mathcal Lu=E_0 u$, where $E_0\in\R$ is at the rightmost point of the spectrum $\Sigma(\mathcal{L})$.
 \end{Definition}

 Note if $\mathcal{L}$ is self-adjoint, it is well known that $\Sigma(\mathcal{L}) \subset \R$ and $E_0=\max \Sigma(\mathcal{L})$. However,
 we emphasize that here $\mathcal{L}$ may not be self-adjoint, and then one could select $E_0\geq \Re E$
 for any $E\in\Sigma(\mathcal L)$ since $\Sigma(\mathcal{L})$ might not lie in the real line. In the following, we just call this $E_0$ the ground state energy. Furthermore, we say $E_0$ is simple
 if $u,v$ are two almost periodic eigenfunctions, then $u=\lambda v$ for some $\lambda\in\R.$

In the continuous case, Kozlov \cite{Kozlov1984} first established the existence of quasi-periodic ground state of the elliptic operator
 \begin{equation}\label{koz}Lu=(A(\omega x)u')'+B(\omega x)u'+W(\omega x) u, x\in\mathbb R, A>0\end{equation}
  where $A, B, W\in C^k(\T^d,\R), d\in\N_+$, $\omega\in\R^d\backslash\Q^d$. Based on the celebrated KAM theory,    Kozlov \cite{Kozlov1984} showed that if $\omega$ satisfies the well-known Diophantine condition,
 \begin{equation}\label{Diohigh}
\inf\limits_{j\in\Z}|\langle n,\omega \rangle-j|\geq \frac{\gamma}{|n|^{\tau}} \text{ for all }n\in\Z^d\backslash\{0\},
\end{equation}
and if $k\geq C(\tau)$, and $B,V$ are both sufficiently small, then \eqref{koz} has a quasi-periodic ground state. In the discrete case,  based on reducibility theory of quasi-periodic $SL(2,\R)$ cocycles,  Liang-Wang-Zhou-Zhou \cite{Liang2021} consider
a class of almost-periodic \footnote{The hull
$\mathcal H(f)$  of the almost periodic sequence $f$ is isomorphic to $\T^{\infty}=\prod\limits _{i\in \mathbb N}\mathbb T^1$.} Schr\"odinger operator (i.e. $a=b=1$):
 \begin{equation}\label{schro}
(\mathcal Lu)(n)=u(n+1)+u(n-1)+V(n\omega)u(n), n\in  \mathbb Z,
\end{equation}
and show that if the potential $V$ is small enough, then \eqref{schro}  has an almost periodic ground state.
 One should notice that both results are perturbative, i.e.,
 the smallness of $B, V$ depends on the frequency $\omega$. The main difference is that while Kozlov \cite{Kozlov1984}  dealt with the self-adjoint operator and non-self-adjoint  operator
simultaneously, the method of \cite{Liang2021} can be only generalized to the  Schr\"odinger operator.

In this paper, we are interested in establishing the almost-periodic ground state of the non-self-adjoint  Jacobi operator \eqref{nonsaoperator}. Let us first comment on the history and motivations. While there was already an almost-periodic flu around the 1980s \cite{Sim2,Simon2000}. Stemming from methods of dynamical systems and harmonic analysis, many breakthroughs have appeared in recent years \cite{Avila2015, AJ05, AK, Avila2017, Avila2018, JLiu,jk, Martin2017}.  However, because of lacking spectral theorem and spectral measure, the basic spectral theory of  non-self-adjoint operator was even not well-developed. There are very few rigorous results for the non-self-adjoint almost-periodic Jacobi operator. The best result still goes back to Sarnak \cite{Sarnak}, a result of 40 years ago. On the other hand, in recent years, the non-self-adjoint Jacobi operator (Non-Hermitian Hamiltonian in physical literature) has been a hot topic in physics.
While Hermiticity lies at the heart of quantum mechanics, recent experimental advances in controlling dissipation have brought about unprecedented flexibility in engineering non-Hermitian Hamiltonians in open classical and quantum systems \cite{Gong,Klitzing1980}: non-Hermitian Hamiltonians exhibits rich phenomena without Hermitian counterparts, e.g. parity-time ($\mathcal{PT}$) symmetry breaking, topological phase transition, non-Hermitian skin effects, e.t.c \cite{AGU,BBK}, while all of these phenomena can be exhibited in  non-Hermitian quasicrystals (non-self-adjoint almost-periodic Jacobi operator) \cite{longhi,jiang2019interplay}.  All of these motivate us to try to understand the  non-self-adjoint almost-periodic Jacobi operator.

In the field of partial differential equations, when the eigenvalue problem and the principal eigenvalue problem of linear elliptic operators or general positive operators are investigated, the symmetry (self-adjointness) of the operator is also usually considered to bring out better structures and allow weaker conditions.  In fact, in the smooth bounded domain, the classical works on the existence of the principal eigenvalues of elliptic operators require the same smoothness of the coefficients which are regardless of the self-adjointness, see  \cite{Evans1964}.  After establishing the existence of the ground state, the self-adjointness of the operator and the ``Rayleigh" quotient formula of the principal eigenvalue play an important role in the estimation, optimization, and other problems (see \cite{Hamel} and references therein). These problems become very difficult,  and there are only few works when the non-self-adjointness appears \cite{Hamel, liang2014, Nadin2010}.

As mentioned by  Kozlov \cite{Kozlov1984},  usually one can only anticipate better results in the self-adjoint setting. However, in this paper, we will show this is not always true, at least in the problem of the ground state.
Let's consider  the following  physical model:
\begin{equation}\label{non-hermitian} (\mathcal L^g_{W_1,W_2,V}u)(n) =\bigr(\me^{-g}+W_1(n\omega)\bigr)u(n+1)+\bigr(\me^{g}+W_2(n\omega)\bigr)u(n-1)+V(n\omega)u(n),\end{equation}
where $g\in\R, \omega\in\T^d$, $W_1,W_2,V\in C(\T^d,\R)$, $d\in\N_+\cup\{\infty\}$ and $\me^{-g}+W_1>0,\me^g+W_2>0$. If $W_1=W_2=0,$ $V(n\omega)$ is random instead of almost-periodic, this is the well-known Hatano-Nelson model \cite{HN}. To the best knowledge of the authors, this is the {\it first} well-studied discrete non-self-adjoint operator in physics literature.  The case $V(n\omega)$ is quasi-periodic has been received more attention  in recent years \cite{longhi,jiang2019interplay,LZC}.

For the operator \eqref{non-hermitian}, the crucial observation is that non-self-adjointness \footnote{In our case, the non-self-adjointness is given by $g\neq 0$.} will give us certain hyperbolicity (consult section \ref{hyse} for more explanations). Benefitting from this, one can anticipate weaker conditions and better results than the self-adjoint case. However, compared to method of Kozlov \cite{Kozlov1984}, our scheme only works for $\ell^2(\Z)$ (resp. $L^2(\R)$), while \cite{Kozlov1984} works equally for $L^2(\R^d)$.

Now let's state our main results.  Our first result concerns smooth quasi-periodic potential $V$, where the frequency satisfies a kind of (strong) Diophantine condition.
For any $0<\gamma<1$,  we denote  $C^{k,\gamma}(\T)$ the vector space  of all continuous functions such that
$$f^{(l)}\in C(\T), f^{(k)}\in C^{0,\gamma}(\T),0\leq l\leq k,$$
and denote $$\|f\|_{k,\gamma}:=\sum\limits_{0\leq l\leq k}\sup\limits_{x\in\T}|f^{(l)}(x)|+\sup\limits_{x,y\in\T,x\neq y}\frac{|f^{(k)}(x)-f^{(k)}(y)|}{|x-y|^\gamma}<\infty,$$
where $f^{(l)}(x)$ is the $l$-$th$ derivative function of $f$. For any  $\omega\in \R\backslash\Q$,  denote $\frac{p_n}{q_n}$ the $n-th$ order continued fraction expansion of
 $\omega$, and denote
$$\mathcal P= \{ \omega\in \R\backslash\Q \quad | \quad  \lim\limits_{n\to\infty}q_n^{\frac{1}{n}}=:\zeta(\omega)< \infty\}.$$
 It  follows from Khinchin's theorem \cite{Khintchine1936} that the set
$\mathcal P$ has full Lebesgue measure.  Once we have this,  our first result can be stated as follows:

\begin{Theorem}\label{Main1}
Let $g\neq 0$, $W_1,W_2,V\in C^{1,\gamma}(\T,\R)$ for some $ 0<\gamma<1.$ Then there
exists an absolute constant $C_0>0$, such that if
\begin{equation}\label{Main1small}
\|W_1\|_{1,\gamma}+\|W_2\|_{1,\gamma}+\|V\|_{1,\gamma}\leq C_0\me^{-12|g|}g^6,
\end{equation}
then for any $\omega\in\mathcal P$,  $\mathcal L^g_{W_1,W_2,V}$
admits a simple ground state energy  and a quasi-periodic ground state $p(n)=P(n\omega)$ with $P\in C(\T,\R)$ and $P>0$.
\end{Theorem}

We first point out that the smallness of $V$ can't be discarded. To see this, one only need to consider the almost Mathieu operator
\begin{equation*}
(\mathcal{L}u)(n)= u(n+1)+ u(n-1) +  2\kappa\cos(\theta+ n\omega) u(n).
\end{equation*}
it is well-known that if $\omega$ is Diophantine and $|\kappa|>1$, then it has a ground state in $\ell^2(\Z)$ \cite{Jitomirskaya1999}, and thus if $g$ is small enough, $\ell^2(\Z)$ eigenfunction of  $\mathcal L^g_{0,0,2\kappa \cos}$ still persists \cite{jiang2019interplay,Liu2021}, which means the operator doesn't have quasi-periodic ground state.

As was pointed out by Kozlov \cite{Kozlov1984}, the problem of finding an almost-periodic ground state contains small denominators, thus the most efficient method is KAM.
We stress that Theorem \ref{Main1} is out of the scope of the classical KAM method. One can understand this from two aspects:
First, in applying the KAM method, it requires the presence of a significant number of derivatives of functions $W_1,W_2$, and $V$ \cite{Kozlov1984}. Even for self-adjoint operators, by method of analytic smoothing \cite{zehnder}, one may only anticipate the results for $k\geq 2\tau+ \frac{d}{2}+3$, where $\tau$ is the Diophantine constant, $d$ is the dimensional of the frequency (consult Remark 3 of \cite{Kozlov1984}). However, not to mention whether or not analytic smoothing works in the non-self-adjoint setting, even it works, one can't anticipate approaching $C^{1,\gamma}$ smoothness in our result. Secondly, in the classical KAM results, the smallness of the perturbation must depend on the frequency $\omega$ \cite{Kozlov1984, Liang2021}, while in our result, we can select the perturbation to be independent of $\omega$, thus the result lies in the semi-local regime defined by Avila-Krikorian \cite{AK3} (the readers are invited to consult nice ICM proceedings of Fayad-Krikorian \cite{FK} for more history on semi-global results). To the best knowledge of the authors, Theorem \ref{Main1} gives the first semi-local result in the non-self-adjoint setting.

Once we have this, people may be curious about whether we can remove the smoothness condition on $W_1, W_2$, and $V$, or whether one can completely remove the arithmetic condition on $\omega$. The answer is yes at the cost of perturbating the potential $V$.

\begin{Theorem}\label{Main2}
Let $\omega\in\R^d\backslash\Q^d$, $d\in\N_+,$ $g\neq 0$, $W_1, W_2, V\in C(\T^d,\R)$. Then there exists  an absolute constant $C_0>0$
such that if
\begin{equation}\label{Main2small}
\|V\|_0+\|W_1\|_0+\|W_2\|_0\leq C_0\me^{-12|g|}g^6,
\end{equation}
then for any $\varepsilon>0$, there exists $V'\in C(\T^d,\R)$ with $\|V'-V\|_{0}\leq \varepsilon$ such that
$\mathcal L^g_{W_1,W_2,V'}$
admits a simple ground state energy and a quasi-periodic ground state $p(n)=P(n\omega)$ with $P\in C(\T^d,\R)$ and $P>0$.
\end{Theorem}

Theorem \ref{Main2} is also surprising from the point of Hamiltonian systems. As we said above, the smoothness of the systems is necessary to apply the KAM method. If the smoothness $k$ is smaller than some threshold, then KAM tori might be destroyed \cite{CW,H,H2}.  For example, for exact and area-preserving twist maps on the annulus, it was proved by Herman \cite{H} that invariant circles can be destructed by $C^{3-\delta}$ arbitrarily small perturbations where $\delta$ is a small positive constant. Secondly,  arithmetic condition on $\omega$ is also necessary for applying the KAM method.  For example, the well-known results of Mather and Forni \cite{F,Mather} says that for certain Liouvillean frequency,  the invariant circles with the frequency can be destroyed by small perturbations in finer topology.

Theorem \ref{Main1} and Theorem \ref{Main2} are restricted to the quasi-periodic ground state. Now we state the existence of a special kind of ``real" almost-periodic ground state, i.e. the hull
$\mathcal H(f)$  of the almost periodic sequence $f$ is isomorphic to $\T^{\infty}=\prod\limits _{i\in \mathbb N}\mathbb T^1$ with product topology.

We assume that the frequency $\omega=(\omega_j)_{j\in
\N}$ belonging to  the infinite dimensional cube $  \mathcal{R}_0:= [1,2]^{\N}$, which is endowed with the probability measure $\mathcal{P}$
induced by the product measure of the infinite dimensional cube $\mathcal{R}_0$.
Define the set of Diophantine frequencies  that was first developed by Bourgain \cite{Bourgain2005}:
\begin{Definition}[\cite{Bourgain2005}]\label{Almost}
Given $\gamma\in (0,1),\tau>1$, we denote by $DC_\infty(\gamma,\tau)$ the set of Diophantine frequencies $\omega\in \mathcal{R}_0$ such that
$$
\inf\limits_{n\in\Z}|\langle k,\omega\rangle-n|\geq \gamma\prod\limits_{j\in\mathbb N}\frac{1}{1+|k_j|^\tau\langle j\rangle^\tau},\ \forall k\in\mathbb Z^\infty, 0<\sum
\limits_{j\in\mathbb N}|k_j|<\infty,
$$
where  $\langle j\rangle:=\max\{1,|j|\}$.
\end{Definition}
As proved in \cite{Bourgain2005,Biasco2019}, for any $\tau>1$, Diophantine frequencies  $DC_\infty(\gamma,\tau)$ are typical in the set $\mathcal{R}_0$ in the sense that there exists a positive constant $C(\tau)$ such that
$$ \mathcal{P}( \mathcal{R}_0 \backslash DC_\infty(\gamma,\tau) ) \leq C(\tau) \gamma.$$
Denote $\mathbb Z_*^\infty:=\{k\in\mathbb Z^\infty:|k|_1:=\sum\limits_{j\in\mathbb N}\langle j\rangle|k_j|<\infty\}$ the set of infinite integer vectors with finite support. For any $F\in C(\T^\infty,\R)$, $k\in \mathbb Z_*^\infty$,  we define $\hat F(k):=\int_{\T^\infty} F(\theta)\me^{-\mi\langle k,\theta\rangle}d\theta$.
Once we introduce this, our result can be stated as follows:

\begin{Theorem}\label{Main3}
 Let $\omega\in DC_\infty(\gamma,\tau)$, $0<\gamma<1,\tau>1,\ g\neq 0, r>0$,  $W_1,W_2,V\in C(\T^\infty,\R)$. There exists an absolute constant $C_0>0$ such that if
 \begin{equation}\label{Main3small}
 \sum\limits_{k\in\mathbb Z_*^\infty}(|\hat{W}_1(k)|+ |\hat W_2(k)|+ |\hat V(k)|)   \me^{r|k|_1}\leq C_0\me^{-12|g|}g^6,
 \end{equation}
  then $\mathcal L^g_{W_1,W_2,V}$
 admits a simple ground state energy and an almost-periodic ground state $p(n)=P(n\omega)$ with $P\in C(\T^\infty,\R),$ $P>0$, and $\sum\limits_{k\in\mathbb Z_*^\infty}|\hat P(k)|\me^{r'|k|_1}<\infty$ for any $r'\in (0,r)$.
\end{Theorem}

Similar to Theorem \ref{Main1}, the smallness of the perturbation doesn't depend on $\omega$.
To the best knowledge of the authors, our result gives the first semi-local result with infinite dimensional frequency, previous semi-local results are restricted to one-frequency case \cite{Avila2015, A3,AK2, houyou,youzhou}.

\subsection{Applications}

Now we give some applications of our main results.
Before that, to state our application in a broader way, we  notice that our main results actually state the existence of the positive almost periodic eigenfunction of the difference operator
\begin{equation}\label{eq:ellidiv}
D^*\bigr(A_1(n\omega)Du\bigr)+A_2(n\omega)D^*u+W(n\omega)u.
\end{equation}
where  $Du(n):=u(n+1)-u(n),$ $D^*u(n):=u(n)-u(n-1)$. \eqref{eq:ellidiv} is usually used to describe the discrete diffusion processes:  $D^*\bigr(A_1(n\omega)Du\bigr)$ is the diffusion term, $A_2(n\omega)D^*u$ is the drifting term and $W(n\omega)u$ is the local growth term.
Indeed, under the assumption $\bigr(\hat A_1(0)-\hat A_2(0)\bigr),\hat A_1(0)>0$, \eqref{eq:ellidiv} can be rewritten as
\begin{equation}\label{jacobi-diff}
\begin{aligned}
h\biggr[\biggr(\me^{-g}+\frac{W_1(n\omega)}{h}\biggr)p
(n+1)+\biggr(\me^g+\frac{W_2(n\omega)}{h}\biggr)p(n-1)+\frac{W(n\omega)}{h}p(n)\biggr],
\end{aligned}
\end{equation}
where $\me^{-g}=\sqrt{\frac{\hat A_1(0)}{\hat A_1(0)-\hat A_2(0)}},\ h=\sqrt{\bigr(\hat A_1(0)-\hat A_2(0)\bigr)\hat A_1(0)}$ (one can consult section \ref{section4} for details).  Then  $g\neq 0$ corresponds to $\hat A_2(0)\neq 0$. Thus roughly speaking, \cite{Kozlov1984} deals with the non-self-adjoint operator which is close to the self-adjoint  one, while we deal with the non-self-adjoint operator which is far from the self-adjoint  ones. Now we give the applications as follows:

\subsubsection{Steady states of the KPP-type discrete diffusion equation}

First, we consider the positive steady state of the following KPP type equation in discrete media
\begin{equation}\label{eq:diffusion}
u_t=D^*(a(n)Du)+b(n)D^*u+c(n)u-u^2,\  (t,n)\in\R\times\Z,
\end{equation} where  $a,b,c$ are bounded and $\inf\limits_{n\in\Z} a(n),\inf\limits_{n\in\Z} \bigr(a(n-1)-b(n)\bigr)>0.$
This equation is usually used to model the growth and discrete migration of species in mathematical biology. Its dynamics and propagation phenomena are investigated in many papers, e.g. see  \cite{Liang2021,G1,LZ} and the references therein.

A steady state is a function $u=u(n)$ satisfying \eqref{eq:diffusion} which is independent of time $t$. Notice that in the previous work, before the further study, \eqref{eq:diffusion} was supposed to have the monostable property, that is, \eqref{eq:diffusion} has a unique positive steady state. Moreover, this steady state is stable. We should notice that the monostable property can be obtained under some additional hypotheses on $c$. For example, applying the sub- and super-solution method, the existence of a bounded positive steady state can be obtained from the hypothesis $\inf c>0$.  If we further suppose that $a,b,c$ are periodic, the bounded positive steady state is unique and also periodic.

However, there is no progress on the positive steady state of \eqref{eq:diffusion}  when $a,b,c$ are almost periodic and $\inf c\leq 0 $.
Here, after establishing the existence of the almost periodic ground state of $Lu=D^*(a(n)Du)+b(n)D^*u+c(n)u$,  a different sufficient condition for the positive steady state is provided below.

\begin{Proposition}\label{KPPSS}
Let $a,b,c$ be almost periodic, and suppose that for some $E_0>0$, there exists an almost periodic function  $\phi$ with $\inf\phi>0$ satisfying
 $$D^*(a(n)D\phi)+b(n)D^*\phi +c(n)\phi=E_0\phi.$$
  Then \eqref{eq:diffusion} has a unique bounded steady state $u$ with $\inf u>0$. Furthermore, $u$ is almost periodic.
\end{Proposition}

In the following, let's consider a more special equation
\begin{equation}\label{eq:diffusion-1}
u_t=A_0D^*Du+  B_0D^*u+ W(n\omega) u-u^2, \text{ in }\R\times\Z.
\end{equation}
Here we always assume that $A_0,B_0\in\R$ with $A_0>\max\{0,B_0\}$.

Based on Proposition \ref{KPPSS} and  Theorem \ref{Main1} (resp. Theorem \ref{Main2}), we can provide several different sufficient conditions on the existence of the positive steady state of  \eqref{eq:diffusion-1}.

Recall that $\hat V(0)=\int_{\T^d}V(\theta)d\theta$, for $d\in\N_+,$ we have the following result:
\begin{Theorem}\label{APPMain1}
Let $\omega\in\mathcal P,\ 0<\gamma< 1$, $B_0\neq 0$, and suppose that $W\in C^{1,\gamma}(\T,\R)$ with $\hat W(0)\geq 0$, $W\not\equiv 0$. Then there
exists $\epsilon_0=\epsilon_0(A_0,B_0)$ such that for any $W$ with
\begin{equation*}\label{smallcond1}
\|W-\hat W(0)\|_{1,\gamma}\leq \epsilon_0,
\end{equation*}
 \eqref{eq:diffusion-1} has a unique bounded steady state $u(n)=U(n\omega)$ with $U\in C(\T,\R)$ and $U>0$.
\end{Theorem}

\begin{Theorem}\label{APPMain2}
Let $\omega\in\R^d\backslash\Q^d$, $d\in\N_+$, and suppose that $B_0\neq 0$, $W'\in C(\T^d, \R)$, $\hat W'(0)\geq 0$. There exists $\epsilon_0=\epsilon_0(A_0,B_0)>0$, and if $$\|W'-\hat W'(0)\|_0<\epsilon_0,$$
 then for any $\epsilon>0$, there exists $W\in C(\T^d, \R)$ which satisfies $\|W-W'\|_0<\epsilon$ and $\hat W(0)=\hat W'(0)$ such that
 \eqref{eq:diffusion-1} has a unique bounded steady state $u(n)=U(n\omega)$ with $U\in C(\T^d,\R)$ and $U>0$.
\end{Theorem}
\begin{Remark}
Notice that $\hat W(0)\geq 0, W\not=0$ is a weaker assumption than $\inf\limits_{n\in\Z} W(n\omega)>0$. Actually, if a potential $W$ satisfies $\hat W(0)=0$ and $W\not=0$, then  $\inf\limits_{n\in\Z} W(n\omega)<0$.
\end{Remark}
\begin{Remark}
The assumption of $\hat W(0)\geq 0, W\not= 0$ guarantees $E_0>0$ in Proposition \ref{KPPSS}, and it is a necessary condition as we could check from the case $W=0$. In fact, Theorem \ref{APPMain1} (Theorem \ref{APPMain2}) is also valid for a more general equation:
 $$u_t=D^*(A_1(n\omega)Du)+  B_0D^*u+ W(n\omega) u-u^2, \text{ in }\R\times\Z,$$
 if we impose on an additional assumption $\|A_1-\hat A_1(0)\|_{1,\gamma}\leq\epsilon_0$ ($\|A_1-\hat A_1(0)\|_{0}\leq\epsilon_0$) with $\epsilon_0$ depending on $\hat A_1(0), B_0$ .
\end{Remark}

\subsubsection{Averaging problem for the discrete parabolic equation with large lower term}

The averaging problem for the parabolic equations is to investigate the asymptotic behavior of the solution for the parabolic equations with rapidly oscillating coefficients (such as periodic, almost periodic, random).

Let $\Omega$ be a Lipschitz domain in $\R^n, n\in\N_+$, and $u_\epsilon\in H^1(\Omega)$ be the weak solution of the Dirichlet problem
\begin{equation}\label{eq:paraver}
\left\{
\begin{aligned}
\rho_\epsilon(x)\partial_tu_\epsilon-div(\mathcal A_\epsilon(x)\nabla u_\epsilon)&=0 & & \text{ in }[0,T]\times\Omega,\\
u_\epsilon|_{\partial\Omega}&=\varphi & & \text{ in }\Omega,
\end{aligned}
\right.
\end{equation}
where $\mathcal A=\{a_{ij}\}, \mathcal A_\epsilon(x)=\mathcal A(\epsilon^{-1}x), \rho_\epsilon(x)=\rho(\epsilon^{-1}x)$ are periodic, almost periodic or random. Moreover, $\mathcal A(x)$ satisfies the ellipticity condition
$$a_{ij}(x)\xi_i\xi_j\geq \lambda|\xi|^2,\quad \forall \xi, x\in\R^n, \lambda>0.$$
In the periodic case, the qualitative homogenization theory for   \eqref{eq:paraver} has been known since 1970s \cite{Lions2011}. Indeed, as $\epsilon\to 0$, the solution $u_\epsilon$ converges to the solution $u_0$ of a parabolic equation with constant coefficients, and the equation with constant coefficients is called averaged (homogenized) equation.

In recent years, there have been a great many works in the quantitative homogenization theory for partial differential equations. Some progress has been made for elliptic equations and systems of the divergence type in the periodic and non-periodic settings. Besides that, the convergence rate of $u_\epsilon$ was also obtained, e.g. see \cite{Avellaneda1989, Kozlov1978, Kozlov1979, spagnolo1968, Papanicolaou1979, Jikov1994, Kenig2012}.

In 1984, Kozlov \cite{Kozlov1984} studied the averaging problem for a singularly perturbed parabolic equation with large lower order term:
\begin{equation}\label{eq:Kozlov'sresult}
\left\{
\begin{aligned}
(u_\epsilon)_t&=\bigr(a(\frac{x}{\epsilon})(u_\epsilon)_x\bigr)_x+\frac{1}{\epsilon} b(\frac{x}{\epsilon})(u_\epsilon)_x+\frac{1}{\epsilon^2}c(\frac{x}{\epsilon})u_\epsilon & & \text{ in }[0,T]\times\R,\\
u_\epsilon|_{t=0}&=\psi & & \text{ in }\R,
\end{aligned}
\right.
\end{equation}
where $a,b,c$ are sufficiently smooth quasi-periodic functions, and $\psi\in L^2(\R)$. He proved that if the frequencies of $a,b,c$ satisfy some kind of arithmetic condition, and if there exist quasi-periodic functions $p, p_*$ and $\lambda_0\in\R$
 that satisfy
 \begin{equation}\label{eq:Kozlov's assump}
 (a(x)p_x)_x+b(x)p_x+c(x)p=\lambda_0 p,\quad  (a(x)(p_*)_x)_x-(b(x)p_*)_x+c(x)p_*=\lambda_0 p_*,
 \end{equation}
 and their means $\langle p\rangle=\langle p^*\rangle=1,$
then we have
 $$\lim\limits_{\epsilon\rightarrow 0}\frac{u_\epsilon(t,x-lt/\epsilon)}{p(x/\epsilon-lt/\epsilon^2)}\me^{-\lambda_0t/\epsilon^2}=v_0(t,x),$$
 where the convergence is uniform in any compact set in $\R$ for any $t\in (0,T]$, and $v_0$ is the solution of the averaged problem
 \begin{equation}\label{eq:Kozlovaveraged}
 \left\{
 \begin{aligned}
 (v_0)_t&=\hat a(v_0)_{xx}& & \text{ in }[0,T]\times\R,\\
 v_0|_{t=0}&=\frac{1}{\langle pp_*\rangle}\psi(x)& &  \text{ in }\R,
 \end{aligned}
 \right.
 \end{equation}
with the constant $\hat a>0$ only depending on $a,b,c$.

Besides that, equation \eqref{eq:Kozlov'sresult} was also studied in the periodic case \cite{Patrizia2005, Allaire2007}. However, as for the almost periodic case, we still don't know anything about the averaging.

Now we consider the averaging problem for the discrete version of \eqref{eq:Kozlov'sresult}. First we introduce the difference operator on $L^2(\epsilon\Z)$: $D_\epsilon u(x)=\frac{u(x+\epsilon)-u(x)}{\epsilon},$ $D^*_\epsilon u(x)=\frac{u(x)-u(x-\epsilon)}{\epsilon}$, and define the following function space on $\epsilon\Z$:
$$L^2(\epsilon\Z):=\biggr\{u:\epsilon\Z\to \R|\ \|u\|^2_{L^2(\epsilon\Z)}:=\epsilon^2\sum
\limits_{x\in \epsilon\Z}|u(x)|^2<\infty\biggr\}.$$

Let  $\omega\in\mathrm{DC}_{\infty}(\gamma,\tau)$, $0<\gamma<1$, $\tau>1$. 
Consider the discrete parabolic equation with almost periodic coefficients in  $[0,T]\times \epsilon\Z$, with $T>0$ as follows:
\begin{equation}\label{eq:averagingpara}
\left\{
\begin{aligned}
&\partial_tu_\epsilon=D_\epsilon^*(A_1(\frac{\omega x}{\epsilon})D_\epsilon u_\epsilon)+\frac{1}{\epsilon}A_2(\frac{\omega x}{\epsilon})D_\epsilon^*u_\epsilon+\frac{1}{\epsilon^2}W( \frac{\omega x}{\epsilon})u_\epsilon,\\
&u_\epsilon(0,x)=\varphi(x) \text{ \ in }\epsilon\Z,
\end{aligned}
\right.
\end{equation}
where $A_1, A_2, W\in C(\T^\infty,\R)$ with $\inf A_1(\theta), \inf (A_1(\theta)-A_2(\cdot+\omega))>0$,  $\varphi\in C^1(\R)$ and satisfies $|\varphi(z)|\leq C|z|^{-k}$ with $k\geq 2,\ z\in\R$. Note that in \eqref{eq:averagingpara}, $\varphi$ can be restricted as a function in $\epsilon\Z$.

The averaging result involves, as in quasi-periodic media \cite{Kozlov1984}, the principal eigenvalue problem of the operator
\begin{equation}\label{eq:discreteeig}
\mathcal Lp=D^*(A_1(n\omega)Dp)+A_2(n\omega)D^*p+W(n\omega)p=E_0p.
\end{equation}
As the application of Theorem \ref{Main3}, the existence of the almost periodic ground state of \eqref{eq:discreteeig} is given by:
\begin{Proposition}\label{Main1ell}
 Let $\omega\in DC_\infty(\gamma,\tau)$, $0<\gamma<1,\tau>1,\ r>0$,  $A_1,A_2,W\in C(\T^\infty,\R)$. Suppose that $\hat A_2(0)\neq 0.$ Then there
exists $\epsilon_0>0$,  depending only on $\hat A_1(0),\hat A_2(0)$, such that if
 \begin{equation*}
 \sum\limits_{k\in\mathbb Z_*^\infty}(|\hat {A_1}(k)|+|\hat {A_2}(k)|+|\hat {W}(k)|) \me^{r|k|_1}-|\hat A_1(0)|-|\hat A_2(0)|-|\hat W(0)|\leq\epsilon_0,
\end{equation*}
  then \eqref{eq:discreteeig}  admits a  simple ground state energy and an almost-periodic ground state $p(n)=P(n\omega)$ with $P\in C(\T^\infty,\R), P>0$. Moreover, $\sum\limits_{k\in\mathbb Z_*^\infty}|\hat {P}(k)| \me^{r'|k|_1}<\infty$ for any $r'\in (0,r)$.
\end{Proposition}

Denote $\lfloor z\rfloor$ the floor of $z$. Now we state the result for the discrete parabolic equation:
\begin{Theorem}\label{APPMain3}
Under the assumption of Proposition \ref{Main1ell}, let $p$ be the almost periodic ground state of \eqref{eq:discreteeig}. Then the solution $u_\epsilon\in L^2(0,T;L^2(\epsilon\Z))$ of the equation \eqref{eq:averagingpara}
 satisfies the relation
\begin{equation}\label{eq:Main3relation}
\lim\limits_{\epsilon\rightarrow 0}\frac{u_\epsilon(t,\lfloor z/\epsilon-lt/\epsilon^2\rfloor\epsilon)}{p(\lfloor z/\epsilon- lt/\epsilon^2\rfloor)}\me^{-E_0t/\epsilon^2}=u_0(t,z),\ z\in\R,
\end{equation}
where the convergence is uniform in any compact set in $\R$ for any $t\in (0,T]$, and $u_0$ is the solution of
the averaged equation
\begin{equation}\label{eq:Mainaveraged}
\left\{
\begin{aligned}
(u_0)_t&=\bar a\partial_{zz}u_0 & & \text{ in }[0,T]\times\R,\\
 u_0|_{t=0}&=c_0\varphi & & \text{ in }\R,
\end{aligned}
\right.
\end{equation}
for $\bar a, c_0>0, l\in \R$ being constants which only depend on $A_1,A_2,W$.
\end{Theorem}

Although the result is very similar to Kozlov's, our discrete model is different from the Kozlov's continuous setting in some ways. First, as the application of Theorem \ref{Main3}, we discuss the averaging result for discrete parabolic equation in the almost periodic media, while in \cite{Kozlov1984}, Kozlov only established the averaging result in quasi-periodic media. Secondly, Kozlov's work can not be directly applied to the discrete model due to the loss of continuity. In fact, under the assumption of \eqref{eq:Kozlov's assump}, Kozlov observed that $v_\epsilon(t,x):=\me^{-\frac{\lambda_0t}{\epsilon^2}}\frac{u_\epsilon(t,x-\frac{lt}{\epsilon})}{p(\frac{x}{\epsilon}-\frac{lt}{\epsilon^2})}$ satisfies the equation defined in $[0,T]\times\R$,
\begin{equation}\label{eq:Kozlov's transformed}
\left\{
\begin{aligned}
(pp_*)(x-\frac{lt}{\epsilon})(v_\epsilon)_t&=\bigr(\tilde a(x-\frac{lt}{\epsilon})(v_\epsilon)_x\bigr)_x+\frac{1}{\epsilon}\bigr(\bar c-l(pp_*)(x-\frac{lt}{\epsilon})\bigr)(v_\epsilon)_x, & &  \\
v_\epsilon|_{t=0}&=\frac{\psi(x)}{p(x/\epsilon)}, & & x\in\R,
\end{aligned}
\right.
\end{equation}
where $\tilde a>0, \bar c\equiv \const$ which only depend on $a,b,c$, and $l\in\R$ to be specified. After solving the cell problem of \eqref{eq:Kozlov's transformed}:
\begin{equation}\label{eq:cell problem}
(\tilde a(x)u_x)_x-\bar c u_x=\tilde a_x+\bar c-lpp_*,
\end{equation}
the averaging result can be obtained by the multiscale expansion method.
Note that \eqref{eq:cell problem} is solvable for $l=\frac{\bar c}{\langle pp_*\rangle}$ and the additional arithmetic condition on frequencies.
Unfortunately, the transformation $v_\epsilon(t,x)=\me^{-\frac{\lambda_0t}{\epsilon^2}}\frac{u_\epsilon(t,x-\frac{lt}{\epsilon})}{p(\frac{x}{\epsilon}-\frac{lt}{\epsilon^2})}$ seems to lose its effect in our discrete setting.
To overcome this difficulty, we will extend \eqref{eq:averagingpara} continuously in $[0,T]\times\R$ by lifting it in the infinite dimensional torus and restricting it in the line $\{\omega z| z\in\R\}$, and then the resulting equation can be analyzed in Kozlov's method.


\subsection{Structure of the paper}
Now we organize our paper as follows:
In section 2, we introduce some notations and preliminary knowledge which will be needed in the rest of the paper. In section 3,  based on the reducibility of $\rm{GL}(2,\R)$ cocycles,  we establish a  general criterion of existence of positive almost periodic eigenfunction (Proposition \ref{criterion}), and then apply it to obtain the existence of the positive almost periodic eigenfunction in various conditions.

In section 4, we prove the positive almost periodic eigenfunction is exactly the almost periodic ground state and the associated eigenvalue is the simple ground state energy.  Then we discuss the fundamental properties of the ground state. This finishes the proof of Theorem \ref{Main1},  Theorem \ref{Main2},  Theorem \ref{Main3}.

In section 5, as applications of Theorems \ref{Main1} and Theorem \ref{Main2}, we establish the existence of the positive almost periodic steady state of \eqref{eq:diffusion}. Finally, we discuss the asymptotic behavior of the solution of the equation \eqref{eq:averagingpara} and finish the proof of Theorem \ref{APPMain3}.

\section{Preliminaries}

\subsection{$\GL$ cocycle and Jacobi operator}
Let $X$ be a compact metric space, $(X, \nu, T)$ be ergodic, and $A:X\rightarrow \mathrm{GL}(2,\mathbb R)$ be a continuous map.
 A $\mathrm{GL}(2,\mathbb R)$ cocycle over $(T,X)$ is an action defined on $X\times\mathbb R^2$ such that
$$(T,A):(x,v)\in X\times \mathbb R^2\mapsto (Tx,A(x)\cdot v)\in X\times \mathbb R^2.$$
For $n\in\mathbb{Z}$, $A_n$ is defined by $(T,A)^n=(T^n,A_n)$, where $A_{0}(x)=\id$,
\begin{equation*}\quad
A_{n}(x)=\prod_{j=n-1}^{0}A(T^{j}x)=A(T^{n-1}x)\cdots A(Tx)A(x),\  n\ge1,
\end{equation*}
and $A_{-n}(x)=A_{n}(T^{-n}x)^{-1}$.

Let $W_1,W_2,V$ be a continuous function defined on torus $\T^d$ ($d\in \N_+$ or $d=\infty$). For $\omega\in \R^d\backslash\Q^d$,
 we define the non-self-adjoint  Jacobi operator on $\ell^2(\Z):$
$$(\mathcal L_{W_1,W_2,V}^gu)(n)=\bigr(\me^{-g}+W_1(n\omega)\bigr)u(n+1)+\bigr(\me^g+W_2(n\omega)\bigr)u(n-1)+V(n\omega)u(n).$$
Note that the eigenfunction problem of $\mathcal L_{W_1,W_2,V}^g$ can be represented as:
$$
\begin{pmatrix}
u(n+1)\\
u(n)
\end{pmatrix}
=S^{E,g}_{W_1,W_2,V}(\theta+n\omega)
\begin{pmatrix}
u(n)\\
u(n-1)
\end{pmatrix},
$$
where we denote
$$S^{E,g}_{W_1,W_2,V}(\theta)=\begin{pmatrix}
\frac{E-V(\theta)}{\me^{-g}+W_1(\theta)} &-\frac{\me^{g}+W_2(\theta)}{\me^{-g}+W_1(\theta)}\\
1      &0
\end{pmatrix}
.$$ Then $(\omega,S_{W_1,W_2,V}^{E,g})$ defines a $\GL$ cocycle on $\T^d \times \R^2,$ and we will call this non-self-adjoint Jacobi cocycle.

\begin{Remark}\label{schro-form}
If $g=0$, $W_1=W_2=0$, then it reduces to the classical (self-adjoint) Schr\"odinger cocycle \cite{Liang2021}.
\end{Remark}

\subsection{Maximum principle, existence and uniqueness for the Cauchy Problem}

The maximum principle on the whole space can be stated as follows:
\begin{Proposition}[Maximum principle\cite{ShenCao}]\label{Maximum}
Let $v\in \ell^\infty(\mathbb Z)$. Assume that for any bounded interval $I=[0,t_0]\subset[0,\infty)$, $u$ is bounded in $I\times\mathbb Z$.  
If $u$ satisfies

\begin{equation}\label{eq:elliptic}
\left\{
\begin{aligned}
u_t-D^*\bigr(a(n)Du\bigr)-b(n)D^*u+c(n)u&\geq 0 &\text{ in }I\times\Z,\\
u(0,n)&\geq 0  &\text{ in }\Z,
\end{aligned}
\right.
\end{equation}
then $u\geq 0$ in $I \times\mathbb Z.$
\end{Proposition}

The following Harnack inequality is a very useful technique when we study the properties of the solution of \eqref{eq:elliptic}. We will present it here for the reader's convenience.

\begin{Proposition}[Harnack inequality\cite{Liang2018}]\label{Harnack}
Assume that $u$ is bounded on $(0,\infty)\times\mathbb Z$ and solves \eqref{eq:elliptic}. Then for any $(t,n)\in(0,\infty)\times\mathbb Z,\ T>0$, there exists a positive constant $C=C(T)$ such that
$$u(t,n)\leq C(T)u(t+T,m),\ m\in\{ n\pm 1,n\}.$$
\end{Proposition}

\begin{Remark}
The proof of the Harnack inequality could be found in \cite{Liang2018} with the initial value $u(0,n)$ having a finite support. However, we should notice that the argument can be applied to \eqref{eq:elliptic} similarly with minor modification.
\end{Remark}

Combining the Harnack inequality with the maximum principle, we can deduce the strong maximum principle as follows:
\begin{Corollary}[Strong maximum principle\cite{Liang2018}]\label{StrongMaximum}
Under the assumption of Proposition \ref{Maximum}, either $u\equiv 0$ or $u>0$ in $I\times\mathbb Z$.
\end{Corollary}

The comparison principle is a consequence of the strong maximum principle, and it is useful for us to construct the almost periodic steady state. To state it, we first define super-sub solutions of the discrete Fisher-KPP type equation:

Let $\bar u,\underline u\in C(\mathbb R\times\mathbb Z)$ be two bounded functions.
 We say that $\bar u$  is a supersolution of
 \begin{equation}\label{eq:para}
 \bar u_t-D^*\bigr(a(n)D\bar u\bigr)-b(n)D^*\bar u-c(n)\bar u+\bar u^2=0
 \end{equation}
  if for any given $n\in\mathbb Z$, $\bar u$ is absolutely continuous in $t$ and satisfies
$$
\bar u_t-D^*(a(n)D\bar u)-b(n)D^*\bar u-c(n)\bar u+\bar u^2\geq 0 \text{ for }t\in (0,\infty),
$$
and $\underline u$ is a subsolution  if  for any given $n\in\mathbb Z$, $\underline u$ is absolutely continuous in $t$ and satisfies
$$
\underline u_t-D^*(a(n)D\underline u)-b(n)D^*\underline u-c(n)\underline u+\underline u^2\leq 0 \text{ for }t\in (0,\infty).
$$

The strong comparison principle is given by
\begin{Proposition}[Strong comparison principle]\label{Strong comparison}
Let $\overline u$ and $\underline u$ be a supersolution and a subsolution of \eqref{eq:para}
 repectively. If $\underline u(0,n)\leq \overline u(0,n)$ in $\mathbb Z$, then $\underline u(t,n)<\overline u(t,n)$ or $\underline u\equiv \overline u$ in $(0,\infty)\times\mathbb Z$.
\end{Proposition}

Usually, well-behaved Cauchy Problem possesses the property that it admits a unique global solution. For the convenience of our applications, we give the following existence and uniqueness theorems for discrete Fisher-KPP type equations and linear parabolic equations.
\begin{Theorem}[\cite{Pazy}]\label{existence and uniqueness}
For any initial value $\varphi\in \ell^\infty(\mathbb Z)$, there exists a unique $u\in C([0,\infty)\times\mathbb Z)$ with $u(t,\cdot)\in \ell^\infty(\mathbb Z)$ for any $t\in[0,\infty)$ such that
\begin{equation*}
\left\{
\begin{aligned}
u_t-D^*(a(n)Du)-b(n)D^*u&-c(n)u+u^2=0& &\text{ in }[0,\infty)\times\mathbb Z,\\
u(0,n)&=\varphi(n) & & \text{ in }\Z.
\end{aligned}
\right.
\end{equation*}
\end{Theorem}

\begin{Theorem}[\cite{Pazy}]\label{existenceinl^2}
For any initial value $\varphi\in \ell^2(\mathbb Z)$, there exists a unique $u\in C([0,\infty)\times\mathbb Z)$ with $u[t,\cdot)\in \ell^2(\mathbb Z)$ for any $t\in(0,\infty)$ such that
\begin{equation*}
\left\{
\begin{aligned}
u_t-D^*(a(n)Du)-b(n)D^*u&-c(n)u=0& &\text{ in }[0,\infty)\times\mathbb Z,\\
u(0,n)&=\varphi(n) & & \text{ in }\Z.
\end{aligned}
\right.
\end{equation*}
\end{Theorem}

\section{Construction of positive almost-periodic eigenfunction}

In this section, we develop a new method to construct the positive almost periodic eigenfunction of the non-self-adjoint operator \eqref{nonsaoperator}. The method is from the side of dynamical systems.

\subsection{Criterion of the existence of positive almost periodic eigenfunction}
First we give a criterion, which says that if the non-self-adjoint Jacobi cocycle  $(\omega,S_{W_1,W_2,V}^{E,g})$ is reducible to a constant cocycle with eigenvalue 1, and the conjugacy has special form, then the corresponding non-self-adjoint Jacobi operator admits a positive almost-periodic eigenfunction. Recall that a $\GL$ cocycle $(\omega, A)$ is said to be reducible if there exist $B(\cdot)\in C(\T^d,\GL)$ and  constant  $\tilde A \in \GL$ such that $$B(\theta+\omega)^{-1}A(\theta)B(\theta)=\tilde A.$$
Then we have the following:
\begin{Proposition}\label{criterion}
Let $d\in\N_+\cup\{\infty\}$. Suppose that there exists $B \in C(\T^d,\GL)$ such that
\begin{equation}\label{redu}
  B(\theta+\omega)^{-1}S_{W_1,W_2,V}^{E,g}(\theta)B(\theta)=\begin{pmatrix}
  1 &0\\
  0 &s
  \end{pmatrix},
  \end{equation}
where $s\in\R$. Moreover, $B$ takes the form
 $$B(\theta)=\begin{pmatrix}
 B_{11}(\theta) & B_{12}(\theta)\\
 B_{21}(\theta) & B_{22}(\theta)
 \end{pmatrix}:=P\me^{Y(\theta)}\begin{pmatrix}
\me^{y_1(\theta)} & 0\\
0                 & \me^{y_2(\theta)}
\end{pmatrix}  \in C(\T^d,\GL)$$
where $P\in \GL$, and
\begin{equation}\label{eqy}\|Y\|_{0}\leq  \tilde{\epsilon}\leq \frac{|P_{11}|}{4(|P_{11}|+|P_{12}|)}. \end{equation}
 Then $u(n)=U(n\omega)$ is an almost periodic solution  of
   \begin{equation}\label{schroeq}
 \bigr(\me^{-g}+W_1(n\omega)\bigr)u(n+1)+\bigr(\me^{g}+W_2(n\omega)\bigr)u(n-1)+V(n\omega)u(n)=Eu(n),
 \end{equation}
 where $U\in C(\T^d,\R),\  U>0.$
\end{Proposition}

\begin{Remark}
Note here we don't need any control on $y_i(\theta)$,  that's the main reason why we obtain semi-local results (c.f. Theorem \ref{Main1}, Theorem \ref{Main3}).
\end{Remark}

\begin{Remark}
One should compare Proposition \ref{criterion} with the criterion in the self-adjoint case (Lemma 5.1 of \cite{Liang2021}), where the conjugacy $B$ is only allowed to be close to the identity.
\end{Remark}

\begin{proof}
Denote $B(\theta)=\begin{pmatrix}
B_{11}(\theta) & B_{12}(\theta)\\
B_{21}(\theta) & B_{22}(\theta)
\end{pmatrix}$, then  \eqref{redu} gives us
\begin{eqnarray*}
\frac{E-V(\theta)}{\me^{-g}+W_1(\theta)}B_{11}(\theta)-\frac{\me^{g}+W_{2}(\theta)}{\me^{-g}+W_1(\theta)}B_{21}(\theta)&=&B_{11}(\theta+\omega),\\
B_{11}(\theta)&=&B_{21}(\theta+\omega),
\end{eqnarray*}
that is $$\bigr(\me^{-g}+W_1(\theta)\bigr)B_{11}(\theta+\omega)+\bigr(\me^{g}+W_2(\theta)\bigr)B_{11}(\theta-\omega)+V(\theta)B_{11}(\theta)=EB_{11}(\theta).$$
Hence $u(n)=B_{11}(n\omega)$ is the almost periodic solution of \eqref{schroeq}.

Denote that $\me^{Y(\theta)}=\begin{pmatrix}
Y_{11}(\theta) & Y_{12}(\theta)\\
Y_{21}(\theta) & Y_{22}(\theta)
\end{pmatrix}.$ Direct computation shows that
$$B_{11}(\theta)=\me^{y_1(\theta)}\bigr(P_{11}Y_{11}(\theta)+P_{12}Y_{21}(\theta)\bigr).$$
Without loss of generality,  assume $P_{11}>0$ (otherwise consider $-B_{11}$ instead). By our assumption,
 $$\|Y_{11}-1\|_{0}\leq  2\tilde{\epsilon}, \|Y_{12}\|_{0}\leq 2\|Y\|_{0}\leq 2\tilde{\epsilon},$$
 thus by \eqref{eqy}, we have
 $$B_{11}(\theta) \geq  \frac{P_{11}}{2} \me^{ \min_{\theta}y_1(\theta)}>0.$$
This yields that
$B_{11}(n\omega)$ is a positive almost periodic solution.
\end{proof}

\subsection{Reducibility of the non-self-adjoint Jacobi cocycle}

If we denote
 \begin{equation*}\label{linearizedpart}
A_{g}(E):=\begin{pmatrix}
E\me^{g}      &-\me^{2g}\\
1              &0
\end{pmatrix}
, \tilde{F}_{E,g}(\theta):= \begin{pmatrix}
1                  &0\\
\frac{\me^{-2g}V(\theta)+E\me^{-g}W_1(\theta)}{\me^{-g}+W_1(\theta)} &\frac{\me^{-2g}W_2(\theta)-W_1(\theta)}{\me^{-g}+W_1(\theta)}+1
\end{pmatrix},
\end{equation*}
then $S_{W_1,W_2,V}^{E,g}(\theta)= A_{g}(E)  \tilde{F}_{E,g}(\theta)$. Then basic observation is that if $W_1,W_2,V$ are small enough, then $ \tilde{F}_{E,g}(\theta)$ is close to identity, i.e.  $S_{W_1,W_2,V}^{E,g}$ is close to constant. Hence one can try to use Newton's iteration to conjugate the cocycle $(\omega, S_{W_1,W_2,V}^{E,g})$ to constant.  We will distinguish the proof into three steps:

\subsubsection{Diagonalization  using Implicit Function Theorem}\label{hyse}

The first step is to  diagonalize the cocycle $(\omega, S_{W_1,W_2,V}^{E,g})$ using Implicit Function Theorem. There are two main observations of this step. The first observation (indeed the key observation of the whole proof) is that in the non-self-adjoint case $g\neq 0$ (without loss of generality, assume $g>0$), the constant cocycle $(\omega, A_{g}(E))$ is strongly hyperbolic. To see this,  since we are looking for the almost periodic ground state of the operator  \eqref{non-hermitian}, the energy should lie in the upper bound of the spectrum. Thus without loss of generality, we just consider the case
\begin{equation}\label{hyper}E\geq 2+\delta\end{equation}
for $\delta>0$ to be specified.
Notice that the characteristic equation of $A_{g}(E)$ reads as
\begin{equation}\label{characeq}
\lambda^2-E\me^{g}\lambda+\me^{2g}=0.
\end{equation}
If we denote its eigenvalues as $\lambda(E)$ and $\mu(E)$, direct computation shows that
$$\lambda(E)=\me^{g}\frac{E-\sqrt{E^2-4}}{2}, \qquad  \mu(E)=\me^{g}\frac{E+\sqrt{E^2-4}}{2}.$$
Then under the assumption \eqref{hyper}, $\lambda(E)$, $\mu(E)$ are both real, and furthermore
$$\mu(E) - \lambda(E) \geq  e^{g} \delta.$$
The second observation is that as a consequence of strong hyperbolicity of  $(\omega, A_{g}(E))$,  and if  $V, W_1, W_2$ are small enough,  then one can diagonalize  $(\omega, S_{W_1, W_2, V}^{E,g})$. Just note in this step, we don't need to assume any arithmetic condition on $\omega$, since one will not meet small divisor problem.

To prove this, we need a non-resonance cancelation lemma which will be the basis of our proof. Now we first introduce some settings of it.
Denote $\mathcal B(*)$ as an $*$-valued Banach space with the norm $|\cdot|_{\mathcal B}$, where $*$ will usually denote $\R$, $\gl$ or $\GL$. Moreover, we assume that $\mathcal B(*)$ admits a crucial property:
\begin{Definition}
A Banach space $\mathcal B$ admits condition $(H)$ if $\|fg\|_{\mathcal B}\leq \|f\|_{\mathcal B}\|g\|_{\mathcal B}$ for any $f,g\in\mathcal B.$
\end{Definition}
Let $\mathcal I$ be any closed interval in $\R$. Then we denote the function space
 $$\mathscr B_{\mathcal I}(*):=\{f\in C(\mathcal I\times \T^d,*)| f(\xi,\cdot)\in \mathcal B(*) \text{ for any given } \xi\in \mathcal I\}$$
with the norm $\|f\|_{\mathcal B,\mathcal I}:=\sup\limits_{\xi\in\mathcal I}|f(\xi)|_{\mathcal B}$,  $d\in\N_+\cup\{\infty\}$. For simplicity of the notations, in the following we just denote $\mathcal B(\gl)$ as $\mathcal B$,  and $\mathscr B_{\mathcal I}(\gl)$ as $\mathscr B_{\mathcal I}$.

Assume that for any given $\eta>0,\omega\in \mathbb T^d, A\in C(\mathcal I,\GL)$, and
 we have a decomposition of the Banach space $\mathscr B_{\mathcal I}$ into non-resonant spaces and resonant spaces, i.e.
$\mathscr{B}_{\mathcal I}=\mathscr{B}^{nre}_{\mathcal I}(\eta)\oplus\mathscr{B}^{re}_{\mathcal I}(\eta)$.  Here $\mathscr{B}^{nre}_{\mathcal I}(\eta)$ is defined in the following way:  for any $Y\in\mathscr{B}^{nre}_{\mathcal I}(\eta),$ we have
$$
\begin{aligned}
&A(\xi)^{-1}Y(\xi,\theta+\omega)A(\xi)\in \mathscr{B}^{nre}_{\mathcal I}(\eta) \text{ for any }\xi\in \mathcal I,\\ &|A(\xi)^{-1}Y(\xi,\theta+\omega)A(\xi)-Y(\xi,\theta)|_{\mathcal B}\geq \eta|Y(\xi,\theta)|_{\mathcal B} \text{ for any }\xi\in\mathcal I.
\end{aligned}
$$
For any $f\in\mathscr B_{\mathcal I},$ denote $\widehat{f(\xi)}(k)$ the Fourier coefficient of $f(\xi,\cdot)$ with respect to the second variable $\theta\in\T^d.$

Once we have this, then the following holds:

\begin{Lemma}\label{basic}
Let $\mathcal I$ be a closed interval in $\R$, and
assume that $A\in C\bigr(\mathcal I,\GL\bigr)$.
Then there exists an absolute constant $c_1>0$ such that for any $F\in \mathscr B_{\mathcal I}$ with $\|F\|_{\mathcal B,\mathcal I}\leq\epsilon\leq \frac{(\inf|\det A|)^2 \min \{\eta^2,1\}}{c_1(\|A\|^3+1)^2},$ there exist $Y\in \mathscr{B}_{\mathcal I}$ and $F^{re}\in\mathscr{B}^{re}_{\mathcal I}(\eta)$ such that
$$e^{-Y(\xi,\theta+\omega)}A(\xi)e^{F(\xi,\theta)}e^{Y(\xi,\theta)}=A(\xi)e^{F^{re}(\xi,\theta)},$$
where $\|Y\|_{\mathcal B,\mathcal I}\leq \epsilon^{\frac{1}{2}}$ and $\|F^{re}\|_{\mathcal B,\mathcal I}\leq 2\epsilon.$
\end{Lemma}
\begin{Remark}
This lemma generalizes the previous result of \cite{Cai2019} where $A$ is a constant $\SL$ matrix.
The proof is based on the quantitative Implicit Function Theorem, and it will be given in Appendix A.
\end{Remark}

In the following, we will always fix $\mathcal I:=[2+\frac{1}{9}\min\{g^2,1\},2(\me^{g}+\me^{-g})]$. The key step for us is the following:

\begin{Proposition}\label{ALKAM}
Let  $d\in\N_+\cup\{\infty\},$ $g>0$.
Then there exists  an absolute constant  $C_0>0$ such that if
  $$\|W_1\|_{\mathcal B,\mathcal I}+\|W_2\|_{\mathcal B,\mathcal I}+ \|V\|_{\mathcal B,\mathcal I}\leq\epsilon\leq C_0\me^{-12g}g^6,$$
then there exist $Y, f_i \in \mathscr B_{\mathcal I}$ such that
 $$\me^{-Y(E,\theta+\omega)}P(E)^{-1}S_{W_1,W_2,V}^{E,g}(\theta)P(E)\me^{Y(E,\theta)}=\begin{pmatrix}
 \lambda(E) e^{ f_1(E,\theta)} & 0\\
0   & \mu(E)e^{ f_2 (E,\theta)}
\end{pmatrix},$$
 where  $P(E)=\begin{pmatrix}
 \lambda(E) & \mu(E)\\
 1 & 1
 \end{pmatrix}$ and
\begin{eqnarray}\label{esf}
\|Y\|_{\mathcal B,\mathcal I}\leq\epsilon^{\frac{1}{3}} \leq
\frac{\lambda(E)}{4\lambda(E)+4\mu(E)}, \ \|f_i\|_{\mathcal B,\mathcal I}\leq \epsilon^{\frac{1}{2}}.
\end{eqnarray}
\end{Proposition}

\begin{proof}
It is straightforward to check that for any $E\in \mathcal I$, one has
\begin{equation}\label{eigenvalueseparate}
\begin{aligned}
|\lambda(E)/\mu(E)-1|&=\biggr|\frac{E+\sqrt{E^2-4}}{E-\sqrt{E^2-4}}-1\biggr| 
\geq \frac{\min\{g^2,1\}}{9},\\
|\mu(E)/\lambda(E)-1|&=\biggr|\frac{E-\sqrt{E^2-4}}{E+\sqrt{E^2-4}}-1\biggr|
\geq\frac{\min\{g^2,1\}}{27},\\
\frac{\lambda(E)}{\lambda(E)+\mu(E)}&\geq \frac{E-\sqrt{E^2-4}}{2E}\geq \frac{\me^{-2g}}{16}.
\end{aligned}
\end{equation}

Denote $P(E):=\begin{pmatrix}
\lambda(E) & \mu(E)\\
1     & 1
\end{pmatrix}\in \GL$, then we can diagonalize $A_g(E)$ as
$$A'_{g}(E):=P(E)^{-1}A_{g}(E)P(E)=
diag(\lambda(E),\mu(E)).$$
On the other hand, note that $ \|\tilde{F}_{E,g}(\theta) - \id \|_{\mathcal B,\mathcal I} \leq \me^g\epsilon$. Thus by Implicit Function Theorem, there exists
$F_g\in \mathscr{B}_{\mathcal I}$, such that $ \tilde{F}_{E,g}(\theta)  =\me^{F_{g}(E,\theta)} $ with the estimate $ \|F_{g}\|_{\mathcal B,\mathcal I} \leq 2\me^g\epsilon$.
 Denote $$F'_{g}(E,\theta)=P(E)^{-1}F_g(E,\theta)P(E),$$ then there holds the estimate
\begin{equation}\label{eq:diagonalreduce}
\begin{aligned}
\|F'_{g}\|_{\mathcal B,\mathcal I}&\leq\sup\limits_{E\in\mathcal I} \frac{2\|P(E)\|^2\me^g\epsilon}{|\det P(E)|}\leq \sup\limits_{E\in\mathcal I}\frac{2(E+2)^2\epsilon}{\sqrt{E^2-4}}
\leq C'\me^{2g}\max\{\frac{1}{g^2},1\}\epsilon,
\end{aligned}
\end{equation}
where $C'>0$ is an absolute constant.

Now we are ready to apply the non-resonance cancelation lemma (Lemma \ref{basic}) to the cocycle $(\omega, A'_g(E)e^{F'_g(E,\theta)}).$
Define\begin{equation}\label{LambdaK}
\Lambda=\biggr\{f\in C(\mathcal I\times \T^d,\gl)|f=\begin{pmatrix}
0           & f_1\\
f_2 & 0
\end{pmatrix}\in \mathscr B_{\mathcal I}\biggr\}.
\end{equation}
Then we prove that for any $Y\in \Lambda$, the operator
$$Y\rightarrow  A'_g(E)^{-1}Y(E,\theta+\omega)A'_g(E)-Y(E,\theta)$$ has a bounded inverse. To prove this, we only need to consider the equation
\begin{equation}\label{eq:homologicaleq}
A'_g(E)^{-1}Y(E,\theta+\omega)A'_g(E)-Y(E,\theta)=G(\theta).
\end{equation}
Write that
$$  Y(E,\theta)=
\begin{pmatrix}
0        & Y_{1}(E,\theta)\\
Y_{2}(E,\theta) & 0
\end{pmatrix},\qquad G(E,\theta)=
\begin{pmatrix}
0     & G_{1}(E,\theta)\\
G_{2}(E,\theta) & 0
\end{pmatrix}.$$
Comparing the Fourier coefficients of the equation \eqref{eq:homologicaleq}, we arrive at
\begin{equation*}
\left\{
\begin{aligned}
\widehat{Y_{1}(E)}(k)&=\frac{\widehat{G_{1}(E)}(k)}{\frac{\mu}{\lambda}(E)e^{\mi\langle k,\omega\rangle}-1},\\
\widehat{Y_{2}(E)}(k)&=\frac{\widehat{G_{2}(E)}(k)}{\frac{\lambda}{\mu}(E)\me^{\mi\langle k,\omega\rangle}-1}.
\end{aligned}
\right.
\end{equation*}
Notice \eqref{eigenvalueseparate} implies that
\begin{eqnarray*}
|\lambda(E)/\mu(E)-\me^{\mi\langle k,\omega\rangle}|
& \geq& |1-\lambda(E)/\mu(E)| \geq \frac{\min\{g^2,1\}}{27}, \\
|\mu(E)/\lambda(E)-\me^{\mi\langle k,\omega\rangle}|
 &\geq& |\mu(E)/\lambda(E)-1| \geq \frac{\min\{g^2,1\}}{27}.
\end{eqnarray*}
Thus for any $Y\in \Lambda$, we have
\begin{equation}\label{nonresonant}
\|A'_g(E)^{-1}Y(E,\theta+\omega)A'_g(E)-Y(E,\theta)\|_{\mathcal B,\mathcal I}\geq \frac{\min\{g^2,1\}}{27}\|Y\|_{\mathcal B,\mathcal I},
\end{equation}
which just means $\Lambda \subset \mathscr{B}_{\mathcal I}^{nre}(\frac{\min\{g^2,1\}}{27})$.

Just note  $\det A'_g(E)=\lambda(E)\mu(E)=e^{2g}$, and then
$$  C'\me^{3g}\epsilon\max\{\frac{1}{g^2},1\} \leq \frac{e^{4g}\min\{g^4,1\}}{c_2 \sup\limits_{E\in\mathcal I}|\mu(E)|^6} .,$$
where $c_2$ is an absolute constant.
 It follows from Lemma \ref{basic} that there exist $Y, f_i \in \mathscr B_{\mathcal I}$ such that
\begin{eqnarray*}
&&e^{-Y(E)(\theta+\omega)}P(E)^{-1}S_{W_1,W_2,V}^{E,g} (\theta)P(E)\me^{Y(E)(\theta)}\\
&=&e^{-Y(E)(\theta+\omega)}A'_g(E)e^{F'_g(E,\theta)} \me^{Y(E)(\theta)}\\
&=&\begin{pmatrix}
 \lambda(E) e^{ f_1(E,\theta)} & 0\\
0   & \mu(E)e^{ f_2 (E,\theta)}
\end{pmatrix},
\end{eqnarray*}
 Moreover, there hold the estimates
 \begin{eqnarray*}
\|f_i\|_{\mathcal B,\mathcal I} &\leq& 2C'\me^{2g}\max\{\frac{1}{g^2},1\}\epsilon\leq \epsilon^{\frac{1}{2}},\\
\|Y\|_{\mathcal B,\mathcal I}&\leq& C'^{\frac{1}{2}}\me^{g}\max\{\frac{1}{g},1\}\epsilon^{\frac{1}{2}}\leq \epsilon^{\frac{1}{3}}.
\end{eqnarray*}
\end{proof}

\subsubsection{Solving the cohomological equation}
 Note Proposition \ref{ALKAM} says that one can diagonalize the cocycle, thus to make the cocycle reducible, one only need to solve the typical cohomological equation
  \begin{equation}\label{coho}
y(\theta+\omega)-y(\theta)=f(\theta)-\widehat {f}(0).
\end{equation}
To solve this equation, there arises the small divisor problem, and then one has to assume some regularity condition on $f$, and some arithmetic property on $\omega$. It is well-known that if $\omega$ is Diophantine, say  $\omega\in DC_d(\gamma,\tau),$ where  $\tau>d+1$, $f\in C^k(\T^d,\R)$ then \eqref{coho} has a solution if $k>\tau+1$. \\


 \smallskip

\textbf{Finitely differentiable function}\\

However, if $f$ is a function with lower regularity, this result can be improved in the following:
\begin{Proposition}[\cite{Jaksic2000}]\label{prop:lowregu}
Assume that $f\in C^{k,\gamma}(\T)$ for some $k\geq 1$ and $0<\gamma<1.$
Then for any $\omega\in\mathcal P$, the equation \eqref{coho} possesses a solution $y\in C^{k-1,\gamma'}(\T)$ for any $\gamma'<\gamma$.
\end{Proposition}

Denote $\mathcal B_{k,\gamma}:=C^{k,\gamma}(\T,\gl)$, and the function space $$\mathscr B_{k,\gamma,\mathcal I}:=\{f\in C(\mathcal I\times \T,\gl)| f(E,\cdot)\in \mathcal B_{k,\gamma} \text{ for any given } E\in \mathcal I\}$$
which is equipped with the norm $\|f\|_{k,\gamma,\mathcal I}:=\sup\limits_{E\in\mathcal I}\|f(E)\|_{k,\gamma},$ and it is straightforward to check that $C^{1,\gamma}(\T,\gl)$ admits condition $(H).$

Then as a result of  Proposition \ref{ALKAM} and Proposition \ref{prop:lowregu}, we have the following:

\begin{Lemma}\label{prop:finitediffredu}
Let $g>0,$  $\omega\in\mathcal P,\gamma\in (0,1)$, and suppose that $W_1,W_2,V\in C^{1,\gamma}(\T,\R)$. Then there exists  an absolute constant  $C_0>0$ such that if
$$\|W_1\|_{1,\gamma}+\|W_2\|_{1,\gamma}+\|V\|_{1,\gamma} \leq\epsilon \leq C_0\me^{-12g}g^6,$$
then for any $0<\gamma'<\gamma,$ there exists $B\in\mathscr B_{0,\gamma',\mathcal I}$ such that
$$B(E,\theta+\omega)^{-1}S_{W_1,W_2,V}^{E,g}(\theta)B(E,\theta)=\begin{pmatrix}
 \lambda(E)e^{ \widehat {f_{1}(E)}(0) }& 0\\
0   & \mu(E) e^{\widehat {f_{2}(E)}(0)}
\end{pmatrix},
$$ where $B$ takes the form
$$B(E,\theta):=\begin{pmatrix}
\lambda(E) & \mu(E)\\
1          & 1
\end{pmatrix}\me^{Y(E,\theta)}\begin{pmatrix}
 \me^{ y_{1}(E,\theta)} & 0\\
0   &  \me^{ y_{2}(E,\theta)}
\end{pmatrix}$$
 with the estimate
$$\|Y\|_{0,\gamma,\mathcal I}\leq \epsilon^{\frac{1}{3}}\leq \frac{\lambda(E)}{4(\lambda(E)+\mu(E))}.$$
\end{Lemma}

\begin{proof}
Denote $P(E)=\begin{pmatrix}
\lambda(E) & \mu(E)\\
1          & 1
\end{pmatrix}$.
By Proposition \ref{ALKAM},  there exist  $Y, f_i \in \mathscr B_{1,\gamma,\mathcal I}$ such that
 $$\me^{-Y(E,\theta+\omega)}P(E)^{-1}S_{W_1,W_2,V}^{E,g}(\theta)P(E)\me^{Y(E,\theta)}=\begin{pmatrix}
 \lambda(E) e^{ f_1(E,\theta)} & 0\\
0   & \mu(E)e^{ f_2 (E,\theta)}
\end{pmatrix}.$$
 By Proposition \ref{prop:lowregu},  the equation
\begin{equation}\label{eq:onestep}
\begin{aligned}
y_{i}(E,\theta+\omega)-y_{i}(E,\theta)=f_{i}(E,\theta)-\widehat {f_{i}(E)}(0), i=1,2.
\end{aligned}
\end{equation}
has a solution $ f_i \in \mathscr B_{0,\gamma',\mathcal I}$.
Denote $$B(E,\theta):=P(E)\me^{Y(E,\theta)}\begin{pmatrix}
 \me^{ y_{1}(E,\theta)} & 0\\
0   &  \me^{ y_{2}(E,\theta)}
\end{pmatrix},$$
then the result follows.
\end{proof}

\smallskip
\textbf{Continuous function without any frequency restriction}\\

If we don't add any arithmetic assumption on the frequency $\omega$, or if we just assume $f$ to be a continuous function, then  \eqref{coho} might have no continuous solution, even no $L^2$ solution \cite{Herman2004}.  However, if $f$ is  trigonometric polynomial,  then \eqref{coho} always have a solution $y(\theta)$. Indeed, the price is the norm of $y(\theta)$ might lose uniform control.
 This motivates us to prove that if we perturb the potential $V$, then the resulting operator has a positive almost periodic eigenfunction. The key step is Lemma \ref{Schrodinger}, which says a perturbation of the (non-self-adjoint) Jacobi cocycle can be converted into a  (non-self-adjoint) Jacobi cocycle.

Define $A^{V,W}=\begin{pmatrix}
V & W\\
1 & 0
\end{pmatrix}$, and denote the function space $$\mathscr B_{0,\mathcal I}(*)= C(\mathcal I\times \T^d,*)$$
with the norm $\|f\|_{\mathcal I}:=\sup\limits_{E\in\mathcal I}\|f(E,\cdot)\|_{0}.$
It is straightforward to check that $C(\T^d)$ admits condition $(H)$.

\begin{Lemma}\label{Schrodinger}
Let $\omega\in\R^d\backslash\Q^d$ be rationally independent, and let $\tilde V,\frac{1}{\tilde V}, W, \frac{1}{W}\in \mathscr B_{0,\mathcal I}(\R)$. There exists $\epsilon>0$ such that if $A\in \mathscr B_{0,\mathcal I}(\GL)$ satisfies $\|A-A^{\tilde V,W}\|_{\mathcal I}<\epsilon$, then there exist $V'\in \mathscr B_{0,\mathcal I}(\R)$ and $Y\in \mathscr B_{0,\mathcal I}(\gl)$ satisfying
$$\me^{-Y(\theta+\omega)}A(\theta)\me^{Y(\theta)}=A^{V',W}.$$
Moreover, $\|Y\|_{0}\leq C\epsilon$, and $\|\tilde V-V'\|_{0}\leq C\epsilon$, where $C$ depends on $\tilde V,W$.
\end{Lemma}
\begin{Remark}
If  $W\equiv -1,$  i.e., $(\omega, A^{E-V,-1})$ is a Schr\"odinger cocycle, it was first proved  in analytic topology by Avila-Jitomirskaya (Lemma 2.3 in \cite{AJ}).
\end{Remark}

\begin{proof}
Denote $A_0=A^{\tilde V,W}$, and $A=A_0\me^{X_0}$, $X_0=\begin{pmatrix}
f_{1} & f_{2}\\
f_{3} & f_{4}
\end{pmatrix}\in \mathscr B_{0,\mathcal I}(\gl)$ be such that $\|X_0\|_{\mathcal I}$ small enough. Let $y=\begin{pmatrix}
y_{1} & y_{2}\\
y_{3} & -y_{1}
\end{pmatrix}\in \mathscr B_{0,\mathcal I}(\gl)$ be defined by $y_{1}=0, y_{2}=f_{2}-\frac{Wf_1}{\tilde V}, y_{3}(\theta)=-\frac{f_1(\theta-\omega)}{\tilde V(\theta-\omega)}$, and let $V_1\in \mathscr B_{0,\mathcal I}(\R)$ be given by
$$V_1(\theta)=\tilde V(\theta)+W(\theta)f_{3}(\theta)+f_{2}(\theta+\omega)-\frac{Wf_{1}(\theta+\omega)}{\tilde V(\theta+\omega)}+\tilde V(\theta)f_{1}(\theta)+\frac{W(\theta)f_{1}(\theta-\omega)}{\tilde V(\theta-\omega)}.$$
Then $\|V_1-\tilde V\|_0\leq C\|X_0\|\leq C\epsilon$ and $\me^{y(E,\theta+\omega)}A(E,\theta)\me^{-y(E,\theta)}$ is for the form $A^{V_1,W}\me^{X_1}$ where $\|X_1\|_0\leq C\|X_0\|_0^2$ for some constant $C$ depending on $\|\tilde V\|_{\mathcal I}$, $\|\frac{1}{\tilde V}\|_{\mathcal I}, \|W\|_{\mathcal I}, \|\frac{1}{W}\|_{\mathcal I}$. Then the result follows by iteration.
\end{proof}

As a consequence, we have the following:

\begin{Lemma}\label{prop:withoutfre}
Let $\omega\in\R^d\backslash\Q^d$ be rationally independent, $d\in\N_+,$ $g>0$, and suppose that $W_1,W_2, V\in C(\T^d,\R)$. Then there exists an absolute constant  $C_0>0$
such that if
$$\|W_1\|_{0}+\|W_2\|_0+\|V\|_0\leq\epsilon\leq C_0\me^{-12g}g^6,$$
then for any $\varepsilon>0$,
there exist $B\in\mathscr B_{0,\mathcal I}(\gl)$, $V'\in \mathscr B_{0,\mathcal I}(\R)$ with $\|V'\|_{0,\mathcal I}\leq  \varepsilon$
 satisfying
\begin{equation}\label{conj}B(E,\theta+\omega)^{-1}S_{W_1,W_2,V'}^{E,g}B(E,\theta)=\begin{pmatrix}
 \lambda(E)e^{ \widehat {f_{1}(E)}(0) }& 0\\
0   & \mu(E) e^{\widehat {f_{2}(E)}(0)}
\end{pmatrix},
\end{equation} where $B$ takes the form
$$B(E,\theta):=\begin{pmatrix}
\lambda(E) & \mu(E)\\
1          & 1
\end{pmatrix}\me^{Y(E,\theta)}\begin{pmatrix}
 \me^{ y_{1}(E,\theta)} & 0\\
0   &  \me^{ y_{2}(E,\theta)}
\end{pmatrix}$$
 with the estimate
$$\|Y\|_{\mathcal I}\leq 4\epsilon^{\frac{1}{3}}\leq \frac{\lambda(E)}{4(\lambda(E)+\mu(E))}.$$
\end{Lemma}

\begin{proof}
We only need to consider the case $\varepsilon<\epsilon$. Denote $P(E)=\begin{pmatrix}
\lambda(E) & \mu(E)\\
1          & 1
\end{pmatrix}$. By Proposition \ref{ALKAM},  there exist  $Y', f_i \in \mathscr B_{0,\mathcal I}, i=1,2$ with $\|Y'\|_{\mathcal I}\leq \epsilon^{\frac{1}{3}}, \|f_i\|_{\mathcal I}\leq \epsilon^{\frac{1}{2}}$ such that  $B'(E,\theta):=P(E)\me^{Y'(E,\theta)}
$ satisfying
$$\begin{aligned}
B'(E,\theta+\omega)^{-1}S^{E,g}_{W_1,W_2,V}(\theta)B'(E,\theta)&=\begin{pmatrix}
 \lambda(E) e^{ f_1(E,\theta)} & 0\\
0   & \mu(E)e^{ f_2 (E,\theta)}
\end{pmatrix}.
\end{aligned}$$

For any $K>0$, we  define the truncating operator $\mathcal{T}_{K}$ as
$$\mathcal{T}_{K}F(E,\theta)=\sum\limits_{k\in\mathbb{Z}^{d},|k|<K}\widehat{F(E)}(k)\me^{\mi\langle k,\theta\rangle}.$$
 Let
$$A_K(E,\theta):=B'(E,\theta+\omega)\begin{pmatrix}
 \lambda(E) e^{ \mathcal{T}_{K} f_1(E,\theta)} & 0\\
0   & \mu(E)e^{ \mathcal{T}_{K} f_2 (E,\theta)}
\end{pmatrix}B'(E,\theta)^{-1}.$$ Then for any $\epsilon>0$, there exists sufficiently large $K=K(\varepsilon)>0$ such that
\begin{eqnarray*}
&& \|S^{E,g}_{W_1,W_2,V}-A_K\|_{\mathcal I}\leq \|B'\|_{\mathcal{I}}\|(B')^{-1}\|_{\mathcal{I}} \|diag(
\lambda(\me^{f_{1}}-\me^{\mathcal T_Kf_{1}}),\mu(\me^{f_{2}}-\me^{\mathcal T_Kf_{2}})
)\|_{\mathcal I}\leq C'\varepsilon.
\end{eqnarray*}
where $C'$ only depends on $g$.

On the other hand, just note  $A^{\frac{E-V}{\me^{-g}+W_1},-\frac{\me^{g}+W_2}{\me^{-g}+W_1}}=S_{W_1,W_2,V}^{E,g}$, then one can verify that $\frac{E-V}{\me^{-g}+W_1}, \frac{\me^{-g}+W_1}{E-V}, -\frac{\me^{g}+W_2}{\me^{-g}+W_1}, -\frac{\me^{-g}+W_1}{\me^{g}+W_2}\in \mathscr B_{0,\mathcal I}(\R)$. As consequence of
 Lemma \ref{Schrodinger} with $W= -\frac{\me^{g}+W_2}{\me^{-g}+W_1}$, there exist $Y''\in \mathscr B_{0,\mathcal I}(\gl), \tilde V\in \mathscr B_{0,\mathcal I}(\R)$
such that
$$\me^{-Y''(E,\theta+\omega)}A_K(E,\theta)\me^{Y''(E,\theta)}=A^{\tilde V(E,\cdot),-\frac{\me^{g}+W_2}{\me^{-g}+W_1}},$$
with the estimates
$$\|Y''\|_{\mathcal{I}}\leq C''\varepsilon, \ \biggr\|\tilde V-\frac{E-V}{\me^{-g}+W_1}\biggr\|_{\mathcal{I}}\leq C''\varepsilon,$$
where $C''$ only depends on $g$.
Let $V'(E,\theta)=E-\tilde V(E,\theta)(\me^{-g}+W_1(\theta))$. We can check that
$$\|V-V'\|_{\mathcal{I}}\leq \|\me^{-g}+W_1\|_{\mathcal{I}}\biggr\|\tilde V-\frac{E-V}{\me^{-g}+W_1}\biggr\|_{\mathcal{I}}\leq 2C''\me^{-g}\varepsilon.$$

Let $\tilde B(E,\theta):=\me^{-Y''(E,\theta)}B'(E,\theta)$. Note that $A^{\tilde V(E,\cdot),-\frac{\me^{g}+W_2}{\me^{-g}+W_1}}=S^{E,g}_{W_{1},W_{2},V'(E,\cdot)},$ and then we have
\begin{eqnarray*}
 \tilde B(E,\theta+\omega)^{-1}S^{E,g}_{W_{1},W_{2},V'(E,\cdot)}\tilde B(E,\theta)
&=&B'(E,\theta+\omega)^{-1}A_{K}(E,\theta)B'(E,\theta)\\
&=&
\begin{pmatrix}
 \lambda(E) e^{ \mathcal{T}_{K} f_1(E,\theta)} & 0\\
0   & \mu(E)e^{ \mathcal{T}_{K} f_2 (E,\theta)}
\end{pmatrix}.
\end{eqnarray*}
Also note for any $\omega\in\R^d\backslash\Q^d$,
\begin{equation*}
\begin{aligned}
y_{i}(E,\theta+\omega)-y_{i}(E,\theta)=\mathcal T_Kf_{i}(E,\theta)-\widehat {f_{i}(E)}(0), i=1,2
\end{aligned}
\end{equation*}
always have a continuous solution.
Let $
B(E,\theta)
=\tilde B(E,\theta) diag(\me^{y_{1}(E)},\me^{y_{2}(E)}).$ It is clear that $B$ satisfies \eqref{conj}.

 Finally we show that $B$ has desired form. To prove this,
 note  if $B, D$ are small $\gl$ matrices,
then there exists $H\in \gl$ such that
\begin{equation*}
\me^{B}\me^{D}=\me^{B+D+H},
\end{equation*}
where $H$ is a sum of terms at least 2 orders in $B,D.$ Thus
  there exists $Y\in \mathscr B_{0,\mathcal I}(\gl)$ such that
 $$\me^{Y(E,\theta)}=\me^{-P(E)^{-1}Y''(E,\theta)P(E)}\me^{Y'(E,\theta)}$$ with the estimate
 $$
 \begin{aligned}
 \|Y(E,\theta)\|_{\mathcal{I}} 
 &\leq 2\|P(E)Y''(E,\theta)P(E)^{-1}\|_{\mathcal{I}}+2\|Y'(E)\|_{\mathcal{I}}\leq 4\epsilon^{\frac{1}{3}}.
 \end{aligned}$$
Consequently, one has
\begin{eqnarray*}
B(E,\theta)
&=&\me^{-Y''(E,\theta)}P(E)\me^{Y'(E,\theta)} diag(\me^{y_{1}(E)},\me^{y_{2}(E)}) \\
&=&P(E)\me^{Y(E,\theta)} diag(\me^{y_{1}(E)},\me^{y_{2}(E)}),
\end{eqnarray*} and then we finish the proof.
\end{proof}

\smallskip

\textbf{Analytic almost periodic function.}\\

Finally we consider the analytic almost periodic case. First let's define the almost periodic functions in the context of analytic functions on a thickened infinite dimensional torus $\mathbb T_r^\infty$, where $\mathbb T_r^\infty$ is defined as
$$\theta=(\theta_j)_{j\in\mathbb N},\ \theta_j\in \mathbb C:\Re(\theta_j)\in \mathbb T,|\Im(\theta_j)|\leq r\langle j\rangle.$$
For any $r>0$, we define the space of analytic functions $\mathbb T^\infty_r\rightarrow *$ as
\begin{equation}\label{almostinfinite}
\begin{aligned}
C^\omega_r&(\mathbb T^\infty,*):=\\
&\left\{F(\theta)=\sum\limits_{k\in\mathbb Z_*^\infty}\hat F(k)\mathrm e^{\mathrm i\langle k,\theta\rangle}\in\mathcal F:|F|_r:=\sum\limits_{k\in\mathbb Z_*^\infty}\mathrm e^{r|k|_1}|\hat F(k)|<\infty\right\}
\end{aligned}
\end{equation}
 where $*$ usually denotes $\mathbb R,\gl$ or $\GL$,  and $\mathcal F$ denotes the space of pointwise absolutely convergent formal Fourier series $\mathbb T_r^\infty\rightarrow *$ as
$$F(\theta)=\sum\limits_{k\in\mathbb Z_*^\infty}\hat F(k)\mathrm e^{\mathrm i\langle k,\theta\rangle},\quad \hat F(k):=\int_{\T^\infty} F(\theta)\me^{-\mi\langle k,\theta\rangle}d\theta,$$
 and $\mathbb Z_*^\infty:=\{k\in\mathbb Z^\infty:|k|_1:=\sum\limits_{j\in\mathbb N}\langle j\rangle|k_j|<\infty\}$
denotes  the set of infinite integer vectors with finite support.

  In this case, we can solve the cohomological equation as follows:

\begin{Lemma}\label{al-small}
Let $\gamma\in (0,1),\tau>1, r>0$,  $f\in C^\omega_r(\mathbb T^\infty,\R)$. Then for any $\omega\in DC_\infty(\gamma,\tau)$ and any $0<r'<r$, the equation
\eqref{coho} possesses a solution $y\in C^\omega_{r'}(\mathbb T^\infty,\R)$.
\end{Lemma}

\begin{proof}First we recall  the following estimate:
\begin{Lemma}[\cite{Riccardo2021}]\label{Prop:smalldivsorestimate}
Let $\tau>0$. Then
$$\sup\limits_{k\in\Z_*^\infty,|k|_1<\infty}\prod\limits_{i\in\Z}(1+\langle i\rangle^\tau|k_i|^\tau)\me^{-r|k|_1}\leq \exp\bigr(\frac{\mu}{r}\ln(\frac{\mu}{r})\bigr)$$
for some constant $\mu=\mu(\tau).$
\end{Lemma}
Then as a consequence, if $\omega \in \mathrm{DC_{\infty}(\gamma,\tau)}$,
 for  any $0< r'<r$, we have
$$
\begin{aligned}
\|y_{i}\|_{r'}&\leq \sum\limits_{k\in\Z_*^\infty\backslash\{0\}}\frac{|\widehat {f_{i}(E)}(k)|}{|\me^{\mi\langle k,\omega\rangle}-1|}\me^{|k|_1r}\me^{|k|_1(r'-r)}\\
&\leq \sum\limits_{k\in\Z^\infty_*\backslash\{0\}}\frac{|\widehat{ f_{i}(E)}(k)|\me^{r|k|_1}}{\gamma}\prod\limits_{j\in\mathbb N}({1+|k_j|^\tau\langle j\rangle^\tau})\me^{|k_1|(r'-r)}\\
&\leq \frac{2\epsilon}{\gamma}\sup\limits_{k\in\Z^\infty_*\backslash\{0\}}\prod\limits_{j\in\mathbb N}({1+|k_j|^\tau\langle j\rangle^\tau})\me^{|k|_1(r'-r)}\\
&\leq \frac{2\epsilon}{\gamma}\exp\bigr(\frac{\mu}{r-r'}\ln(\frac{\mu}{r-r'})\bigr).
\end{aligned}
$$

\end{proof}

For any $r>0$,  denote $\mathcal B_r:=C^\omega_r(\T^\infty,\gl)$, and the function space $$\mathscr B_{r,\mathcal I}:=\{f\in C(\mathcal I\times \T^\infty,\gl)| f(E,\cdot)\in \mathcal B_r \text{ for any given } E\in \mathcal I\}$$
with the norm $\|f\|_{r,\mathcal I}:=\sup\limits_{E\in\mathcal I}|f(E)|_r.$ It is clear that $\mathcal B_r$ admits condition $(H).$

 \begin{Lemma}\label{prop:ALredu}
Let $g,r>0$, $\omega\in DC_{\infty}(\gamma,\tau)$, $\gamma\in (0,1),\tau>1$. Suppose that $W_1,W_2,V\in C^\omega_r(\T^\infty,\R)$. Then there exists  an absolute constant $C_0>0$ such that if
$$|W_1|_{r}+|W_2|_r+|V|_r\leq\epsilon\leq C_0\me^{-12g}g^6,$$
then for any $0<\tilde r<r,$  there exists $B\in\mathscr B_{\tilde r,\mathcal I}$ such that
$$B(E,\theta+\omega)^{-1}S_{W_1,W_2,V}^{E,g}B(E,\theta)=\begin{pmatrix}
 \lambda(E)e^{ \widehat {f_{1}(E)}(0) }& 0\\
0   & \mu(E) e^{\widehat {f_{2}(E)}(0)}
\end{pmatrix},
$$ where $B$ takes the form
$$B(E,\theta):=\begin{pmatrix}
\lambda(E) & \mu(E)\\
1          & 1
\end{pmatrix}\me^{Y(E,\theta)}\begin{pmatrix}
 \me^{ f_{1}(E,\theta)} & 0\\
0   &  \me^{ f_{2}(E,\theta)}
\end{pmatrix}$$
 with  estimates
$$\|Y\|_{\tilde r,\mathcal I}\leq \epsilon^{\frac{1}{3}}\leq \frac{\lambda(E)}{4(\lambda(E)+\mu(E))}.$$
\end{Lemma}

\begin{proof}
The proof is same as Lemma \ref{prop:finitediffredu}, one only needs to replace Proposition \ref{prop:lowregu} by Lemma \ref{al-small}.
\end{proof}

\subsubsection{Continuity argument}
Motivated by Proposition \ref{criterion} and Lemma \ref{prop:finitediffredu}, it suffices to find some $E_0\in\mathcal I$ such that
$ \lambda(E_0)e^{ \widehat {f_{1}(E_0)}(0) } =1,$ which is the content of the following lemma:

\begin{Lemma}\label{contin}
Under the assumption of Proposition \ref{ALKAM}, there exists $E_0\in \mathcal I$, such that
$$ \lambda(E_0)e^{ \widehat {f_{1}(E_0)}(0) } =1.$$
\end{Lemma}

\begin{proof}
By \eqref{esf}, for any $E\in \mathcal I$, we have
 \begin{equation}\label{diff}
 |\lambda(E) e^{ \widehat {f_{1}(E)}(0) }-\lambda(E)|\leq 2|\lambda(E)|\epsilon^{\frac{1}{2}}.
 \end{equation}

On the other hand,  for $E_1=2+\min\{1,\frac{g^2}{9}\}$, by \eqref{characeq}, the minimum eigenvalue of $A_{g}(E_1)$ is $\lambda(E_1)=\me^{g}\frac{E_1-\sqrt{E_1^2-4}}{2}>1.$
And for $E_2=2\me^g+2\me^{-g}$,
  the minimum eigenvalue  of $A_{g}(E)$ is $\lambda(E_2)=\me^{g}\frac{E_2-\sqrt{E_2^2-4}}{2}<1$.

By \eqref{diff}, one has
 $$|\lambda(E_i)-\lambda(E_i) e^{ \widehat {f_{1}(E_i)}(0) }|\leq 2\epsilon^{\frac{1}{2}}|\lambda(E_i)|\leq \frac{\min\{g^2,1\}}{18}, \ i=1,2.$$
which implies that  $$
\begin{aligned}
\lambda(E_1) e^{ \widehat {f_{1}(E_1)}(0) }&\geq\lambda(E_1)-\lambda(E_1)|1-e^{ \widehat {f_{1}(E_1)}(0) }|
\geq \me^{g}\frac{E_1-\sqrt{E_1^2-4}}{2}-\frac{\min\{g^2,1\}}{18}>1,\\
\lambda(E_2) e^{ \widehat {f_{1}(E_2)}(0) }&\leq \lambda(E_2)+\lambda(E_2)|1-e^{ \widehat {f_{1}(E_2)}(0) }|\leq \me^{g}\frac{E_2-\sqrt{E_2^2-4}}{2}+\frac{\min\{g^2,1\}}{18}<1.
\end{aligned}
$$

By Proposition \ref{ALKAM}, $f_1(E,\cdot)$ is continuous in $E$, which implies that $\lambda(E) e^{ \widehat {f_{1}(E)}(0) }$ in continuous in $E$, and then the result follows by the intermediate value theorem.
\end{proof}

\subsection{Positive almost-periodic function}

As a consequence, we get the positive almost-periodic eigenfunction of the non-self-adjoint Jacobi operator $\mathcal L^g_{W_1, W_2, V}$.

\begin{Corollary}\label{ap-ir}
Let $\omega\in\mathcal P, g\neq 0$, and suppose that $W_1,W_2,V\in C^{1,\gamma}(\T,\R)$ for some $ 0<\gamma<1.$ Then there
exists  an absolute constant $C_0>0$ such that if
$$\|W_1\|_{1,\gamma}+\|W_2\|_{1,\gamma}+\|V\|_{1,\gamma}\leq C_0\me^{-12|g|}g^6,$$
then there exist $E_0\in\R, U\in C(\T,\R)$, with $U>0$ such that $u(n)=U(n\omega)$ is a quasi-periodic solution of
$$\bigr(\me^{-g}+W_1(n\omega)\bigr)u(n+1)+\bigr(\me^{g}+W_2(n\omega)\bigr)u(n-1)+V(n\omega)u(n)=E_0u(n).$$
\end{Corollary}

\begin{proof}
It suffices to consider the case that $g>0$, since one can transform the operator $$(\mathcal L_{W_1,W_2,V}^gp)(n):=\bigr(\me^{-g}+W_1(n\omega)\bigr)p
(n+1)+\bigr(\me^{g}+W_2(n\omega)\bigr)p(n-1)+V(n\omega)p(n)$$ by
$\mathcal S: \delta_n\to\delta_{-n}$ into $(\mathcal S\mathcal L_{W_1,W_2,V}^g\mathcal S^{-1}p)(n)=$
 $$ (\me^{g}+W_2(-n\omega))p(n+1)+(\me^{-g}+W_1(-n\omega))p(n-1)+V(-n\omega)p(n).$$

In this case, by Lemma \ref{prop:finitediffredu}, there exists a conjugation $B$ such that
$$B(E,\theta+\omega)^{-1}S_{W_1,W_2,V}^{E,g}(\theta)B(E,\theta)=\begin{pmatrix}
 \lambda(E)e^{ \widehat {f_{1}(E)}(0) }& 0\\
0   & \mu(E) e^{\widehat {f_{2}(E)}(0)}
\end{pmatrix}.
$$
Moreover, $B$ takes the form as
$$B(E,\theta):=P(E)\me^{Y(E,\theta)} diag(\me^{y_{1}(E,\theta)},\me^{y_{2}(E,\theta)}).$$
By Lemma \ref{contin}, there exists $E_0\in \mathcal I$, such that
$$ \lambda(E_0)e^{ \widehat {f_{1}(E_0)}(0) } =1 .$$
Thus for this $E_0$, one has
$$B(E_0,\theta+\omega)^{-1}S_{W_1,W_2,V}^{E_0,g}(\theta)B(E_0,\theta)=\begin{pmatrix}
1 & 0\\
0   & \mu'
\end{pmatrix},
$$
where $\mu' \in \R$ is a constant. Moreover, by Lemma \ref{prop:finitediffredu}, one has
$$\|Y(E_0,\cdot) \|_{0}\leq \epsilon^{\frac{1}{3}}\leq \frac{\lambda(E)}{4\lambda(E)+4\mu(E)}.$$
Thus one can apply the criterion (Proposition \ref{criterion}) to get the desired result.
\end{proof}

\begin{Corollary}\label{dense-ir}
Let $\omega\in\R^d\backslash\Q^d$, $d\in\N_+,$ $g\neq 0$, and suppose that $W_1, W_2, V\in C(\T^d,\R)$. Then there exists  an absolute constant $C_0>0$
such that if
$$\|V\|_0+\|W_1\|_0+\|W_2\|_0\leq C_0\me^{-12|g|}g^6,$$
then for any $\varepsilon>0$,  there exist $V'\in C(\T^d,\R)$ with $\|V'-V\|_{0}\leq \varepsilon$, and $E_0\in\R, U\in C(\T^d,\R)$ with $U>0$ such that $u(n)=U(n\omega)$ is a quasi-periodic solution of
$$\bigr(\me^{-g}+W_1(n\omega)\bigr)u(n+1)+\bigr(\me^{g}+W_2(n\omega)\bigr)u(n-1)+V'(n\omega)u(n)=E_0u(n).$$

\end{Corollary}

\begin{proof}
The proof is the same as Corollary \ref{ap-ir}, if one only needs to replace Lemma \ref{prop:finitediffredu} by Lemma \ref{prop:withoutfre}.
\end{proof}

\begin{Corollary}\label{Theorem:positive bloch wave}
 Let $\omega\in DC_\infty(\gamma,\tau)$, $0<\gamma<1,\tau>1, g\neq 0, r>0$, and suppose that $W_1,W_2,V\in C^\omega_r(\mathbb T^\infty,\mathbb R)$. Then there exists an absolute constant $C_0>0$ such that if $$|W_1|_r+|W_2|_r+|V|_r\leq C_0\me^{-12|g|}g^6,$$ then there exists $E_0\in\mathbb R, U\in C^\omega_{r'}(\T^\infty,\R)$ for any $0<r'<r$ with $U>0$ such that $u(n)=U(n\omega)$ is an almost periodic solution of
$$\bigr(\me^{-g}+W_1(n\omega)\bigr)p
(n+1)+\bigr(\me^{g}+W_2(n\omega)\bigr)p(n-1)+V(n\omega)p(n)=E_0p(n).$$
\end{Corollary}

\begin{proof}
The proof is the same as Corollary \ref{ap-ir}, if one only needs to replace Lemma \ref{prop:finitediffredu} by Lemma \ref{prop:ALredu}.
\end{proof}

\section{Ground state of the non-self-adjoint Jacobi operator}
\label{section4}

Define the ground state energy set of a  linear bounded operator $\mathcal L$
\begin{equation}
E_{\max}(\mathcal L):=\{E\in\Sigma(\mathcal L)| \Re E'\leq\Re E, \text{ for any } E'\in\Sigma(\mathcal L)\},
 \end{equation}
 where $\Sigma(\mathcal L)$ denotes the spectrum of $\mathcal L$.
  Since $\mathcal L$ is a linear bounded operator, it is well known that $\Sigma(\mathcal L)$ is a compact, non-empty set, and for any $E\in \Sigma(\mathcal L)$, $|E|\leq \|\mathcal L\|$, where $\|\cdot\|$ denotes the operator norm. Hence $E_{\max}(\mathcal L)$ is not an empty set.

Last section, we have constructed a positive almost-periodic eigenfunction of $\mathcal L_{W_1,W_2,V}^{g}$ with energy $E_0$. In this section, we will prove actually $E_0\in E_{\max}(\mathcal L_{W_1,W_2,V}^{g})$,   and  further show it is simple. It is direct to see the Jacobi operator  $\mathcal L_{W_1,W_2,V}^{g}$ can be rewritten as
the discrete elliptic operator:
\begin{equation}\label{eq:ellipi section4}
\mathcal Lu=D^*\bigr(a(n)Du\bigr)+b(n)D^*u+c(n)u.
\end{equation}
where
$  \inf\limits_{n\in\Z} a(n)>0,\inf\limits_{n\in\Z} (a(n-1)-b(n))>0. $
As we have explained in the introduction,  the results for \eqref{eq:ellipi section4} will be given first for applications. We should notice the results for \eqref{eq:ellipi section4} still hold for $\mathcal L_{W_1, W_2, V}^{g}$.

\subsection{Ground state of the discrete elliptic operator}

We introduce the adjoint operator of $\mathcal L$ in $\ell^2(\Z)$:
$$
\mathcal L^*u=D^*(a(n)Du)-D(b(n)u)+c(n)u.
$$
Now we will show if there exist $E_0,E_0'\in\R$ and two almost-periodic functions $p, p_*$ with $\inf p,\inf p_*>0$  satisfying
\begin{equation}\label{eq:basassumpground}
\mathcal Lp=E_0p, \ \mathcal L^*p_*=E'_0p_*,
\end{equation}
then $E_0= E_0' \in E_{\max}(\mathcal L)$, and thus $E_0$ is the ground state energy.

The first observation is the following:
\begin{Lemma}\label{unique}
Let $p$ and $p_*$ be two bounded solutions of $\mathcal Lp=E_0p$ and $\mathcal L^*p_*=E'_0p_*$. If $\inf p,\inf p_*>0$, then $E_0=E_0'$.
\end{Lemma}

\begin{proof}
Let $\mathcal \chi_k$ be the characteristic function such that $$\chi_k(n)=\left\{\begin{aligned}
&1 &\text{if }|n|\leq k\\
&0   &\text{otherwise}.
\end{aligned}
\right.$$ Denote $\tilde b(n)=a(n-1)-b(n)$. Since for any $k\in\Z, \chi_kp,\chi_kp_*\in \ell^2(\Z)$, one has
\begin{equation}\label{eq:E_0=E_0'1}
\langle \chi_k\mathcal L(\chi_kp),\chi_kp_*\rangle= \sum\limits_{n=-k}^{k}E_0(pp_*)(n)-(ap_*)(k)p(k+1)-(\tilde bp_*)(-k)p(-k-1),
\end{equation}
and similarly,
\begin{equation}\label{eq:E_0=E_0}
\langle \chi_k p,\chi_k\mathcal L^*(\chi_kp_*)\rangle=\sum\limits_{n=-k}^{k}E'_0(pp_*)(n)-(ap_*)(-k-1)p(-k)-(\tilde bp_*)(k+1)p(k).
\end{equation}
Note that $\langle\chi_k\mathcal L(\chi_kp),\chi_kp_*\rangle=\langle \chi_kp,\chi_k\mathcal L^*(\chi_kp_*)\rangle$. Subtracting \eqref{eq:E_0=E_0} from \eqref{eq:E_0=E_0'1}, one has
$$
\begin{aligned}
\sum\limits_{n=-k}^{k}(E_0-E_0')(pp_*)(n)=&(ap_*)(k)p(k+1)+(\tilde bp_*)(-k)p(-k-1)\\
&\quad\quad-(\tilde bp_*)(k+1)p(k)
-(ap_*)(-k-1)p(-k).
\end{aligned}$$
Since $p,p_*$ are two bounded solutions with positive infimums, the right-hand side is bounded. However, the left-hand side tends to infinite if $E_0\neq
E_0'.$  Then the result follows.
\end{proof}

This lemma also shows such $E_0$, with the property that the corresponding eigenfunction is positive almost periodic, is unique.
Once we have this, we shall show a crucial observation that is needed in the rest of the paper.
\begin{Lemma}\label{reduc}
Let $p,p_*\in\ell^\infty(\Z)$ be two bounded solutions of $\mathcal Lp=E_0p,$ $\mathcal L^*p_*=E_0p_*$. If $\inf p,\inf p_*>0,$ then we have
$$p_*(\mathcal L-E_0)(pu)=D^*\bigr(\tilde a(n)Du\bigr)+\bar c(D^*u+Du),$$
where  $\bar c\equiv \const$ and $\inf\tilde a>|\bar c|$.
\end{Lemma}

\begin{proof}
Denote $c'=c-E_0$.
For any $u\in \ell^\infty(\Z),$ since $\mathcal Lp=E_0p$, one has
\begin{equation*}
\begin{aligned}
\tilde{\mathcal L}u&:=p_*(\mathcal L-E_0)(pu)\\
&=p_*(n)D^*\biggr(a(n)D\bigr(p(n)u\bigr)\biggr)+(bp_*)(n)D^*\bigr(p(n)u\bigr)+(p_*c'p)(n)u\\
&=p_*(n)D^*\biggr(a(n)D\bigr(p(n)u\bigr)\biggr)+(bp_*)(n)D^*\bigr(p(n)u\bigr)\\
&\quad\quad\qquad\qquad\qquad-p_*(n)\biggr[\biggr(D^*\bigr(a(n)Dp(n)\bigr)\biggr)+b(n)D^*p(n)\biggr]u\\
&=D^*\bigr(a(n)p_*(n)p(n+1)Du\bigr)+\tilde c(n)D^*u=(I),
\end{aligned}
\end{equation*}
where $\tilde c(n)=(ap_*)(n-1)p(n)-(ap)(n-1)p_*(n)+(bp_*)(n)p(n-1)$.

Meanwhile, let $\tilde b(n)=a(n-1)-b(n)$. One can rewrite that
$$
(\mathcal L-E_0)u=D^*(a(n)Du)+b(n)D^*u+c'(n)u=D(\tilde b(n)D^*u)+b(n+1)Du+c'(n)u.
$$
Then we have,
$$
\begin{aligned}
(\tilde{\mathcal L}u)=&p_*(n)D\biggr(\tilde b(n)D^*\bigr(p(n)u\bigr)\biggr)+b(n+1)p_*(n)D\bigr(p(n)u\bigr)+(c'pp_*)(n)u\\
=&p_*(n)D\biggr(\tilde b(n)D^*\bigr(p(n)u\bigr)\biggr)+b(n+1)p_*(n)D\bigr(p(n)u\bigr)\\
&\quad\quad\qquad\qquad\qquad -p_*(n)\biggr[D\bigr(\tilde b(n)D^*p(n)\bigr)+b(n+1)Dp(n)\biggr]u\\
=&D\bigr(\tilde b(n)p_*(n)p(n-1)D^*u\bigr)+\tilde c(n+1)Du=(II).
\end{aligned}
$$

It is straightforward to check from $\mathcal Lpp_*=\mathcal L^*p_*p=E_0pp_*$ that $\tilde c(n)=\tilde c(n+1).$
Hence \begin{equation}
\tilde c(n)=a(n-1)p_*(n-1)p(n)-\tilde b(n)p(n-1)p_*(n)\equiv const:=2\bar c.
\end{equation}
Adding $(I)$ and $(II)$,  we have
$$
\begin{aligned}
2\tilde {\mathcal L}u&=D\bigr(\tilde b(n)p(n-1)p_*(n)D^*u\bigr)+2\bar cDu+D^*\bigr(a(n)p_*(n)p(n+1)Du\bigr)+2\bar cD^*u,\\
&=2D^*\bigr(\tilde a(n)Du\bigr)+2\bar c\bigr(Du+D^*u\bigr),
\end{aligned}
$$
where $\tilde a(n)=\frac{1}{2}\bigr[(ap)(n)p_*(n+1)-(bp_*)(n+1)p(n)+(ap_*)(n)p(n+1)\bigr].$ One can check that $\inf\tilde a>0$.
Then we finish the proof.
\end{proof}

\begin{Proposition}\label{Th:location of groundstate}
 Let $p,p_*\in\ell^\infty(\Z)$ be two bounded solutions of $\mathcal Lp=E_0p,$ $\mathcal L^*p_*=E_0p_*$. If $\inf p,\inf p_*>0,$ then $E_0\in E_{\max}$.
\end{Proposition}

\begin{proof}
For $\Re(E-E_0)>0$, consider the equation
\begin{equation}\label{eq:resolventeq}
(E_0-\mathcal L)v=f, \ f\in \ell^2(\Z).
\end{equation}
Since $\mathcal Lp=E_0p, \mathcal L^*p_*=E_0p_*$, it follows from Lemma \ref{reduc} that one can transform $\mathcal L-E_0$ into
$$p_*(n)(E_0-\mathcal L)\bigr(p(n)v\bigr)=-D^*(\tilde a(n)Dv)-\bar c(Dv+D^*v),$$
where $\inf\tilde a>0$, $\bar c\in\R$.
 Denote that $v=u/p\in\ell^2(\Z)$. For $\Re(E-E_0)>0$, equation \eqref{eq:resolventeq} becomes
\begin{equation}\label{resolventestimate}
-D^*(\tilde a(n)Dv)-\bar c\bigr(D^*v+Dv\bigr)+(E-E_0)(pp_*)(n)v=(p_*f)(n).
\end{equation}
 Since $pp_*\geq \inf p\inf p_*:=\delta$,  multiplying both sides of \eqref{resolventestimate} by $v$, and taking their sums, one could notice from its real part that for any $\Re(E-E_0)>0$,
$$
\begin{aligned}
\delta\Re(E-E_0)\|v\|_{\ell^2(\Z)}^2&\leq \sum\limits_{n\in\Z}\tilde a(n)|Dv(n)|^2+\Re(E-E_0)\langle pp_*v,v\rangle\\
&\leq \frac{2\|p_*f\|^2_{\ell^2(\Z)}}{\delta\Re(E-E_0)}+\frac{\delta\Re(E-E_0)}{2}\|v\|^2_{\ell^2(\Z)}\\
&\leq \frac{C_1\|f\|^2_{\ell^2(\Z)}}{\Re(E-E_0)}+\frac{\delta\Re(E-E_0)}{2}\|v\|^2_{\ell^2(\Z)},
\end{aligned}
$$
where $C_1$ depends only on $p,p_*$.
 Hence $$\|u\|_{\ell^2(\Z)}\leq C_2\|v\|_{\ell^2(\Z)}\leq C'\frac{\|f\|_{\ell^2(\Z)}}{\Re(E-E_0)},$$
where $C_2$ only depends only on $p,p_*$. Then \eqref{resolventestimate} can be solved by Lax-Milgram theorem, and this yields that $\mathcal L-E$ is invertible for $\Re E>\Re E_0$. Thus $E$ lies in the resolvent set of $\mathcal L$ if $\Re E>\Re E_0.$

Finally, we will check that $E_0$ belongs to the spectrum of $\mathcal L$. In fact, for any $v\in\ell^2(\Z)$,  $$\sum\limits_{n\in\Z}p_*(n)(\mathcal L-E_0)(pv)(n)=\sum\limits_{n\in\Z}D^*\bigr(\tilde a(n)Dv(n)\bigr)+\bar c\sum\limits_{n\in\Z}\bigr(D^*v(n)+Dv(n)\bigr)=0,$$ and this yields that the image of $\mathcal L-E_0$ is contained in the space of
$$\biggr\{u\in \ell^2(\mathbb Z)|\sum\limits_{n\in\Z}u(n)p_*(n)=0\biggr\}\neq\ell^2(\Z).$$ Thus the result follows since $\mathcal L-E_0$ is not surjective.
\end{proof}
In the following, we notice the following
fact that positive almost periodic solution $p$ of $\mathcal Lp=E_0p$ is unique (up to linear dependence) in some sense, and this is an important property of the ground state.

\begin{Lemma}\label{Lemma:simple}
Let $p_1, p_2$ be two almost periodic functions  that satisfy $(\mathcal L-E_0)p_1=(\mathcal L-E_0)p_2=0$. If  $\inf p_1>0$, then we have $p_1=Cp_2$ for some constant $C.$
\end{Lemma}

\begin{proof}
 Assume by contradiction that $p_1\neq Cp_2$ for any $C\in\R$. Without loss of generality, we assume that $p_2(n_0)>0$ at some point $n_0$ (Otherwise, we consider $-p_2$ instead).
 Define the function $F(\lambda)=\inf\limits_{n\in\Z}(p_1-\lambda p_2)(n).$ It is clear that $F$ is continuous, since $(p_1-\lambda p_2)(n)$ is continuous in $\lambda\in\R$, uniformly with respect to $n\in\Z$.
  Note since $\inf\limits_{n\in\Z} p_1>0$
  for $\lambda<\frac{\inf p_1}{\sup p_2}$, there holds $F(\lambda)>0.$ For $\lambda$ sufficiently large, there of course holds $F(\lambda)<0$. Hence by the intermediate value theorem , there exists $\lambda_0>0$ such that $F(\lambda_0)=\inf\limits_{n\in\Z} (p_1-\lambda_0 p_2)(n)=0$.

  Let $n_i\in\Z$ be such that $(p_1-\lambda_0 p_2)(n_i)\to 0$ as $i\to\infty.$ By the almost periodicity of $a, b,c,$ the limits $a_*(n)=\lim\limits_{i\to\infty}a(n+n_i),\ b_*(n)=\lim\limits_{i\to\infty}b(n+n_i), c_*(n)=\lim\limits_{i\to\infty}c(n+n_i)$ exist after we pass along a subsequence, and we still denote the subsequence $\{n_i\}$.

 Consider the limit operator $$\mathcal L_*u=D^*(a_*(n)Du)+b_*(n)D^*u+c_*(n)u.$$
 We can check that $w_*(n):=\lim\limits_{i\to\infty}(p_1-\lambda p_2)(n+n_i)\geq 0$ is also the eigenfunction of $\mathcal L_*$ associated with $E_0$.

 Now we claim that $w_*=0.$ From the choice of ${n_i},$ we have $w_*(0)=0.$  Note that $w_*(\pm 1)=0$ immediately implies that $w_*\equiv 0$, hence we only consider the case that $w_*(\pm 1)>0$.
  Since $$\inf\limits_{n\in\Z} a_*(n)\geq \inf\limits_{n\in\Z} a(n)>0, \inf\limits_{n\in\Z} \bigr(a_*(n)-b_*(n-1)\bigr)\geq \inf\limits_{n\in\Z} \bigr(a(n)-b(n-1)\bigr)>0,$$
  it follows that
 $$(\mathcal L_*w_*)(0)=a_*(0)w_*(1)+(a_*(0)-b_*(-1))w_*(-1)>0,$$  which contradicts the fact that $\mathcal L_*w_*(0)=E_0w_*(0)=0$. Then the claim is proved.

 By the almost periodicity of $w$, we have $$0=\lim\limits_{i\to\infty}w_*(n-n_i)=\lim\limits_{j\to\infty}\lim\limits_{i\to\infty}w(n+n_i-n_j)=\lim\limits_{i\to\infty}w(n+n_i-n_i)$$ which yields that $p_1=\lambda p_2$. Then we obtain the result.
\end{proof}

The positivity of the ground state energy is also important in our applications, and we will provide some kind of conditions to guarantee it.

\begin{Proposition}\label{Lemma:groundstateener}
Let $d\in\N_+$, and suppose that $A_1,W\in C(\T^d,\R)$, $A_2\in\R, A_1>A_2$. Let $P\in C(\T^d,\R), P>0$ and $p(n)=P(n\omega)$ be an almost periodic solution of
 $$\mathcal Lp=D^*(A_1(n\omega)Dp)+A_2D^*p+W(n\omega)p=E_0p.$$
If $\hat W(0)\geq 0$ and $W\not\equiv 0$,  then $E_0>0$.
\end{Proposition}

\begin{proof}
In the case that $A_2\leq 0$. Dividing $p$ from both sides of $\mathcal Lp=E_0p$ and taking its average, one has
\begin{equation}\label{eq:groundstateenergy}
E_0=\lim\limits_{n\to\infty}\frac{1}{n}\sum\limits_{i=1}^n\biggr(\frac{D^*\bigr(A_1(\omega i)Dp(i)\bigr)}{p(i)}+A_2\frac{D^*p(i)}{p(i)}+W(\omega i)\biggr).
\end{equation}
Note that for any $u(n)=U(n\omega)$ with $U\in C(\T^d,\R)$, by Birkhoff ergodic theorem, there holds
 $$\lim\limits_{n\to\infty}\frac{1}{n}\sum\limits_{i=1}^nu(i)=\int_{\T^d}U(\theta)d\theta.$$
By \eqref{eq:groundstateenergy}, it follows that
\begin{equation*}\label{dominatedtermpo}
\begin{aligned}
\lim\limits_{n\to\infty}\frac{1}{n}\sum\limits_{i=1}^n\frac{D^*\bigr(A_1(\omega i)Dp(i)\bigr)}{p(i)}&=\lim\limits_{n\to\infty}\sum\limits_{i=1}^n\frac{A_1(\omega i)}{n}\biggr(\frac{\bigr(Dp(i)\bigr)^2}{p(i)p(i+1)}\biggr)\\
&=\int_{\T^d}\frac{A_1(\theta)\bigr(P(\theta+\omega)-P(\theta)\bigr)^2}{P(\theta)P(\theta+\omega)}d\theta,
\end{aligned}
\end{equation*}
which yields that $\lim\limits_{n\to\infty}\frac{1}{n}\sum\limits_{i=1}^n\frac{D^*\bigr(A_1(\omega i)Dp(i)\bigr)}{p(i)}\geq 0$, and the equality holds only if $P(\theta)\equiv \const$. Meanwhile, by the rearrangement inequality, one has
$$\frac{1}{n}\sum\limits_{i=1}^n\frac{p(i-1)}{p(i)}\geq \frac{1}{n}\sum\limits_{i=1}^n\frac{p(i)}{p(i)}=1,$$
and then $A_2\lim\limits_{n\to\infty}\frac{1}{n}\sum\limits_{i=1}^n\frac{D^*p(i)}{p(i)}\geq 0$ follows directly. Then we conclude that
$$
\begin{aligned}
E_0&= \lim\limits_{n\to\infty}\frac{1}{n}\sum\limits_{i=1}^n \biggr(\frac{D^*\bigr(A_1(\omega i)Dp(i)\bigr)}{p(i)}+A_2\frac{D^*p(i)}{p(i)}+W(\omega i)\biggr)\\
&\geq \int_{\T}\frac{A_1(\theta)\bigr(P(\theta+\omega)-P(\theta)\bigr)^2}{P(\theta)P(\theta+\omega)}d\theta+\int_{\T}W(\theta)d\theta\geq 0,
\end{aligned}
$$
and the equality holds if and only if $\int_{\T} Wd\theta=0$ and $P\equiv \const,$ and this implies $W\equiv 0$ which contradicts with the assumption.

In the case that $0<A_2<A_1(\theta)$,  we can rewrite that
$$\mathcal Lu=D^*(A_1(n\omega)Du)+A_2D^*u+W(n\omega)u=D^*\biggr(\bigr(A_1(n\omega)-A_2\bigr)Du\biggr)+A_2Du+W(n\omega)u.$$
Similar to the above argument, we still have $A_2\lim\limits_{n\to\infty}\frac{1}{n}\sum\limits_{i=1}^n\frac{Dp(i)}{p(i)}\geq 0$, which also yields $E_0>0$. Then we finish the proof.
\end{proof}

\subsection{From Jacobi to discrete elliptic operator}

 It is straightforward to check that the Jacobi operator $\mathcal L^g_{W_1,W_2,V}$ can be rewritten as
  $$\mathcal Lu=D^*\bigr(a(n)Du\bigr)+b(n)D^*u+c(n)u,$$
where  $a(n)=\me^{-g}+ W_1(n\omega), b(n)=\me^{-g}-\me^g+ W_1(n\omega-\omega)- W_2(n\omega), c(n)=(V-W_1-W_2)(n\omega)-\me^{-g}-\me^g$.
Note that $ \inf\bigr(\me^{-g}+W_1\bigr), \inf\bigr(\me^{g}+W_2\bigr)>0$ provided that $C_0\me^{-12|g|}g^6<\me^{-|g|}$.
Meanwhile, direct computation shows that
 $$\mathcal L^*u=D^*(a(n)Du)-D(b(n)u)+c(n)u=\mathcal L^{-g}_{W_2(\cdot+\omega), W_1(\cdot-\omega),V}u.
$$
Since we have obtained the almost periodic solution of $\mathcal L^g_{W_1, W_2, V}p=E_0p$ in the last section, then the main result can be proved as follows:

\begin{proof}[Proof of Theorem \ref{Main1}]
By Corollary \ref{ap-ir}, there exist $E_0, E_0'\in\R$ and $P,P_*\in C(\T,\R),$ $P,P_*>0$ such that  $p(n)=P(n\omega), p_*(n)=P_*(n\omega)$ satisfy
$$\mathcal L^g_{W_1,W_2,V}p=E_0p, \mathcal L^{-g}_{W_2(\cdot+\omega), W_1(\cdot-\omega),V}p_*=E_0'p_*.$$
 By Lemma \ref{unique}, we have $E_0=E_0'$, and then
 from Proposition \ref{Th:location of groundstate} and Lemma \ref{Lemma:simple}, we have obtained the result.
\end{proof}

\begin{proof}[Proof of Theorem \ref{Main2}]
Replacing Corollary \ref{ap-ir} with Corollary \ref{dense-ir} in the proof of Theorem \ref{Main1}, the result follows directly.
\end{proof}

\begin{proof}[Proof of Theorem \ref{Main3}]
Replacing Corollary \ref{ap-ir} with Corollary \ref{Theorem:positive bloch wave} in the proof of Theorem \ref{Main1}, the result follows directly.
\end{proof}

Conversely, one can rewrite any discrete elliptic operator
\begin{equation*}
\mathcal Lu=D^*\bigr(A_1(n\omega)Du\bigr)+A_2(n\omega)D^*u+W(n\omega)u,
\end{equation*}
as a Jacobi operator. Here we assume $A_1,A_2,W\in C(\T^d,\R), d\in\N_+\cup\{\infty\},$ and $\inf A_1(\theta)>0 ,\inf(A_1(\theta-\omega)-A_2(\theta))>0.$ Indeed, by direct computation, if we denote  $W_1(\theta)=A_1(\theta)-\hat A_1(0), W_2(\theta)=A_1(\theta-\omega)-A_2(\theta)-\hat A_1(0)+\hat A_2(0), V(\theta)=(W -A_1+A_2)(\theta)-A_1(\theta-\omega)$, then one has
\begin{equation}\label{eq:jacobitoellip}
 \frac{1}{h}(\mathcal Lu)(n)=\biggr(\me^{-g}+\frac{W_1}{h}(n\omega)\biggr)u(n+1)+\biggr(\me^{g}+\frac{W_2(n\omega)}{h}\biggr)u(n-1)+\frac{V(n\omega)}{h}u(n),
 \end{equation}
where $\me^{-g}=\sqrt{\frac{\hat A_1(0)}{\hat A_1(0)-\hat A_2(0)}}$, $ h=\sqrt{\bigr(\hat A_1(0)- \hat A_2(0)\bigr)\hat A_1(0)}$. Thus  $\hat A_2(0)\neq 0$ means $g\neq 0$. Consequently, we have the following:

\begin{proof}[Proof of Proposition \ref{Main1ell}]
Note that adding any constant to potential $W$ or multiplying $\mathcal L$ by any constant does not affect the existence of the almost periodic ground state of $\mathcal L$. Hence we only need to consider $\frac{\mathcal L}{h}-\frac{\hat V(0)}{h}.$

Since we have \eqref{eq:jacobitoellip}, applying Theorem \ref{Main3}, the existence of the ground state energy $E_0$ and the almost periodic ground state can be obtained provided that
 $$
 \sum\limits_{k\in\mathbb Z_*^\infty}\frac{1}{h}(|\hat {A_1}(k)|+|\hat {A_2}(k)|+|\hat {W}(k)|) \me^{r|k|_1}-|\hat A_1(0)|-|\hat A_2(0)|-|\hat W(0)|)\leq \frac{3\epsilon_0}{h}\leq C_0\me^{-12|g|}|g|^6,
$$
where $C_0$ is given by Theorem \ref{Main3}, and $g=-\frac{1}{2}\ln\bigr(\frac{\hat A_1(0)}{\hat A_1(0)-\hat A_2(0)}\bigr)$.
Then we finish the proof.
\end{proof}

\begin{Corollary}\label{Main2ell}
Let $A_2\in\R, A_2\neq 0$, $A_1,W\in C^{1,\gamma}(\T,\R)$ for some $ 0<\gamma< 1.$ Suppose that $\hat W(0)\geq 0, W\not\equiv 0$. Then there
exists $\epsilon_0=\epsilon_0(\hat A_1(0),A_2)>0$, such that if
\begin{equation*}
\|A_1-\hat A_1(0)\|_{1,\gamma}+\|W-\hat W(0)\|_{1,\gamma}\leq \epsilon_0,
\end{equation*}
then for any $\omega\in\mathcal P$,  $\mathcal L$
admits a simple ground state energy $E_0>0$ and a quasi-periodic ground state $p(n)=P(n\omega)$ with $P\in C(\T,\R)$ and $P>0$.
\end{Corollary}

\begin{proof}
Replacing Theorem \ref{Main3} with Theorem  \ref{Main1} in the proof of Proposition \ref{Main1ell} with minor modification, one can establish the existence of the quasi-periodic ground state. Then the result follows from Proposition \ref{Lemma:groundstateener}.
\end{proof}

\begin{Corollary}\label{Main3ell}
Let $\omega\in\R^d\backslash\Q^d$, $d\in\N_+,$ and suppose that $A_2\in\R, A_2\neq 0$, $A_1, W'\in C(\T^d,\R)$, $\hat W'(0)\geq 0$.  There
exists $\epsilon_0=\epsilon_0(\hat A_1(0),A_2)>0$ such that if \begin{equation*}
\|A_1-\hat A_1(0)\|_0+\|W-\hat W(0)\|_0\leq \epsilon_0,
\end{equation*}
then for any $\epsilon>0$, there exists $W\in C(\T^d, \R)$ which satisfies $\|W-W'\|_0<\epsilon$ and $\hat W(0)=\hat W'(0)$ such that
   $\mathcal L$
admits a simple ground state energy $E_0>0$ and a quasi-periodic ground state $p(n)=P(n\omega)$ with $P\in C(\T^d,\R)$, $P>0$.
\end{Corollary}

\begin{proof}
The result can be proved by replacing Theorem \ref{Main3} with Theorem \ref{Main2} in the proof of Proposition \ref{Main1ell} with minor modification.
\end{proof}

\section{Steady state of the discrete Fisher-KPP equation}\label{section5}

As our first application,  we establish the existence of the positive almost periodic steady state of the Fisher-KPP type equation in discrete media
\begin{equation}\label{NP}
u_t=D^*(a(n)Du)+b(n)D^*u+c(n)u-u^2=0, \text{ in }\R\times\Z.
\end{equation}
where $a,b,c$ are almost periodic and $\inf\limits_{n\in\Z} a(n),\inf\limits_{n\in\Z}\bigr(a(n)-b(n-1)\bigr)>0,$ $\langle c\rangle:=\frac{1}{n}\sum\limits_{i=1}^nc(i)\geq 0$ and $c\not= 0$.
Here, we will always assume that the linearized operator
\begin{equation}\label{eq:linearized}
\mathcal Lp=D^*(a(n)Dp)+b(n)D^*p+c(n)p=E_0p
\end{equation}
has a ground state energy $E_0>0$ and a bounded function $p$ with positive infimum.

\begin{Lemma}\label{prop:nonpara}
Let $p$ be a bounded solution of \eqref{eq:linearized}  with positive infimum. Then \eqref{NP} has a unique steady state with positive infimum.
\end{Lemma}
\begin{proof}
 Now we prove the result in the following steps:\\
 Step 1: First we claim there exists a solution $u$ of the equation \eqref{NP}. Moreover, $u(t,n)$ is nondecreasing in $t$.

Since $-c(n)M+M^2>0$ for $M>\sup c$, then $\overline u:=M$ is a supersolution of \eqref{NP}.
Denote $\underline u:=\epsilon p$ with $\epsilon$ to be chosen.
Then for $\epsilon<\frac{\min\{M,E_0\}}{\sup p}$,
$$\underline u_t-\mathcal L\underline u+\underline u^2=-\epsilon(\mathcal Lp)+\epsilon^2p^2=-\epsilon E_0p+\epsilon^2p^2<0.$$
This yields that $\underline u$ is a subsolution of \eqref{NP}. Moreover, $\underline u=\epsilon p<M=\bar u.$

Define the sequence $\{u_{i}\}$ as follows: $u_{i}$ is the solution of \eqref{NP} for $t>-i$ with initial condition $u_{i}(-i,n)=\overline u(-i,n)$, and the existence and uniqueness of $u_{i}$ is due to Theorem \ref{existence and uniqueness}. By Proposition \ref{Strong comparison}, $u_{i}$ satisfies
$$\forall t>-i,\ n\in\mathbb{Z}, \ \underline u(t,n)\leq u_{i}(t,n)\leq \overline u(t,n).$$
Thus, for $i,j\in\mathbb{N}$ with $j<i$ and for any $0<h<1$, using the monotonicity of $\overline u$, we will get
$$u_{j}(-j,n)=\overline u(-j,n)\geq \overline u(-j-h,n)\geq u_{i}(-j-h,n).$$
Note that $u_{i}(\cdot-h,\cdot)$ is also a solution of \eqref{NP}. Proposition \ref{Strong comparison} gives us
$$\forall j<i, \ 0<h<1,t>-j,\ \ u_{j}(t,n)\geq u_{i}(t-h,n).$$
Now by the arguments before, we can prove $\{u_{i}\}_i$ converges locally uniformly to a global function $\underline u\leq u\leq \overline u$ of \eqref{NP}. Then passing to the limit as $i,j\rightarrow \infty$,  $u(t,n)\geq u(t-h,n)$ for all $(t,n)\in\mathbb{R}\times\mathbb{Z}$ and $0<h<1$. This means that $u$ is nondecreasing in $t$.  Then the claim is proved.

Step 2: Then we get the steady state of the equation \eqref{NP}.

Note that $u(t,n)\leq M$ with $M>\sup c$ and $u$ is nondecreasing in $t$. Hence by the monotone convergence theorem, $u_0(n):=\lim\limits_{t\to\infty}u(t,n)\geq\underline u\geq \epsilon \inf p$. Note that $\lim\limits_{t\to\infty}u_t(t,n)=0$. Taking limit in $t\to\infty$ from the both sides of the equation \eqref{NP}, one has
$$\mathcal Lu_0-u_0^2=D^*(a(n)Du_0)+b(n)D^*u_0+c(n)u_0-u_0^2=0.$$
Moreover, $\epsilon p\leq u_0\leq M.$

Step 3: Now we conclude the proof by proving uniqueness.

Assume by contradiction that $v_0$ is a steady state with positive infimum which is different from $u_0$. Without loss of generality we assume that $u_0(0)>v_0(0)$. Note that $\inf v_0>0$ and $u_0,v_0$ are both bounded, then as we have explained in the proof of Lemma \ref{Lemma:simple}, there exists $\epsilon\in (0,1)$
such that $\inf(v_0-\epsilon u_0)=0$. Let $w=\epsilon\me^{\delta t}u_0$ with $\delta$ to be chosen later.
By the direct calculation,
$$
\begin{aligned}
w_t-\mathcal Lw+w^2&\leq \epsilon\me^{\delta t}(\delta u_0-\mathcal Lu_0+\epsilon\me^{\delta t}u_0^2)\leq \epsilon\me^{\delta t}u_0(\delta+\epsilon\me^{\delta t}u_0-u_0).
\end{aligned}
$$
Choose $\delta<\frac{1-\epsilon}{2}\inf\limits_{n\in\Z}u_0(n)$, then for any $t\leq\frac{1}{\delta}\ln \bigr(\frac{1}{2\epsilon}+\frac{1}{2}\bigr)$, we have $\delta+\epsilon\me^{\delta t}u_0-u_0\leq 0.$
This implies that
$$
\left\{
\begin{aligned}
w_t-\mathcal Lw+w^2&\leq 0 & & \text{ in }\biggr[0, \frac{1}{\delta}\ln\bigr(\frac{1}{2\epsilon}+\frac{1}{2}\bigr)\biggr]\times\Z,\\
w(0,n)&\leq v_0(n) & & \text{ in }\Z,
\end{aligned}
\right.
$$
and thus $w$ is a subsolution of $u_t-\mathcal Lu+u^2=0$. By Proposition \ref{Strong comparison}, we have $w\leq v_0$ in $\bigr[0,\frac{1}{\delta}\ln\bigr(\frac{1}{2\epsilon}+\frac{1}{2}\bigr)\bigr]\times\Z$. This contradicts with the assumption $\inf\limits_{n\in\Z} (v_0-\epsilon u_0)(n)=0$. Hence we have proved uniqueness.
The proof is finished.
\end{proof}

Using the sliding method, we can prove the almost periodicity of the steady state:

\begin{Proposition}\label{Prop:steadystate}
Let $u_0$ be the steady state which is obtained in Lemma \ref{prop:nonpara}. Then the steady state $u_0$ of the equation \eqref{NP} is almost periodic. Moreover, if for any sequence $\{n_i\}_i$ such that $a(\cdot+n_i), b(\cdot+n_i), c(\cdot+n_i)$ converge uniformly, so does $u_0(\cdot+n_i)$.
\end{Proposition}
\begin{proof}
Now we divide the proof in the following steps:

Step 1: First we claim that there exists a unique bounded solution with positive infimum of the following equation
 \begin{equation}\label{eq:limitNP}
 D^*(\tilde a(n)Du)+\tilde b(n)D^*u+\tilde c(n)u-u^2=0.
 \end{equation}
 where $\tilde a(n)=\lim\limits_{i\to\infty} a(n+n_i), \tilde b(n)=\lim\limits_{i\to\infty}b(n+n_i), \tilde c(n)=\lim\limits_{i\to\infty} c(n+n_i)$ for some sequence $\{n_i\}_i$.

   From $\mathcal Lp=E_0p$, we have  $|p(n)|\leq C(n)$, and then $p(\cdot+n_{i})$ converges locally uniformly to a positive bounded solution $\tilde p$ of  \begin{equation*}\label{eq:limitp}
D^*(\tilde a(n)Dp)+\tilde b(n)D^*p+\tilde c(n)p=E_0p.
 \end{equation*}
 Meanwhile, $\inf p(\cdot+n_i)\geq \inf p$ implies that $\inf\tilde p>\inf p.$ Now by the arguments in the proof of Lemma \ref{prop:nonpara}, the claim is proved.

Step 2: Now we prove the steady state is almost periodic.  Let $u_0$ be the steady state of \eqref{NP}.
 It is sufficient to prove for any pair of sequences $\{n_{i}\}_{i\in\mathbb{N}},\ \{m_{i}\}_{i\in\mathbb{N}}$ in $\mathbb{Z}$, we can extract a common subsequences $\{n'_{i}\}_{i\in\mathbb{N}},\ \{m'_{i}\}_{i\in\mathbb{N}}$ such that
    $$\lim\limits_{i\rightarrow\infty}u_0(n+m'_{i}+n'_{i})=\lim\limits_{i\rightarrow \infty}\left(\lim\limits_{j\rightarrow \infty}u_0(n+n'_{i}+m'_{j})\right).$$
    Now consider two sequences $\{n_{i}\}_{i\in\mathbb{N}},\ \{m_{i}\}_{i\in\mathbb{N}}$ in $\mathbb{Z}$ such that $\tilde a_*(n):=\lim\limits_{i\to\infty}a(n+n_i+m_i), \tilde b_*(n):=\lim\limits_{i\to\infty}b(n+n_i+m_i)$, $\tilde c_*(n):=\lim\limits_{i\to\infty}c(n+n_i+m_i)$ exist. Since the limits $$\quad \quad\quad\tilde u_0:=\lim\limits_{i\rightarrow \infty}u_0(\cdot+n_{i}),\ \tilde u_{*}:=\lim\limits_{i\rightarrow \infty}\tilde u_0(\cdot+m_{i}),\ u_{*}:=\lim\limits_{i\rightarrow\infty}u_0(\cdot+n_{i}+m_{i}),$$ exist (up to subsequences) locally uniformly in $n\in\mathbb{Z}$, and from the almost periodicity of $a,b,c$, we have $\tilde u_{*}$ and $u_{*}$ are both solutions of \eqref{eq:limitNP} with the coefficients of $\tilde a_*,\tilde b_*,\tilde c_*$. Moreover, $$\sup \tilde u_*\leq \sup u_0, \sup u_*\leq \sup u_0, \inf u_0\leq \inf\tilde u_*, \inf u_0\leq \inf u_*.$$ By step 1, $u_*=\tilde u_*$, and then the result is proved.
\end{proof}

With these at hand, we will prove the main result in this application as follows.
\begin{proof}[Proof of Proposition \ref{KPPSS}]
Combing Proposition \ref{prop:nonpara} with Lemma \ref{Prop:steadystate}, the result follows.
\end{proof}

Then as the application of Corollaries \ref{Main2ell} and \ref{Main3ell}, we now begin to prove Theorems \ref{APPMain1} and \ref{APPMain2}.

\begin{proof}[Proof of Theorem \ref{APPMain1}]
We will prove the result for more general equation
\begin{equation}\label{eq:APPMAINeq}
u_t=D^*(A_1(n\omega)Du)+  B_0D^*u+ W(n\omega) u-u^2, \text{ in }\R\times\Z,
\end{equation}
 where $\|A_1-\hat A_1(0)\|_{1,\gamma}+|W(0)-\hat W(0)\|_{1,\gamma}\leq\epsilon_0,$ and $\epsilon_0$ depends on $\hat A_1(0),B_0$.
By Corollary \ref{Main2ell}, there exist a ground state energe $E_0>0$ and a quasi-periodic ground state $p(n)=P(n\omega)$ with $P\in C(\T,\R),\ P>0.$ Thus the existence and uniqueness of the positive bounded steady state with positive infimum can be obtained in Proposition \ref{KPPSS}.

Now it remains to  prove $u_0(n)=U_0(n\omega)$ with $U_0\in C(\T,\R)$ and $U_0>0$.
Similar to \eqref{eq:APPMAINeq}, for any fixed $\theta\in\T$, there exists a unique bounded steady state $u_\theta(n)$ with positive infimum of \eqref{eq:APPMAINeq}, where $A_1(\cdot), W(\cdot)$ is replaced by $A_1(\cdot+\theta), W(\cdot+\theta)$.
Redefining function on $\T^d:$ $U_0(\theta):=u_\theta(0).$
It is clear from the uniqueness of the bounded steady state with positive infimum, we have $u_{k\omega+\theta}(0)=u_\theta(k\omega).$
Now we prove that $U_0$ is continuous in $\theta\in\T.$
It is sufficient to prove for any $\theta_k\to\theta$, $U_0(\theta_k)\to U_0(\theta).$

  For any $\theta\in\T$ and any sequence $\{n_i\}_i$ such that $A_1(n_i\omega+\theta+\cdot), W(n_i\omega+\theta+\cdot)$ converge uniformly to $A(\cdot+\theta), W(\cdot+\theta)$ respectively, by Proposition \ref{Prop:steadystate} and the uniqueness of the bounded steady state with positive infimum,  $u_{\theta}(n_i\omega+\cdot)$ converges uniformly to $u_\theta(\cdot)$. Hence this implies that
$$\lim\limits_{i\to\infty}U_0(n_i\omega+\theta)=\lim\limits_{i\to\infty}u_{n_i\omega+\theta}(0)=\lim\limits_{i\to\infty}u_\theta(n_i\omega)=u_\theta(0)=U_0(\theta).$$
Then for any $\theta_k$ such that $\theta_k\to\theta$, since $\{n\omega|n\in\Z\}$ is dense in $\T$, there exists a subsequence $\{n_{i,k}\}_{i}$ such that $n_{i,k}\omega\to\theta_k$ as $i\to\infty$ for any $k\in\Z$.
Hence by diagonal extraction of a subsequence, there holds $$\lim\limits_{k\to\infty}U_0(\theta+\theta_k)=\lim\limits_{k\to\infty}\lim\limits_{i,k\to\infty}U_0(\theta+{n_{i,k}\omega})= U_0(\theta).$$
 Thus we have proved the result.
\end{proof}

\begin{proof}[Proof of Theorem \ref{APPMain2}]
Replacing Corollary \ref{Main2ell} with Corollary \ref{Main3ell} in the proof of Theorem \ref{APPMain1}, the result follows directly.
\end{proof}

\section{Averaging of discrete parabolic equations}\label{section6}
The other application is to investigate the averaging of the discrete parabolic operator of the divergence type with large lower terms. For any $\epsilon>0$, define the following family of operators $\mathcal L_\epsilon$ from $L^2(\epsilon\Z)$ to $L^2(\epsilon\Z)$ that is constructed from the operator $\mathcal L$:
\begin{equation}\label{eq:familyop}
\begin{aligned}
&\mathcal L_\epsilon u_\epsilon:=D_\epsilon^*\biggr(A_1(\frac{\omega x}{\epsilon})D_\epsilon u_\epsilon\biggr)+\frac{1}{\epsilon}A_2\bigr(\frac{\omega x}{\epsilon}\bigr)D^*_\epsilon u_\epsilon+\frac{1}{\epsilon^2}W\bigr(\frac{\omega x}{\epsilon}\bigr)u_\epsilon, \ x\in\epsilon\Z,
\end{aligned}
\end{equation}
where $A_1,A_2,W\in C^\omega_{r_0}(\T^\infty,\R), r_0>0$ (c.f. \eqref{almostinfinite}) with $A_1>0,\ A_1(\theta-\omega)>A_2(\theta)$, $\omega\in \mathrm{DC}_\infty(\gamma,\tau)$, $\gamma\in (0,1),\tau>1$.

For any $\epsilon>0$, we consider the evolution equation defined in $[0,T]\times\epsilon\Z$ with $T>0$:
\begin{equation}\label{eq:initialpara}
\left\{
\begin{aligned}
(u_\epsilon)_t&=\mathcal L_\epsilon u_\epsilon & & \text{ in }[0,T]\times\epsilon\Z,\\
u_\epsilon(0,x)&=\varphi(x) & &\text{ in }\epsilon\Z,
\end{aligned}
\right.
\end{equation}
where $\varphi\in C^2(\R)$ and $\varphi(z)\leq C|z|^{-k}$ for some $C>0, k\geq 2, z\in\R.$ Note that $\varphi$ can be restricted as a function in $\epsilon\Z$.
Motivated by Lemma \ref{reduc},  we will discuss the asymptotic behavior of $u_\epsilon$ of \eqref{eq:initialpara} based on the following crucial assumption:\\
There exist $E_0\in\R$ and positive functions $P, P_*\in C^{\omega}_{r_0}(\T^\infty,\R)$ such that
\begin{equation}\label{averagingtorus}
\tilde D^*\bigr(A_1\tilde DP\bigr)+A_2\tilde D^*P+WP=E_0P,\
\tilde D^*\bigr(A_1\tilde DP_*\bigr)-\tilde D\bigr(A_2P_*\bigr)+WP_*=E_0P_*,
\end{equation}
where $\tilde DP(\theta)=P(\theta+\omega)-P(\theta),\ \tilde D^*P(\theta)=P(\theta)-P(\theta-\omega)$.

Now for any $U\in C^\omega_{r_0}(\T^\infty,\R),$ denoting $\langle U\rangle=\int_{\T^\infty} Ud\theta$,
  without loss of generality, we assume  that
\begin{equation}\label{eq:normalized}
\langle P\rangle=\langle P_*\rangle =1,\ P,P_*\geq \delta>0.
\end{equation}

Consider the operator $\tilde D^*\bigr(A_1\tilde DP\bigr)+A_2\tilde D^*P+WP=E_0P$ restricted to a line $\theta=\omega z$, $z\in\R$,
$$
\begin{aligned}
Lu:=D^*(A_1(\omega z) Du)+A_2(\omega z) D^*u+W(\omega z)u=E_0u,
\end{aligned}
$$
where $D=D_1$, and $D_\epsilon$ defined on $\ell^2(\Z)$ is naturally extended to $L^2(\R)$ as $D_\epsilon^*u(z)=\frac{u(z)-u(z-\epsilon)}{\epsilon}$,  $D_\epsilon u=\frac{u(z+\epsilon)-u(z)}{\epsilon}$, for any $u\in L^2(\R)$ and any $\epsilon>0.$
By the same calculation in Lemma \ref{reduc}, there exists $ \bar c(z)$ such that
$$
P_*(\omega z)(L-E_0)\bigr(P(\omega z) u(z)\bigr)=D^*\bigr(\tilde A_1(\omega z) Du\bigr)+\bar c(\omega z)(D+D^*)u,
$$
where $\tilde A_1(\theta)=\frac{1}{2}\bigr[(A_1P)(\theta)P_*(\theta+\omega)-(A_2P_*)(\theta+\omega)P(\theta)+(A_1P_*)(\theta)P(\theta+\omega)\bigr]$, $\bar c(\theta)=(A_1P_*)(\theta-\omega)P(\theta)-(A_1P)(\theta-\omega)P_*(\theta)+(A_2P_*)(\theta)P(\theta-\omega)$. Notice that the proof of Lemma \ref{reduc} shows $\bar c(\theta)=\bar c(\theta+\omega)$, and thus $\bar c(\theta)$ is a constant function since $\{k\omega|k\in\Z\}$ is dense in $\T^\infty$, and $\tilde A_1>|\bar c(\theta)|$.
Hence for simplicity, we denote $\bar c(\theta)=\bar c\in\R,$ and
\begin{equation}\label{eq:reducinR}
P_*(\omega z)(L-E_0)\bigr(P(\omega z) u(z)\bigr)=D^*\bigr(\tilde A_1(\omega z) Du\bigr)+\bar c(D+D^*)u.
\end{equation}

In this way,  equation \eqref{eq:familyop} and \eqref{eq:initialpara} can be lifted up to the family of equations defined in $[0,T]\times\R$
\begin{equation}\label{eq:familyopnew}
L_\epsilon u_\epsilon:=D_\epsilon^*\biggr(A_1\bigr(\frac{\omega z}{\epsilon}\bigr)D_\epsilon u_\epsilon\biggr)+\frac{1}{\epsilon}A_2\bigr(\frac{\omega z}{\epsilon}\bigr)D^*_\epsilon  u_\epsilon+\frac{1}{\epsilon^2}W\bigr(\frac{\omega z}{\epsilon}\bigr)u_\epsilon, \ z\in\R,
\end{equation}
and
\begin{equation}\label{eq:extendedeq}
\left\{
\begin{aligned}
(u_\epsilon)_t&=L_\epsilon u_\epsilon,\ (t,z)\in [0,T]\times\R\\
 u_\epsilon(0,z)&=\varphi(z), \ z\in\R,
\end{aligned}
\right.
\end{equation}
respectively,
and the equations \eqref{eq:familyop}(\eqref{eq:initialpara}) can be regarded as the equations \eqref{eq:familyopnew}(\eqref{eq:extendedeq}) restricted into $\Z$.
 It is straightforward to check from \eqref{eq:reducinR} that
 \begin{equation}\label{eq:reduclinemulsca}
 P_*(\frac{\omega z}{\epsilon})\biggr(L_\epsilon-\frac{E_0}{\epsilon^2}\biggr) \bigr(P(\frac{\omega z}{\epsilon})u_\epsilon\bigr)=D^*_\epsilon\biggr(\tilde A_1\bigr(\frac{\omega z}{\epsilon}\bigr) D_\epsilon u_\epsilon\biggr)+\frac{\bar c}{\epsilon}(D_\epsilon+D_\epsilon^*)u_\epsilon.  \end{equation}

Denote $Q=PP_*$. Multiplying $P_*(\frac{\omega z}{\epsilon})$ from both sides of \eqref{eq:extendedeq}, and using \eqref{eq:reduclinemulsca}, we can check that  $v_\epsilon(t,z)=\me^{-\frac{E_0t}{\epsilon^2}}\frac{u_\epsilon(t,z-\frac{lt}{\epsilon})}{P(z')}$ with $z'=\omega\bigr(\frac{z}{\epsilon}-\frac{lt}{\epsilon^2}\bigr)$ satisfies the equation defined in $(t,z)\in [0,T]\times\R$,
\begin{equation}\label{eq:transpara}
\left\{
\begin{aligned}
&\mathcal P_\epsilon v_\epsilon:=Q(z') (v_\epsilon)_t-D^*_\epsilon\biggr(\tilde A_1(z')D_\epsilon v_\epsilon\biggr)-\biggr[\frac{\bar c}{\epsilon}(D^*_\epsilon+D_\epsilon)-\frac{l}{\epsilon}Q(z')\partial_z\biggr]v_\epsilon=0,\\
&v_\epsilon(0,z)=\frac{\varphi(z)}{P(\frac{\omega z}{\epsilon})} \text{ in }\R,
\end{aligned}
\right.
\end{equation}
where $\partial_zu(t,z)$ denotes the derivative in the second variable, and $l\in\R$ is to be chosen later. Then we only need to analyse the asymptotic behaviour of the solution $v_\epsilon$ of \eqref{eq:transpara}.

To solve \eqref{eq:transpara}, we derive certain energy estimate. Let us introduce some function spaces:
 $$H^1(\R)=\{u:\R\rightarrow \R:\|u\|^2_{H^1}:=\| u'\|^2_{L^2}+\|u\|^2_{L^2}<\infty\},$$
 and we denote $H^{-1}(\R)$ the dual space of $H^1(\R)$.
Let $\mathscr S:=L^2([0, T];H^1(\R))$ which is endowed with the norm $\|\cdot\|_{\mathscr S}$, and $\mathscr S'$ 
be the dual of $\mathscr S$ which is endowed with the norm  $\|\cdot\|_{\mathscr S'}$.

\begin{Proposition}\label{proposition:energy estimate}
Let $f\in \mathscr S'\cap C^1([0,T]\times\R),\phi\in L^2(\R)\cap C^1(\R)$. Then there exists a unique
solution $v_\epsilon\in \mathscr S\cap C^1([0,T]\times\R)$ of
\begin{equation}\label{eq:generalizedeq}
\left\{
\begin{aligned}
\mathcal P_\epsilon v_\epsilon&=f & & \text{ in } [0,T]\times\R,\\
v_\epsilon(0,z)&=\phi(z)& & \text{ in }\R.
\end{aligned}
\right.
\end{equation}
Moreover, there holds the estimate
\begin{equation}\label{eq:energyestimates}
\sup\limits_{0\leq t\leq T}\|v_\epsilon(t,\cdot)\|^2_{L^2(\R)}\leq C\bigr(\|f\|^2_{\mathscr
S'}+\|\phi\|^2_{L^2(\R)}\bigr),
\end{equation}
where the constant $C$ depends only on $\inf P\inf P_*$, $\inf A_1, \inf \bigr(A_1(\theta-\omega)-A_2(\theta)\bigr)$.
\end{Proposition}

\begin{proof}
Recall that $z'=\omega\bigr(\frac{z}{\epsilon}-\frac{lt}{\epsilon^2}\bigr)$.
Let $\sigma(z)=Q(z'),\tilde \alpha(z)=\frac{1}{2}\bigr[(A_1P)(z')P_*(z'+\omega)-(A_2P_*)(z'+\omega)P(z')+(A_1P_*)(z')P(z'+\omega)\bigr]$. Then it is straightforward to check that $\sigma,\tilde \alpha$ are  bounded analytic, hence smooth.
By the standard semigroup theory \cite{Pazy}, the existence of the solution $v_\epsilon$ of equation \eqref{eq:generalizedeq} can be obtained, and moreover, $v_\epsilon\in C^1([0,T]\times \R)$.
 Multiplying both sides of \eqref{eq:generalizedeq} by $v_\epsilon$ and integrating over $[0,t]\times\R$, we obtain for any $0<t'\leq T,$
$$
\begin{aligned}
&\frac{1}{2}\int_0^{t'}\int_{\R}\sigma(v^2_\epsilon)_tdzdt+\int_0^{t'}\int_{\R}\tilde \alpha (D_\epsilon v_\epsilon)^2dzdt-\frac{l}{2\epsilon}\int_0^{t'}\int_{\R}\sigma \partial_zv_\epsilon^2dzdt=\int_0^{t'}\int_{\R}fv_\epsilon dzdt.
\end{aligned}
$$
Thus we have
$$
\begin{aligned}
&\frac{1}{2}\int_{\R}\sigma v^2_\epsilon(z,t')dz+\int_0^{t'}\int_{\R}\tilde \alpha(D_\epsilon v_\epsilon)^2dzdt=\int_0^{t'}\int_{\R}fv_\epsilon dzdt+\frac{1}{2}\int_{\R}\sigma\phi^2dz.
\end{aligned}
$$
Since $\inf\sigma\geq \inf P\inf P_*>0$ and $\inf\tilde \alpha\geq \inf A_1+\inf \bigr(A_1(\theta-\omega)-A_2(\theta)\bigr)$,  by Cauchy Schwartz inequality, \eqref{eq:energyestimates} is obtained.
Lastly, as we could notice from \eqref{eq:generalizedeq}, $v_\epsilon\in \mathscr S\cap C^1([0,T]\times\R)$, as desired.
\end{proof}

The following lemma is crucial for us to determine the exact value of $l$, and to analyze the asymptotic behavior of the solution $v_\epsilon$ of \eqref{eq:transpara}.

\begin{Lemma}\label{Lemma:solve nondivergenceeq}
Let $r>0, \omega\in \mathrm{DC}_\infty(\gamma,\tau),\gamma\in (0,1),\tau>1, \bar c\in\R.$ Suppose that $A, F\in C^\omega_{r}(\T^\infty,\R)$ with $\langle F\rangle=0$, $A>|\bar c|$. Then for any $0<r'<r$,  there exists $N\in C^\omega_{r'}(\T^\infty,\R)$ such that $N(\omega z)$ satisfies
 \begin{equation}\label{eq:ellip}
 D^*\bigr(A(\omega z)DN\bigr)+\bar c(D+D^*)N=F(\omega z) \text{ in }\R.
 \end{equation}
 \end{Lemma}

\begin{proof}
Recall that $\tilde DP(\theta)=P(\theta+\omega)-P(\theta),\ \tilde D^*P(\theta)=P(\theta)-P(\theta-\omega)$.
To solve equation \eqref{eq:ellip}, it suffices to solve the following equation:
\begin{equation}\label{eq:1:solve nondivergenceeq}
\begin{aligned}
&\tilde D^*\bigr((A(\theta)+\bar c)\tilde DN\bigr)+2\bar c\tilde D^*N=F.
\end{aligned}
\end{equation}

 Changing of variable $Y=(A+\bar c)\tilde D N+2\bar cN$, \eqref{eq:1:solve nondivergenceeq} becomes
$Y(\theta)-Y(\theta-\omega)=F(\theta)$ which can be solved if and only if
$\langle F\rangle=0$, and the solution is not unique.   It can be calculated directly that $Y$ has the following representation $$Y(\theta)= Y_0(\theta)+c'=\sum\limits_{k\in\Z_*^\infty\backslash\{0\}}\frac{\hat { F}(k)}{1-\me^{-\mi\langle k,\omega\rangle}}\me^{\mi \langle k,\theta\rangle}+c',$$ for any $c'\in\R$.
Since $\omega\in \mathrm{DC}_\infty(\gamma,\tau)$, then as the application of Lemma \ref{Prop:smalldivsorestimate},
for any $0<r'_1:=\frac{r+r'}{2}<r$,
$$
\begin{aligned}
|Y_0|_{r'_1}&=\sum\limits_{k\in\Z^\infty_*\backslash\{0\}}\biggr|\frac{\hat { F}(k)}{1-\me^{-\mi\langle k,\omega\rangle}}\biggr|\me^{r|k|_1}\me^{(r'_1-r)|k|_1}\\
&\leq \frac{1}{\gamma}\sup\limits_{k\in\Z_*^\infty}\prod\limits_{i\in\Z}(1+\langle i\rangle^\tau|k_i|^\tau)\me^{(r'_1-r)|k|_1}\sum\limits_{k\in\Z^\infty_*\backslash\{0\}}|\hat { F}(k)|\me^{r|k|_1}\\
&\leq \frac{1}{\gamma}\exp\biggr(\frac{2\mu}{r-r'}\ln\bigr(\frac{2\mu}{r-r'}\bigr)\biggr)|F|_{r},
\end{aligned}
$$
for some $\mu=\mu(\tau).$

Now from \eqref{eq:1:solve nondivergenceeq}, it suffices to prove that there exists $N\in C^\omega_{r'}(\T^\infty,\R)$ such that
\begin{equation}\label{eq:secondineq}
N(\theta+\omega)-N(\theta)+\frac{2\bar c}{ A+\bar c} N(\theta)=\frac{ Y_{0}(\theta)+c'}{A(\theta)+\bar c}:=G.
\end{equation}

Note that $c'$ plays an important roles in solving \eqref{eq:secondineq}. For $c'={-\langle\frac{Y_0}{ A+\bar c}\rangle}/{\langle\frac{1}{A+\bar c}\rangle}$, we have $\langle  G\rangle=0.$ Thus one can check that  \begin{equation}\label{eq:convergent}
U(\theta)=\sum\limits_{k\in\Z^\infty_*\backslash\{0\}}\frac{\hat { G}(k)}{\me^{\mi{\langle k,\omega\rangle}}-1+\frac{2\bar c}{A+\bar c}}\me^{i\langle k,\theta\rangle}.
\end{equation}
Since $\inf\limits_{n\in\Z}|\me^{\mi{\langle k,\omega\rangle}}-1+\frac{2\bar c}{A+\bar c}-n|\geq \inf\limits_{n\in\Z}|\langle k,\omega\rangle-n|$ for any $k\in \Z_*^\infty\backslash\{0\}$,  applying  Proposition \ref{Prop:smalldivsorestimate} once again, we have
$$
\begin{aligned}
|U|_{r'}&\leq \frac{\sup(A+\bar c)}{\gamma\inf(A+\bar c)}\exp\biggr(\frac{2\mu}{r-r'}\ln\bigr(\frac{2\mu}{r-r'}\bigr)\biggr)|Y_0|_{r_1'}\\
&\leq \frac{\sup(A+\bar c)}{\gamma^2\inf(A+\bar c)} \exp\biggr(\frac{4\mu}{r-r'}\ln\bigr(\frac{2\mu}{r-r'}\bigr)\biggr)|F|_{r}.
\end{aligned}
$$
Hence \eqref{eq:convergent} admits a solution $U\in C^\omega_{r'}(\T^\infty,\R)$.
Then we finish the proof.
\end{proof}

Let us return to the averaging problem, first we consider the Cauchy problem with $\ell^2$ initial value without oscillation term:

\begin{Proposition}\label{prop:homononoscillation}
Let $f_\epsilon\in \mathscr S'\cap C([0,T]\times\R)$, and $v_\epsilon\in\mathscr S\cap C([0,T]\times\R)$ be the solution of the Cauchy problem
\begin{equation}\label{boudwithoutosc}
\left\{
\begin{aligned}
\mathcal P_\epsilon v_\epsilon&=f_\epsilon & & \text{ in }[0,T]\times \R,\\
 v_\epsilon(0,z)&=\phi(z) & & \text{ in } \R,
\end{aligned}
\right.
\end{equation}
 with $\|f_\epsilon\|_
{\mathscr S'}\rightarrow 0,\phi\in L^2(\R)\cap C^1(\R)$. Then as $\epsilon\rightarrow 0$, $v_\epsilon\rightarrow u_0$  weakly in $\mathscr S$, and strongly in $L^2([0,T],L^2(\epsilon\Z))$, where $u_0$ is the solution of the
averaged Cauchy problem
\begin{equation}\label{eq:averagedeq}
\left\{
\begin{aligned}
\hat{\mathcal P}u_0:&=\langle Q\rangle(u_0)_t-\bar a \partial_{zz}u_0=0& & \text{ in }[0,T]\times\R,\\
u_0(0,z)&=\phi(z) & & \text{ in }\R,
\end{aligned}
\right.
\end{equation}
with $\bar a\equiv \const>0.$
\end{Proposition}

\begin{proof}
Now we approximate the solution $v_\epsilon^a$ of \eqref{boudwithoutosc} in the form:
$$v_\epsilon^a(t,z)=u_0(t,z)+\epsilon N(\frac{\omega z}{\epsilon}-\frac{\omega lt}{\epsilon^2})\partial_zu_0(t,z) \text{ in }[0,T]\times \R$$
where $l\in\R$ and $N\in C(\T^\infty,\R)$ will be determined later.

Recall that $z'=\omega\bigr(\frac{z}{\epsilon}-\frac{lt}{\epsilon^2}\bigr)$. Since $u_0\in C^\infty$, we collect up $O(1), O(\frac{1}{\epsilon})$ terms
in the left-hand side of the equation $\mathcal P_\epsilon v^a_\epsilon=
f_\epsilon$,  and present them below:
\begin{enumerate}
\item [$O(1):$] $Q(z')(u_0)_t(t,z)-\bigr[\tilde A_1(z'-\omega)+\bigr(\epsilon\tilde A_1 D_\epsilon N+\epsilon D_\epsilon^*(\tilde A_1 N)
+2\bar cN-lQ N\bigr)(z')\bigr]\partial_{zz}u_0(t,z).$
\item [$O(\frac{1}{\epsilon})$:] $-\bigr[\epsilon D^*_\epsilon
\bigr(\tilde A_1(z')D_\epsilon N(z')\bigr)+\bar c(D^*_\epsilon+D_\epsilon)
 N(z')+D^*_\epsilon \tilde A_1(z')+\frac{2\bar c}{\epsilon}-
 \frac{l}{\epsilon}Q(z')\bigr]\partial_zu_0(t,z),$
\end{enumerate}
where we have used $\partial_zu_0=D_\epsilon u_0+o(1), \partial_zu_0=D^*_\epsilon u_0+o(1), \partial_{zz}u_0=D^*_\epsilon D_\epsilon u_0+o(1).$

To guarantee the convergence of the left hand side of the equation $\mathcal P_\epsilon v^a_\epsilon=f_\epsilon$, $O(\frac{1}{\epsilon})
$ are expected to vanish in $\mathscr S',$
 and thus we only need to consider the equation:
\begin{equation}\label{eq:approximate}
D^*\bigr(\tilde A_1(\omega z)DN(\omega z)\bigr)+\bar c(D^*+D) N(\omega z)+D^* \tilde A_1(\omega z)+2\bar c-lQ(\omega z)=0.
\end{equation}
Since $\tilde A_1(\theta)>|\bar c|$, as the application of Lemma \ref{Lemma:solve nondivergenceeq}, for any $0<r'<r_0$,
equation \eqref{eq:approximate} admits an almost periodic sequence $N(\omega z)$ with $N\in C^\omega_{r'}(\T^\infty,\R)$ provided
$l=\frac{2\bar c}{\langle Q\rangle}$. Meanwhile,
as we could notice from $O(1)$, there holds
\begin{equation}\label{eq:averaged}
\mathcal P_\epsilon v_\epsilon^a=\langle Q\rangle (u_0)_t-\bar a\partial_{zz}u_0+o(1)+f_\epsilon+R_0
(\frac{\omega z}{\epsilon}-\frac{\omega lt}{\epsilon^2})(u_0)_t+R_1(\frac{\omega z}{\epsilon}-\frac{\omega lt}{\epsilon^2})\partial_{zz}u_0,
\end{equation}
where $\bar a=\langle \tilde A_1(1+\tilde DN)\rangle+\langle (2\bar c-lQ) N\rangle$, and
$R_0,R_1\in C^\omega_{r'}(\T^\infty,\R)$ with $\langle R_0\rangle=\langle R_1\rangle=0.$
By \eqref{eq:approximate}, one has
$$
\begin{aligned}
\bar a&=\langle \tilde A_1(1+\tilde DN)\rangle+\langle (2\bar c-lQ)N\rangle\\
&=\langle\tilde A_1(1+\tilde D N)\rangle-\langle \tilde D^*\bigr(\tilde A_1 \tilde D N\bigr)N\rangle+\bar c(\tilde D^*+\tilde D) NN+\tilde D^*\tilde A_1  N\rangle\\
&=\langle\tilde A_1(1+\tilde D N)^2\rangle\geq\inf \tilde A_1\langle (1+\tilde D N)^2\rangle>\inf \tilde A_1>0.
\end{aligned}
$$

Now we turn to analyze \eqref{eq:averaged}. We claim the component of \eqref{eq:averaged} $R_0
(\frac{\omega z}{\epsilon}-\frac{\omega lt}{\epsilon^2})(u_0)_t+R_1(\frac{\omega z}{\epsilon}-\frac{\omega lt}{\epsilon^2})\partial_{zz}u_0$ converges strongly to
0 in $\mathscr S'$. Notice that similar to the proof of Lemma \ref{Lemma:solve nondivergenceeq}, there exists an almost periodic solution $w_i, i=1,2$ of the equation
$$\partial_{z}w_i(z)=R_i(\omega z)  \text{ in }\R.$$
 Then we have $R_i(\frac{\omega z}{\epsilon}-\frac{\omega lt}{\epsilon^2})=\epsilon \partial_zw_i(\frac{z}{\epsilon}-\frac{lt}{\epsilon^2}).$
Therefore,
$$
\|R_0(\frac{\omega z}{\epsilon}-\frac{\omega lt}{\epsilon^2})(u_0)_t\|_{\mathscr S'}=\sup\limits_{\|v\|_{\mathscr S}\leq 1}\int_0^T\int_{\R}|R_0(\frac{\omega z}{\epsilon}-\frac{\omega lt}{\epsilon^2})(u_0)_tv|
\leq C\epsilon,
$$
and similarly, $$\|R_1(\frac{\omega z}{\epsilon}-\frac{\omega lt}{\epsilon^2})\partial_{zz}u_0\|_{\mathscr S'}\leq C\epsilon,$$
where $C$ depends only on $\|w_0(\frac{\omega z}{\epsilon}-\frac{\omega lt}{\epsilon^2})(u_0)_t\|_{L^2([0,T]\times\R)},\ \|w_0(\frac{\omega z}{\epsilon}-\frac{\omega lt}{\epsilon^2})\partial_{z}(u_0)_t\|_{L^2([0,T]\times\R)},$ $\|w_1(\frac{\omega z}{\epsilon}-\frac{\omega lt}{\epsilon^2})\partial_{zz}u_0\|_{L^2([0,T]\times\R)},$ $ \|w_1(\frac{\omega z}{\epsilon}-\frac{\omega lt}{\epsilon^2})\partial_{zzz}u_0\|_{L^2([0,T]\times\R)}$.
Thus we have $o(1)+R_0
(\frac{\omega z}{\epsilon}-\frac{\omega lt}{\epsilon^2})(u_0)_t+R_1(\frac{\omega z}{\epsilon}-\frac{\omega lt}{\epsilon^2})\partial_{zz}u_0\rightarrow 0$ strongly in $\mathscr S'$.

Finally, we obtain
$$
\left\{
\begin{aligned}
&\mathcal P_\epsilon(v_\epsilon-v_\epsilon^a)=o(1)+R_0
(z')(u_0)_t+R_1(z')\partial_{zz}u_0 & &  \text{ in } [0,T]\times\R,\\
 &v_\epsilon-v_\epsilon^a|_{t=0}=-\epsilon N(z')\partial_zu_0 & & \text{ in }\R,
\end{aligned}
\right.
$$
where $o(1)+R_0
(\frac{\omega z}{\epsilon}-\frac{\omega lt}{\epsilon^2})(u_0)_t+R_1(\frac{\omega z}{\epsilon}-\frac{\omega lt}{\epsilon^2})\partial_{zz}u_0\rightarrow 0$ in $\mathscr S'$. It follows from Proposition \ref{proposition:energy estimate} that
$\|v_\epsilon-v_\epsilon^a\|_{\mathscr S}\rightarrow 0$, and then $v_\epsilon\to u_0$ weakly in $\mathscr S$ since $\langle N\rangle=0$. From the form of $v_\epsilon^a$, $v_\epsilon\to u_0$ strongly in $L^2([0,T],L^2(\R))$.
The lemma is proved.
\end{proof}

However, the initial value $v(0,z)=\frac{\varphi(z)}{P(\frac{\omega z}{\epsilon})}$ in \eqref{eq:transpara} is oscillatory decaying, unlike \eqref{boudwithoutosc}. For this kind of initial value, we have the following:

\begin{Lemma}\label{Lemma:homogenizationoscillation}
Suppose that the function $v_\epsilon(t,z)$ satisfies the equation
\begin{equation}\label{eq:oscillationinitial}
\left\{
\begin{aligned}
\mathcal P_\epsilon v_\epsilon&=0 & & \text{ in }[0,T]\times\R,\\
v_\epsilon(0,z)&=\varphi(z)\Phi(\frac{\omega z}{\epsilon}) & & \text{ in }\R,
\end{aligned}
\right.
\end{equation}
where $\varphi\in L^2(\R),\Phi\in C(\T^\infty,\R)$ and $\langle Q\Phi\rangle=0$. Then $v_\epsilon\rightarrow 0$ weakly in $\mathscr S$
as $\epsilon\rightarrow 0.$
\end{Lemma}

\begin{proof}
Since $\mathscr S$ is a reflexive space, by Proposition \ref{proposition:energy estimate}, $v_\epsilon$ is bounded, and thus weakly compact. We may assume that after passing to a subsequence , $v_\epsilon\rightarrow v$ weakly in
$\mathscr S.$ The result is proved if we show that $v=0$.

 For any $w\in C_0^\infty([0,T]\times\R),$
consider the dual averaging problem
\begin{equation}\label{eq:dual averaging problem}
\left\{
\begin{aligned}
&\mathcal P_\epsilon^* w_\epsilon=\hat{\mathcal P}^*w & & \text{ in }[0,T]\times\R, \\
& w_\epsilon(T,z)=w(T,z) & & \text{ in }\epsilon\Z,\\
\end{aligned}
\right.
\end{equation}
where $\mathcal P_\epsilon^*,\hat{\mathcal P^*}$ denotes the adjoint operator of $\mathcal P_\epsilon,\hat{\mathcal P}$ in $\mathscr
S$ respectively, and $\hat{\mathcal P}$ is defined in \eqref{eq:averagedeq}. Then as we could check from the argument of Proposition \ref{proposition:energy estimate}, \eqref{eq:dual averaging problem} admits a unique solution $w_\epsilon$.
Now we prove the result in the following steps:

Step 1: First, we show that $ w_\epsilon\rightarrow w_0$ weakly in $\mathscr S$, where $ w_0$ is the solution of the averaged
equation
\begin{equation}\label{eq:averaged parabolic}
\left\{
\begin{aligned}
&-\langle Q\rangle (w_0)_t-\bar a_*\partial_{zz}w_0=\hat{\mathcal P}^*w= -\langle Q\rangle w_t-\bar a\partial_{zz}w& &\text{ in }[0,T]\times\R,\\
& w_0(T,z)=w(T,z) & & \text{ in }\R,
\end{aligned}
\right.
\end{equation}
where $\bar a$ is defined in Proposition \ref{prop:homononoscillation}, and  $\bar a_*=\const>0$ to be specified.

We look for an approximate solution of \eqref{eq:dual averaging problem} in the form
$$w_\epsilon^a=w_0+\epsilon N_*(\frac{\omega z}{\epsilon}-\frac{\omega lt}{\epsilon^2})\partial_z w_0,$$
where $l\in\R$ and $N_*\in C(\T^\infty,\R)$ to be specified. Applying the argument in the proof of Proposition \ref{prop:homononoscillation}, $N_*$ satisfies the following equation
\begin{equation}\label{eq:dualapproximate}
 D^*\bigr(\tilde A_1(\omega z) DN_*(\omega z)\bigr)-\bar c(D^*+D)N_*(\omega z)+D^*\tilde A_1(\omega z)-\bigr(2\bar c-lQ(\omega z)\bigr)=0,
\end{equation}
 and $
\mathcal P_\epsilon^* w_\epsilon^a=-\langle Q\rangle (w_0)_t-\bar a\partial_{zz}w_0+o(1)+\tilde R_0
(\frac{\omega z}{\epsilon}-\frac{\omega lt}{\epsilon^2})(w_0)_t+\tilde R_1(\frac{\omega z}{\epsilon}-\frac{\omega lt}{\epsilon^2})\partial_{zz}w_0,
$ where  $\bar a_*=\langle  \tilde A_1(1+\tilde DN_*)\rangle-\langle (2\bar c-lQ)N_*\rangle$, and $\tilde R_0, \tilde R_1\in C(\T^\infty,\R)$ with $\langle \tilde R_0\rangle=\langle \tilde R_1\rangle=0$. Moreover, $\tilde R_0(\frac{\omega z}{\epsilon}-\frac{\omega lt}{\epsilon^2})(w_0)_t+\tilde R_1(\frac{\omega z}{\epsilon}-\frac{\omega lt}{\epsilon^2})\partial_{zz}w_0\to 0$ strongly in $\mathscr S'$. Hence we have
\begin{equation*}
\left\{
\begin{aligned}
&\mathcal P_\epsilon^*( w_\epsilon- w_\epsilon^a)=o(1)+\tilde R_0
(\frac{\omega z}{\epsilon}-\frac{\omega lt}{\epsilon^2})(w_0)_t+\tilde R_1(\frac{\omega z}{\epsilon}-\frac{\omega lt}{\epsilon^2})\partial_{zz} w_0 & & \text{ in } [0,T]\times\R, \\
&( w_\epsilon- w_\epsilon^a)(T,z)=-\epsilon N_*(\frac{\omega z}{\epsilon}-\frac{\omega lt}{\epsilon^2})\partial_z w_0(T,z) & & \text{ in }\R.
\end{aligned}
\right.
\end{equation*}
Since $o(1)+\tilde R_0
(\frac{\omega z}{\epsilon}-\frac{\omega lt}{\epsilon^2})(w_0)_t+\tilde R_1(\frac{\omega z}{\epsilon}-\frac{\omega lt}{\epsilon^2})\partial_{zz} w_0\rightarrow 0$ strongly in $\mathscr S'$, by the same argument in the proof of Proposition \ref{proposition:energy estimate}, one has $\| w^a_\epsilon-
w_\epsilon\|_{\mathscr S}\rightarrow 0$, and so $w^a_\epsilon\rightarrow w_0$ weakly in $\mathscr S$ and strongly in $L^2([0,T];L^2(\R))$.

Step 2: Now we claim that $w_0=w$.

From \eqref{eq:averaged parabolic}, it
is sufficient to prove the identity $\bar a_*=\bar a$. Multiplying both sides of \eqref{eq:approximate} by $N_*$ and \eqref{eq:dualapproximate} by $N$ respectively,
we arrive at the identity
$$
\begin{aligned}
-\langle \tilde A_1 \tilde D N\tilde D N_*\rangle+\bar c\langle (\tilde D+\tilde D^*) N N_*\rangle
&=\langle \tilde D^*(\tilde A_1 \tilde D N) N_*\rangle+\bar c\langle (\tilde D+\tilde D^*) N  N_*\rangle\\
&=-\langle \tilde D^*\tilde A_1 N_*\rangle-\langle (2\bar c-lQ) N_*\rangle\\
&=\langle \tilde A_1 \tilde D N_*\rangle-\langle (2\bar c-lQ)  N_*\rangle,
\end{aligned}
$$
and
$$
\begin{aligned}
-\langle \tilde A_1\tilde DN\tilde DN_*\rangle+\bar c\langle (\tilde D+\tilde D^*) N  N_*\rangle&=\langle \tilde D^*(\tilde A_1\tilde DN_*) N\rangle-\bar c\langle (\tilde D+\tilde D^*) N_* N\rangle\\
&=-\langle \tilde D^*\tilde A_1N\rangle+\langle (2\bar c-lQ) N\rangle\\
&=\langle \tilde A_1\tilde DN\rangle+\langle (2\bar c-lQ) N\rangle.
\end{aligned}
$$
Hence $\bar a=\bar a_*.$ Thus from \eqref{eq:averaged parabolic}, one has
$$
\left\{
\begin{aligned}
\hat{\mathcal P}^*(w_0-w)&=0 & &\text{ in }[0,T]\times\R,\\
(w_0-w)(T,z)&=0 & &\text{ in }\R,
\end{aligned}
\right.
$$ and this implies $ w_0=w$ by the uniqueness.

Step 3: Finish the proof by proving $v=0$.

 Multiplying $w_\epsilon$ and integrating from $[0,T]\times\R$ from both sides of $\mathcal P_\epsilon v_\epsilon=0$, one has
\begin{equation}\label{eq:parabolicand dual}
\begin{aligned}
0&=\int_0^T\int_{\R}\mathcal P_\epsilon v_\epsilon w_\epsilon dtdz=\int_0^T\int_{\R}v_\epsilon\mathcal P_\epsilon^*w_\epsilon dtdz+\\
&\qquad\qquad\qquad\int_{\R}\bigr[Q(\frac{\omega z}{\epsilon}-\frac{\omega lT}{\epsilon^2})(v_\epsilon w_\epsilon)(T,z)-Q(\frac{\omega z}{\epsilon})(v_
\epsilon w_\epsilon)(0,z)\bigr]dz.
\end{aligned}
\end{equation}
By Proposition \ref{proposition:energy estimate}, $v_\epsilon$ is continuous in $t\in [0,T]$, and then
 one conclude from Proposition \ref{prop:homononoscillation} that $v_\epsilon(T,z)\rightarrow v(T,z)$ strongly in $L^2(\R)$. We also have $w_\epsilon(0,z)
\rightarrow w(0,z)$ by step 1 and step 2, and meanwhile, by assumption, $Q(\frac{\omega z}{\epsilon}) v_\epsilon(0,z)=Q(\frac{\omega z}{\epsilon})\Phi(\frac{\omega z}{\epsilon})\varphi(z)\rightarrow 0$ weakly in $L^2(\R)$ as
$\epsilon\rightarrow 0$. Passing along to a subsequence in equation \eqref{eq:parabolicand dual} as $\epsilon\to 0$, one has for any function $w\in C_0^\infty([0,T]\times\R)$, there holds
$$\int_0^T\int_{\R} v(t,z)\hat{\mathcal P}^*w(t,z)dtdz+\langle Q\rangle\int_{\R} v(T,z)w(T,z)dz=0,$$
Hence $v$ is a weak solution of the equation $\hat{\mathcal P}v=0$ in $[0,T]\times\R$ and $v(0,z)=0$ in $\R$, and it is of course 0. Thus the lemma
is proved.
\end{proof}

\begin{proof}[Proof of Theorem \ref{APPMain3}]
As we have seen from Proposition
 \ref{Main1ell}, there exist $P, P_*\in C(\T^\infty,\R)$ such that $p(n)=P(n\omega), p_*(n)=P_*(n\omega)$ satisfy
\begin{equation}
\begin{aligned}
&D^*\bigr(A_1(n\omega) Dp\bigr)+A_2(n\omega)D^*p+W(n\omega)p=E_0p,\\
&D^*\bigr(A_1(n\omega)Dp_*\bigr)-D\bigr(A_2(n\omega)p_*\bigr)+W(n\omega)p_*=E_0p_*,
\end{aligned}
\end{equation}
where $\sum\limits_{k\in\mathbb Z_*^\infty}|\hat {P}(k)| \me^{r'|k|_1}<\infty.$ Hence we have $P, P_*\in C^\omega_{r'}(\T^\infty,\R)$ for any $0<r'<r.$

 Note that for any $k\in\Z$, $p_k(n)=P(n\omega+k\omega)$ solves
$$D^*\bigr(A_1(n\omega+k\omega) Dp_k\bigr)+A_2(n\omega+k\omega)D^*p_k+W(n\omega+k\omega)p_k=E_0p_k,$$
and $q_k(n)=P_*(n\omega+k\omega)$ solves
$$D^*\bigr(A_1(n\omega+k\omega)Dq_k\bigr)- D\bigr(A_2(n\omega+k\omega)q_k\bigr)+W(n\omega+k\omega)q_k=E_0q_k.$$
Hence by continuity of $P$ and $P_*$,
 the crucial assumption \eqref{averagingtorus} is satisfied.

Now we first prove that the equation \eqref{eq:averagingpara} has a unique solution $u_\epsilon\in L^2([0,T], L^2(\epsilon\Z)).$
Changing of variable $v_\epsilon(t,x)=\me^{-\frac{E_0t}{\epsilon^2}}\frac{u_\epsilon(t,x)}{P(\frac{\omega x}{\epsilon})}, x\in\epsilon\Z$, equation \eqref{eq:averagingpara}
becomes
\begin{equation}\label{eq:originalredu}
\left\{
\begin{aligned}
Q(\frac{\omega x}{\epsilon})(v_\epsilon)_t&=D_\epsilon^*(\tilde A_1\bigr(\frac{\omega x}{\epsilon})D_\epsilon v_\epsilon\bigr)-\frac{\bar c}{\epsilon}(D_\epsilon^*+D_\epsilon)v_\epsilon & &\text{ in }[0,T]\times \epsilon\Z, \\
v_\epsilon(0,x)&=\frac{\varphi(x)}{P(\frac{\omega x}{\epsilon})} & & \text{ in }\epsilon\Z.
\end{aligned}
\right.
\end{equation}
 Since $\varphi\in C^1(\R)$ and $|\varphi(z)|\leq \frac{C}{|z|^k}$ for some $k\geq 2,$ it follows from Theorem \ref{existenceinl^2}
that $$v_\epsilon(t,x)=\me^{t\frac{1}{Q(x)}\mathcal L_\epsilon} \frac{\varphi(x)}{P(\frac{\omega x}{\epsilon})}\in L^2([0,T], L^2\bigr(\epsilon\Z)\bigr)\cap C^1([0,T], L^2(\epsilon\Z)),$$ and then $u_\epsilon\in L^2\bigr([0,T], L^2(\epsilon\Z)\bigr)\cap C^1([0,T], L^2(\epsilon\Z)).$
Similar to the proof of Proposition \ref{proposition:energy estimate}, one can deduce that
\begin{equation*}
\sup\limits_{0\leq t\leq T}\|v_\epsilon(t,\cdot)\|^2_{L^2(\epsilon\Z)}\leq C\bigr\|\varphi\|^2_{L^2(\epsilon\Z)},
\end{equation*}
where the constant $C$ depends only on $\inf P\inf P_*$, and this implies the uniqueness of $u_\epsilon$.

Now it follows from Proposition \ref{proposition:energy estimate}, there exists $\tilde v_\epsilon(t,z)\in \mathscr S$ which solves
\begin{equation}\label{eq:torusequ}
\left\{
\begin{aligned}
\mathcal P_\epsilon \tilde v_\epsilon&=0 & & \text{ in }[0,T]\times\R,\\
\tilde v_\epsilon(0,z)&=\frac{\varphi(z)}{P(\frac{\omega z}{\epsilon})} & & \text{ in }\R,
\end{aligned}
\right.
\end{equation}
where $z'=\omega\bigr(\frac{z}{\epsilon}-\frac{lt}{\epsilon^2}\bigr)$.
Restricting $\tilde v_\epsilon(t,z+\frac{lt}{\epsilon})$ in $\epsilon\Z$, we can argue similarly as $v_\epsilon$ that $\tilde v_\epsilon(t,\cdot+\frac{lt}{\epsilon})\in L^2\bigr([0,T],L^2(\epsilon\Z)\bigr)\cap C^1([0,T]\times\R)$, and it solves \eqref{eq:originalredu}. By the uniqueness of the solution of \eqref{eq:originalredu}, $\tilde v_\epsilon(t,x+\frac{lt}{\epsilon})=v_\epsilon(t,x).$

Let $\tilde w_\epsilon$ be the solution of the equation
\begin{equation}\label{eq:l^2initial}
\left\{
\begin{aligned}
\mathcal P_\epsilon \tilde w_\epsilon&=0 & & \text{ in } [0,T]\times\R,\\
\tilde w_\epsilon(0,z)&=\varphi/\langle Q\rangle & & \text{ in }\R.
\end{aligned}
\right.
\end{equation}
Since the initial value $\varphi/\langle Q\rangle\in L^2(\R)$, it follows from Proposition \ref{prop:homononoscillation}, $\tilde w_\epsilon\to u_0$ weakly in $\mathscr S$ as $\epsilon\to 0$, where $u_0$ solves the averaged equation
\begin{equation}\label{eq:averagedeq1}
\left\{
\begin{aligned}
\hat{\mathcal P}u_0:&=\langle Q\rangle(u_0)_t-\bar a \partial_{zz}u_0=0& & \text{ in }[0,T]\times\R,\\
u_0(0,z)&=\varphi/\langle Q\rangle & & \text{ in }\R.
\end{aligned}
\right.
\end{equation}
Subtracting \eqref{eq:l^2initial} from \eqref{eq:torusequ},
 one has the difference
$\tilde v_\epsilon-\tilde w_\epsilon$ satisfying the relation:
$$
\left\{
\begin{aligned}
\mathcal P_\epsilon(\tilde v_\epsilon-\tilde w_\epsilon)&=0 & & \text{ in }[0,T]\times\R,\\
(\tilde v_\epsilon-\tilde w_\epsilon)(0,z)&=\varphi(z)\biggr(\frac{1}{P(\frac{\omega z}{\epsilon})}-\frac{1}{\langle Q\rangle}\biggr) & & \text{ in }\R.
\end{aligned}
\right.
$$
 Due to \eqref{eq:normalized}, $\langle\bigr(\frac{1}{P}-\frac{1}{\langle Q\rangle}\bigr)Q\rangle=0$, and then it follows from Lemma \ref{Lemma:homogenizationoscillation}, $\tilde v_\epsilon-\tilde w_\epsilon\to 0$ weakly as $\epsilon\to 0$ which yields that
 $\tilde v_\epsilon\rightarrow u_0$ weakly in $\mathscr S$.
By the Sobolev embedding theorem, for any $t\in (0,T]$, as $\epsilon\to 0$, $\tilde v_\epsilon\to u_0$ uniformly in any compact set in $\R$. Note that by the continuity in $z\in\R$ of $\tilde v_\epsilon(t,z)$, as $\epsilon\to 0$, $v_\epsilon(t,\lfloor \frac{z}{\epsilon}-\frac{lt}{\epsilon^2}\rfloor\epsilon)-\tilde v_\epsilon(t,z)\to 0$ in $\R$.  Hence by diagonal extracting a subsequence, $$ v_\epsilon(t,\lfloor \frac{z}{\epsilon}-\frac{lt}{\epsilon^2}\rfloor\epsilon)=\me^{-\frac{E_0t}{\epsilon^2}}\frac{u_\epsilon(t,\lfloor\frac{z}{\epsilon}-\frac{lt}{\epsilon^2}\rfloor\epsilon)}{P(\omega\lfloor\frac{ z}{\epsilon}-\frac{lt}{\epsilon^2}\rfloor)}=\me^{-\frac{E_0t}{\epsilon^2}}\frac{u_\epsilon(t,\lfloor\frac{z}{\epsilon}-\frac{lt}{\epsilon^2}\rfloor\epsilon)}{p(\lfloor\frac{ z}{\epsilon}-\frac{lt}{\epsilon^2}\rfloor)}\to u_0(t,z) \text{ in } \R,$$
where the convergence is uniform in any compact set in $\R$ for any $t\in (0,T]$.
Then the result follows directly.
\end{proof}

\appendix

\section{Proof of Lemma \ref{basic}}
We need the following quantitative implicit function theorem:
\begin{Theorem}[\cite{Deimling1985,Berti2006}]\label{thm:ift}
Let $X,Y,Z$ be Banach spaces, $U\subset X$ and $V\subset Y$ neighborhoods of $x_0$ and $y_0$ respectively. Fix $s,\delta>0$ and define $B_s(x_0)=\{x\in
X|\|x-x_0\|_X\leq s\}, B_\delta(y_0)=\{y\in Y|\|y-y_0\|_{Y}\leq \delta\}$. Let $\Psi\in C^1(U\times V,Z)$ and $B_s(x_0)\times B_\delta(y_0)\subset U\times
V$. Suppose also that $\Psi(x_0,y_0)=0$, and that $D_y\Psi(x_0, y_0)\in L(Y,Z)$ is invertible. If
\begin{equation}\label{eq:IFT assumption1}
\sup\limits_{\overline{B_s(x_0)}}\|\Psi(x,y_0)\|_Z\leq \frac{\delta}{2\|(D_y\Psi(x_0,y_0))^{-1}\|},
\end{equation}
\begin{equation}\label{eq:IFT assumption2}
\sup\limits_{\overline B_s(x_0)\times\overline B_\delta(y_0)}\|\id_Y-(D_y\Psi(x_0,y_0))^{-1}D_y\Psi(x,y)\|_{L(Y,Y)}\leq \frac{1}{2},
\end{equation}
then there exists a unique $y\in C^1(B_s(x_0),\overline{B_\delta(y_0)})$ such that $\Psi(x,y(x))=0.$
\end{Theorem}
\begin{proof}[Proof of Lemma \ref{basic}]
Now we construct the nonlinear functional
$$\Psi:\mathcal I\times \mathscr B^{nre}_{\mathcal I}(\eta)\times \mathscr B_{\mathcal I}\rightarrow \mathscr B^{nre}_{\mathcal I}(\eta)$$
by
$$\Psi(\xi,Y,f)=\mathbb P_{nre}\ln(\me^{-A(\xi)^{-1}Y(\xi,\theta+\omega)A(\xi)}\me^{f(\xi,\theta)}\me^{Y(\xi,\theta)}).$$
Let $X=\mathcal I\times \mathscr B_{\mathcal I}, Y=Z=\mathscr B^{nre}_{\mathcal I}(\eta)$. It is clear that $\Psi\in C^1(X\times Y,Z)$.
 Now we check the assumptions of Theorem \ref{thm:ift}. It is clear that for any $\xi\in\mathcal I,$
$$\Psi(\xi,0,0)=0, \|\Psi(\xi,0,f)\|\leq \|f\|_{\mathcal B,\mathcal I}\leq \epsilon.$$

 By the similar calculation in \cite[Lemma 3.1]{Cai2019},
$$
\begin{aligned}
D_Y\Psi(\xi,0,0)(Y')
=-A(\xi)^{-1}Y'(\xi,\theta+\omega)A(\xi)+Y'(\xi,\theta).
\end{aligned}
$$
Thus $$\|D_Y\Psi(\xi,0,0)(Y')\|\geq |A(\xi)^{-1}Y'(\xi,\theta+\omega)A(\xi)-Y'(\xi,\theta)|_{\mathcal B}\geq \eta|Y'|_{\mathcal B}.$$
So we have $\|(D_Y\Psi(\xi,0,0))^{-1}\|\leq \eta^{-1}.$

Fix $\xi_0\in\mathcal I,$
since $A\in C(\mathcal I,\GL)$, there exists $\rho=\rho(\xi_0)>0$ with $|\xi_0-\xi|\leq \rho$ such that $\|A(\xi_0)-A(\xi)\|+\|A(\xi_0)^{-1}-A(\xi)^{-1}\|\leq \epsilon$. Setting $s=\min\{\rho,\epsilon\},\delta=\epsilon^{\frac{1}{2}}$ and $\eta\geq \frac{60(\|A\|^3+1)}{\inf|\det A|}\epsilon^{\frac{1}{2}}$, there holds the estimate
$$
\quad 2\sup\limits_{\overline {B_s((\xi_0,0))}}\|\Psi(\xi,0,f)\|\times\|(D_Y\Psi(\xi_0,0,0))^{-1}\|\leq 2\epsilon\frac{\inf|\det A|}{60(\|A\|^3+1)}\epsilon^{-\frac{1}{2}}\leq \delta,
$$
and then \eqref{eq:IFT assumption1} is satisfied.

On the other hand, since $Y,f$ are small $\mathrm{gl}(2,\R)$ matrices, by Baker-Campbell-Hausdorff formula, one has
$$
\begin{aligned}
&\quad D_Y\Psi(\xi,Y,f)(Y')-D_Y\Psi(\xi_0,0,0)(Y')\\
&=\mathbb P_{nre}((A(\xi_0)^{-1}-A(\xi)^{-1})Y'(\xi,\theta+\omega)A(\xi_0))-A(\xi)^{-1}Y'(\xi,\theta+\omega)(A(\xi)-A(\xi_0))\\
&\quad\quad-\mathbb P_{nre}(O(A(\xi)^{-1}Y(\xi,\theta+\omega)A(\xi))A(\xi)^{-1}Y'(\xi,\theta+\omega)A(\xi)+\frac{1}{2}[Y''',M+H]+\cdots)\\
&\quad\quad +\mathbb P_{nre}(O(Y(\xi,\theta))Y'(\xi,\theta)+\frac{1}{2}[D+G,-Y'']+\cdots).
\end{aligned}
$$
where $Y''(\xi,\theta)=Y'(\xi,\theta)+O(Y(\xi,\theta))Y'(\xi,\theta)$, $Y'''(\xi,\theta)=-A(\xi)^{-1}Y'(\xi,\theta+\omega)A(\xi)-O(A(\xi)^{-1}Y(\xi,\theta+\omega)A(\xi))A(\xi)^{-1}Y'(\xi,\theta+\omega)A(\xi)$, $D(\xi,\theta)=-A(\xi)^{-1}(Y(\xi,\theta+\omega)+Y'(\xi,\theta+\omega))A(\xi)+f(\xi,\theta)+Y(\xi,\theta),$ $M(\xi,\theta)=-A(\xi)^{-1}Y(\xi,\theta+\omega)A(\xi)+f(\xi,\theta)
+Y(\xi,\theta)$ and $G, H$ are higher order terms. 

Therefore, since $\mathcal B$ admits property $(H)$, for $\|Y\|_{\mathcal B,\mathcal I}\leq \delta, \|f\|_{\mathcal B,\mathcal I}\leq s,$
$$
\begin{aligned}
&\quad\|D_Y\Psi(\xi,Y,f)(Y')-D_Y\Psi(\xi_0,0,0)(Y')\|\\
&\leq \frac{6|Y'|_{\mathcal B}(\|A\|^3+1)}{|\det A(\xi_0)|}\bigr(2\|A(\xi)-A(\xi_0)\|+2\|A(\xi_0)^{-1}-A(\xi)^{-1}\|+|Y|_{\mathcal B}+|f|_{\mathcal B}\bigr)\\
&\leq \frac{6(\|A\|^3+1)|Y'|_{\mathcal B}}{\inf|\det A(\xi)|}\bigr(3\epsilon+\epsilon^{\frac{1}{2}}\bigr).
\end{aligned}
$$
which implies
$$\sup\limits_{\overline{B_s((\xi_0,0))}\times\overline{B_\delta(0)}}\|D_Y\Psi(\xi_0,0,0)-D_Y\Psi(\xi,Y,f)\|\leq \frac{6(\|A\|^3+1)(\epsilon^{\frac{1}
{2}}+3\epsilon)}{\inf|\det A(\xi)|}.$$
Thus we have
$$
\begin{aligned}
&\sup\limits_{\overline{B_s((\xi_0,0))}\times\overline{B_\delta(0)}}\|\id_{\mathcal B_{\xi_0}^{nre}(\eta)}-(D_Y\Psi(\xi_0,0,0))^{-1}\times
D_Y\Psi(\xi,Y,f)\|\\
&\quad \leq \sup\limits_{\overline{B_s((\xi_0,0))}\times\overline{B_\delta(0)}}\|D_Y\Psi(\xi_0,0,0)-D_Y\Psi(\xi,Y,f)\|\|(D_Y\Psi(\xi_0,0,0))^{-1}\|\\
&\quad \leq \frac{6(\|A\|^3+1)(\epsilon^{\frac{1}
{2}}+3\epsilon)}{\inf|\det A(\xi)|}\frac{\inf|\det A(\xi)|}{60(\|A\|^3+1)}\epsilon^{-\frac{1}{2}}\leq \frac{1}{2}
\end{aligned}
$$
which implies that \eqref{eq:IFT assumption2} is satisfied. By Theorem \ref{thm:ift}, for $\|f\|_{\mathcal B,\mathcal I}\leq \epsilon$ and $\eta\geq \frac{60(1+\|A\|^3)\epsilon^{\frac{1}{2}}}{\inf|\det A|}$, there
exist $Y=Y(\xi,f),\ F^{re}=F^{re}(\xi,f)\in C(\mathcal I\times\mathscr B_{\mathcal I},\mathscr B^{re}_{\mathcal I})$ with the estimates $\|Y(\xi,f)\|_{\mathcal B,\mathcal I}\leq \epsilon^{\frac{1}{2}}$ such that $$\Psi(A(\xi),Y(\xi,f),f)=0,$$ i.e.,
$$\me^{-A(\xi)^{-1}Y(\xi,f)(\theta+\omega)A(\xi)}\me^{f(\theta)}\me^{Y(\xi,f)(\theta)}=\me^{F^{re}(\xi,f)(\theta)}.$$
It is equivalent to
$$\me^{-Y(\xi,\theta+\omega)}A(\xi)\me^{F(\xi,\theta)}\me^{Y(\xi,\theta)}=A(\xi)\me^{F^{re}(\xi,\theta)}.$$
It is clear to check that $\|F^{re}\|_{\mathcal B,\mathcal I}\leq 2\epsilon$.
The proof is complete.
\end{proof}


\begin{thebibliography}{99}


\bibitem{Avila2015}
A. Avila,
\newblock Global theory of one-frequency Schr\"odinger operators,
\newblock Acta Mathematica, 215(1) (2015) 1-54.

\bibitem{A3}
\newblock A. Avila,
\newblock Almost reducibility and absolute continuity I, arXiv preprint (2010),
\newblock arXiv:1006.0704 [math.DS].

\bibitem{AJ05}
 A. Avila, S. Jitomirskaya,
 \newblock The Ten Martini Problem,
 \newblock Ann. of Math. {170}, (2009) 303--342.

\bibitem{AJ}
 A. Avila, S. Jitomirskaya,
 \newblock H\"{o}lder continuity of absolutely continuous spectral measures for one-frequency Schr\"{o}dinger operators,
 \newblock Commun. Math. Phys.,  301, (2011) 563--581.

 \bibitem{AK}
 A. Avila, R. Krikorian,
 \newblock Reducibility or nonuniform hyperbolicity for quasiperiodic Schr\"odinger cocycles,
  \newblock Ann. of Math.  No.  911--940 (2006).


\bibitem{AK2}
A. Avila, R. Krikorian,
\newblock Almost reducibility of pseudo-rotations of the disk,
\newblock In preparation

\bibitem{AK3}
 A. Avila, R. Krikorian,
\newblock Some remarks on local and semi-local results for Schr\"odinger cocycles,
\newblock In preparation


\bibitem{Avila2017}
A. Avila, J. You, and Q. Zhou,
\newblock Sharp phase transitions for the almost Mathieu operator,
 \newblock Duke Mathematical Journal 166.14 (2017) 2697-2718.

 \bibitem{Avila2018}
A. Avila, S. Jitomirskaya, and Q. Zhou,
\newblock Second phase transition line,
\newblock Mathematische Annalen 370.1 (2018) 271-285.


\bibitem{Avellaneda1989}
M. Avellaneda, F. Lin,
\newblock Compactness methods in the theory of homogenization II: Equations in non-divergence form,
\newblock Communications on Pure and Applied Mathematics,  42(2) (1989) 139-172.

\bibitem{Allaire2007}
G. Allaire,  R. Orive,
\newblock Homogenization of periodic non self-adjoint problems with large drift and potential,
 \newblock ESAIM Control Optimisation and Calculus of Variations, 13(4) (2007) 1547--1567.

\bibitem{AGU}
Y. Ashidaa, Z. Gong and M. Ueda,
\newblock Non-Hermitian Physics,
\newblock  Advances in physics, 69, 3 (2020).


\bibitem{BBK}
 E.J. Bergholtz, J.C. Budich, and F.K. Kunst,
\newblock  Exceptional topology of non-Hermitian systems
\newblock {\it Rev. Mod. Phys}. 93, 015005.


\bibitem{Lions2011}
A. Bensoussan, J.L. Lions, and G. Papanicolaou,
 \newblock Asymptotic analysis for periodic structures,
 \newblock  American Mathematical Soc Vol. 374 (2011).

\bibitem{Berti2006}
M. Berti, L. Biasco,
\newblock Forced vibrations of wave equations with non-monotone nonlinearities,
\newblock Annales De Linstitut Henri Poincare Non Linear Analysis (2006).


\bibitem{Bourgain2005}
J. Bourgain,
\newblock On invariant tori of full dimension for 1d periodic NLS,
\newblock { Journal of Functional Analysis,} 229(1) (2005) 62--94.

\bibitem{Biasco2019}
L. Biasco, J.E. Massetti, and M. Procesi,
\newblock An abstract Birkhoff normal form theorem and exponential type stability of the 1d NLS,
\newblock {Communications in Mathematical Physics,} 375(2) (2019).


\bibitem{Cai2019}
 A. Cai, C. Chavaudret, J. You, and Q. Zhou,
\newblock Sharp H{\"o}lder continuity of the Lyapunov exponent of finitely differentiable quasi-periodic cocycles,
\newblock {Mathematische Zeitschrift,} 291(3) (2019) 931-958.

\bibitem{ShenCao}
F. Cao, W. Shen,
\newblock Spreading speeds and transition fronts of lattice KPP equations in time heterogeneous media,
\newblock{Discrete and Continuous Dynamical Systems,} 37(9) (2016) 4697-4727.

\bibitem{CD} R. Courant, D. Hilbert,
\newblock   Methods of mathematical physics: partial differential equations,
\newblock  John Wiley and Sons (2008).

\bibitem{CW}
C.Q. Cheng, L. Wang,
\newblock Destruction Of Lagrangian Torus For Positive Definite Hamiltonian Systems,
\newblock Geom. Funct. Anal. 23, (2013) 848--866.


\bibitem{Dinaburg1975}
E.I. Dinaburg, Y.G. Sinai,
\newblock The one-dimensional Schr\"odinger equation with a quasiperiodic potential,
\newblock Functional Analysis and Its Applications 9.4 (1975) 279-289.


\bibitem{Patrizia2005}
P. Donato, A. Piatnitski,
\newblock Averaging of nonstationary parabolic operators with large lower order terms,
\newblock Multi Scale Problems and Asymptotic Analysis, GAKUTO Internat. Ser. Math. Sci. Appl 24, (2005) 153-165.


\bibitem{Deimling1985}
K. Deimling,
\newblock Nonlinear Functional Analysis,
\newblock Springer, Berlin, (1985).

\bibitem{Eliasson1992}
L.H. Eliasson,
\newblock Floquet solutions for the 1-dimensional quasi-periodic Schr\"odinger equation,
\newblock Communications in mathematical physics, 146(3) (1992) 447-482.


\bibitem{Evans1964}
L.C. Evans,
\newblock  Partial differential equations,
\newblock Intersxcience Publishers (1964).


\bibitem{FK}
B. Fayad, R. Krikorian,
 \newblock Some questions around quasi- periodic dynamics,
  \newblock In Proceedings of the International Congress of Mathematicians-Rio deJaneiro 2018. Vol. III.,  World Sci. Publ., Hackensack, NJ (2018) pages 1909- 1932.


\bibitem{F}
G. Forni,
\newblock  Analytic destruction of invariant circles,
\newblock Ergod. Th.  Dynam. Sys. 14 (1994) 267-298.

\bibitem {G1}
J.S. Guo, F. Hamel,
\newblock Front propagation for discrete periodic monostable equations,
\newblock { Mathematische Annalen}, 335(3) (2006) 489-525.


\bibitem{Gong}
Z. Gong, Y. Ashida, K. Kawabata, K. Takasan, S. Higashikawa, and M. Ueda,
\newblock Topological Phases of Non-Hermitian Systems,
\newblock Phys. Rev. X, 8, 031079 (2018).

\bibitem {Hamel}
F. Hamel, N. Nadirashvili, and E. Russ,
\newblock  Rearrangement inequalities and applications to isoperimetric problems for eigenvalues,
Annals of mathematics and (2011) 647-755.

\bibitem{HN}
N. Hatano, D.R. Nelson, Localization,
\newblock transitions in non-hermitian quantum mechanics,
\newblock Phys. Rev. Lett. 77, 570 (1996).

\bibitem{H}
M.R. Herman.
\newblock Sur les courbes invariantes par les diffeomorphismes de l'nneau,
\newblock Ast\'erisque 103-104 (1983) 1-221.

\bibitem{H2}
M.R. Herman,
\newblock Non existence of Lagrangian graphs,
\newblock Unpublished preprint (1990).

\bibitem{Herman2004}
M.R. Herman,
\newblock $L^2$ regularity of measurable solutions of a finite-difference equation of the circle,
\newblock Ergodic Theory and Dynamical Systems, 24(5) (2004) 1277-1281.

\bibitem{houyou}
X. Hou, J. You,
 \newblock Almost reducibility and non-perturbative reducibility of quasi-periodic linear systems,
 \newblock  Invent. Math. 190 (2012) 209-260 .

\bibitem{jiang2019interplay}
H. Jiang, L. J. Lang, C. Yang., S. L. Zhu, and S. Chen,
\newblock Interplay of non-hermitian skin effects and anderson localization in nonreciprocal quasiperiodic lattices,
\newblock Phys. Rev. B 100, 054301 (2019).


\bibitem{Jaksic2000}
V. Jak\v{s}i\'c, S. Molchanov,
\newblock A note on the regularity of solutions of linear homological equations.
\newblock Applicable Analysis, 75.3-4 (2000) 371-377.

\bibitem{Jikov1994}
V.V. Jikov, S.M. Kozlov, and O.A. Oleinik,
\newblock Homogenization of differential operators and integral functionals, Springer, Berlin (1994)  MR 96h:35003b.


\bibitem{Jitomirskaya1999}
S.Y. Jitomirskaya,
\newblock Metal-insulator transition for the almost mathieu operator,
\newblock {Annals of Mathematics}, 150(3) (1999) 1159--1175.


 \bibitem{JLiu}
 S.Y. Jitormiskya, W. Liu,
 \newblock {Universal hierarchical structure of quasiperiodic eigenfuctions},
 \newblock Ann. of Math. 187 no.3 (2018).

\bibitem{jk}
S.Y. Jitomirskaya, I. Krasovsky,
\newblock{Critical almost Mathieu operator: hidden singularity, gap continuity, and the Hausdorff dimension of the spectrum},
\newblock To appear in Ann Math.

\bibitem{Khintchine1936}
A. Khintchine,
\newblock Zur metrischen Kettenbruchtheorie,
\newblock Compositio Math (1936) 276-285.

\bibitem{Kenig2012}
C.E, Kenig, F. Lin, and Z. Shen,
\newblock Convergence Rates in {$L^2$} for Elliptic Homogenization Problems,
\newblock archive for rational mechanics \& analysis, 203(3) (2012) 1009-1036.

\bibitem{Klitzing1980}
K.V. Klitzing, D. Gerhard, and P. Michael,
\newblock New method for high-accuracy determination of the fine-structure constant based on quantized Hall resistance,
\newblock Physical review letters 45.6 (1980) 494.

\bibitem{Kozlov1978}
 S.M. Kozlov,
\newblock Averaging of differential operators with almost periodic rapidly oscillating coefficients,
\newblock Doklady Akademii nauk SSSR, 236(5) (1978) 1068-1071.



\bibitem{Kozlov1979}
S.M. Kozlov,
\newblock The averaging of random operators,
\newblock  Matematicheskii Sbornik 151.2 (1979) 188-202.

\bibitem{Kozlov1984}
S.M. Kozlov,
\newblock Reducibility of quasiperiodic differential operators and averaging,
\newblock Transactions of the Moscow Mathematical Society,
(2) (1984), 99-123.

\bibitem{Kozlov1987}
S.M. Kozlov,
\newblock Averaging of difference schemes,
\newblock Mathematics of the USSR-Sbornik, 57(2) (1987) 351.


\bibitem{Liu2021}
Y. Liu, Y. Wang, X.J. Liu, Q. Zhou, and S. Chen,
\newblock Exact mobility edges, PT-symmetry breaking, and skin effect in one-dimensional non-Hermitian quasicrystals,
\newblock Physical Review B, 103(1) (2021) 014203.



\bibitem{LZC}
Y. Liu,  Q. Zhou and S. Chen,
 \newblock Localization transition, spectrum structure, and winding numbers for one-dimensional non-Hermitian quasicrystals,
  \newblock Physical Review B, 104(2), 024201 (2021).

	\bibitem {liang2014}
X. Liang, H. Matano,
\newblock 	Maximizing the spreading speed of KPP fronts in two-dimensional stratified media,
\newblock Proceedings of the London Mathematical Society, 109(5) (2014) 1137-1174.


\bibitem{Liang2018}
X. Liang, T. Zhou,
\newblock Spreading speeds of KPP-type lattice systems in heterogeneous media,
\newblock {Communications in Contemporary Mathematics} (2018) page 1850083.

\bibitem{LZ}
X. Liang, X. Zhao,
\newblock Spreading speeds and traveling waves for abstract monostable evolution systems,
\newblock { Journal of Functional Analysis}, 259 (2010) 857-903.


\bibitem{Liang2021}
X. Liang, H. Wang, Q. Zhou, and T. Zhou,
\newblock Traveling fronts for Fisher-KPP lattice equations in almost periodic media,
\newblock {arXiv preprint,} arXiv:2104.13805 (2021).


\bibitem{Martin2017}
M. Leguil, J. You,  Z. Zhao, and Q. Zhou,
\newblock Asymptotics of spectral gaps of quasi-periodic Schr\" odinger operators,
\newblock arXiv preprint arXiv:1712.04700 (2017).



 \bibitem{longhi}
 S. Longhi,
 \newblock Topological phase transition in non-hermitian quasicrystals,
 \newblock Phys. Rev. Lett. 122, 237601 (2019).

\bibitem{Mather}
J.N. Mather,
\newblock  Destruction of invariant circles,
\newblock Ergod, Th.  Dynam. Sys. 8 (1988), 199-214.

\bibitem{Minc}
 H. Minc,
\newblock Nonnegative matrices,
\newblock John Wiley and Sons, New York (1988).


\bibitem{Riccardo2021}
R. Montalto, M. Procesi,
\newblock Linear Schr\"odinger Equation with an almost periodic potential,
\newblock{SIAM Journal on Mathematical Analysis}, 53.1 (2021) 386-434.

\bibitem {Nadin2010}
G. Nadin,
\newblock The Effect of the Schwarz Rearrangement on the Periodic Principal Eigenvalue of a Nonsymmetric Operator,
\newblock SIAM Journal on Mathematical Analysis, 41(6)(2010) 2388-2406.


\bibitem{Pazy}
A. Pazy,
\newblock Semigroups of linear operators and applications to partial differential equations,
\newblock {Springer-Verlag,} (1983).

\bibitem{Papanicolaou1979}
G.C, Papanicolaou, S.R.S. Varadhan,
\newblock Boundary value problems with rapidly oscillating random coefficients,
\newblock Colloquia Math. Soc., Janos Bolyai. Vol. 27. (1979)  835-873.

\bibitem{RSS}
G.V. Rozenblum, M. A. Shubin, and M. Z. Solomyak,
\newblock  Spectral theory of differential operators. Partial differential equations VII,
\newblock  Springer, Berlin, Heidelberg (1994) 1-235.


\bibitem{Sarnak}
P. Sarnak,
\newblock { Spectral behavior of quasi periodic potentials},
\newblock Commun. Math. Phys., 84, (1982) 377-401 .

	
\bibitem{Sim2}
 B. Simon,
 \newblock  { Almost periodic Schr\"{o}dinger operators: a review},
 \newblock  Adv. in App. Math., 3, (1982) 463-490.


\bibitem{Simon2000}
B. Simon,
\newblock Schr\"odinger operators in the twenty-first century, \newblock Mathematical physics  (2000), 283-288.

\bibitem{Simon2011}
B. Simon,
\newblock Szego's theorem and its descendants: Spectral theory for {$L^2$} perturbations of orthogonal polynomials,
 \newblock Princeton University Press, Princeton and Oxford, (2011).


\bibitem{spagnolo1968}
S. Spagnolo,
\newblock Sulla convergenza di soluzioni di equazioni paraboliche ed ellittiche,
\newblock Annali della Scuola Normale Superiore di Pisa-Classe di Scienze, 22.4 (1968), 571-597.

\bibitem{youzhou}
J. You, Q. Zhou,
\newblock{Embedding of analytic quasi-periodic cocycles into analytic quasi-periodic linear systems and its applications,}
\newblock { Commun. Math. Phys.} 323 (2013), 975-1005 .

\bibitem{zehnder}
E. Zehnder,
\newblock Generalized implicit function theorems with applications to some small divisor problems,
\newblock { Comm. Pure Appl. Math.} (1975), 91-140 .


\end{thebibliography}
\end{document}